\documentclass[11pt]{amsart}
\usepackage{graphicx,verbatim,lineno,titletoc}
\usepackage{amssymb,mathrsfs}
\usepackage{amsmath}
\usepackage{adjustbox}
\usepackage{calc}
\usepackage{array}
\usepackage[font=small]{caption}
\usepackage[T1]{fontenc}
\usepackage{fourier}
\usepackage[english]{babel}
\usepackage{appendix}
\usepackage{enumerate}
\usepackage[english]{babel}
\usepackage{multirow}
\usepackage{float}
\usepackage{enumitem}
\usepackage[numbers]{natbib}
\usepackage[all]{xy}
\usepackage{tikz}
\usetikzlibrary{shapes.multipart,patterns,arrows}
\usepackage{hyperref}
\hypersetup{colorlinks,linkcolor={red},citecolor={olive},urlcolor={red}}
\usepackage[top=1.3in, left=1.5in, right=1.5in, bottom=1.5in]{geometry}
\usepackage{mathtools}
\mathtoolsset{showonlyrefs}
\usepackage{csquotes}
\usepackage{blkarray}
\usepackage[utf8]{inputenc}

\usepackage{tikz-cd}

\newcolumntype{M}[1]{>{\centering\arraybackslash}m{#1}}
\newcolumntype{N}{@{}m{0pt}@{}}

\def\.{\hskip.06cm}

\newcommand{\vertiii}[1]{{\left\vert\kern-0.25ex\left\vert\kern-0.25ex\left\vert #1 
		\right\vert\kern-0.25ex\right\vert\kern-0.25ex\right\vert}}

\def\cF{\mathcal F}

\def\T{\mathbf{T}}

\def\balpha{\boldsymbol{\alpha}}

\def\bbeta{\boldsymbol{\beta}}

\def\I{\mathbf{I}}

\def\bSigma{\boldsymbol{\Sigma}}
\def\bLambda{\boldsymbol{\Lambda}}

\def\<{\langle}
\def\>{\rangle}

\def\0{{\mathbf 0}}

\def\.{\hskip.06cm}

\def\bA{\textbf{\textit{A}}}

\def\P{{\textup{\textsf{P}}}}

\DeclareMathOperator{\Var}{\textnormal{Var}}

\def\liminf{\mathop{\rm lim\,inf}\limits}

\def\D{\mathbf{D}}

\def\bZ{\mathbf{Z}}
\def\R{\mathbb{R}}

\def\E{\mathbb{E}}
\def\P{\mathbb{P}}

\def\T{\mathcal{T}}

\def\r{\mathtt{r}}

\def\eps{\varepsilon}

\def\X{\mathbf{X}}

\def\a{\mathbf{a}}
\def\b{\mathbf{b}}
\def\r{\mathbf{r}}
\def\c{\mathbf{c}}

\def\x{\mathbf{x}}

\usepackage{graphicx}

\makeatletter
\newcommand*\rel@kern[1]{\kern#1\dimexpr\macc@kerna}
\newcommand*\widebar[1]{%
	\begingroup
	\def\mathaccent##1##2{%
		\rel@kern{0.8}%
		\overline{\rel@kern{-0.8}\macc@nucleus\rel@kern{0.2}}%
		\rel@kern{-0.2}%
	}%
	\macc@depth\@ne
	\let\math@bgroup\@empty \let\math@egroup\macc@set@skewchar
	\mathsurround\z@ \frozen@everymath{\mathgroup\macc@group\relax}%
	\macc@set@skewchar\relax
	\let\mathaccentV\macc@nested@a
	\macc@nested@a\relax111{#1}%
	\endgroup
}
\DeclareMathOperator{\diag}{diag}

\DeclareMathOperator*{\argmin}{arg\,min}

\usepackage{theoremref}

\newtheorem{theorem}{Theorem}
\numberwithin{theorem}{section}
\numberwithin{equation}{section}

\newtheorem{lemma}[theorem]{Lemma}
\newtheorem{prop}[theorem]{Proposition}

\newtheorem{assumption}[theorem]{Assumption}

\theoremstyle{definition}
\newtheorem{definition}[theorem]{Definition}

\newtheorem{example}[theorem]{Example}
\newtheorem{remark}[theorem]{Remark}

\allowdisplaybreaks

\definecolor{hancolor}{rgb}{0.0 0.0, 1.0}
\makeatletter
\newcommand{\addresseshere}{%
	\enddoc@text\let\enddoc@text\relax
}
\makeatother

\newcommand{\Y}{\mathbf{Y}}
\newcommand{\A}{\mathbf{A}}
\newcommand{\cX}{\mathcal{X}}

\DeclarePairedDelimiter\norm{\lVert}{\rVert}
\usepackage{bbm} % for indicator function
\def\1{\mathbbm{1}}

\begin{document}
	
	%\title[]{Poisson Random Matrices with Prescribed Margins: Rescaling, Conditioning, and the Schrödinger Bridge} 
	\title[]{Scaling limit of Sinkhorn-rescaled Random Matrices \\ via Stability of Static Schrödinger  Bridges  }

	\author{Danny Duan}
	\address{Danny Duan, Department of Mathematics, University of Wisconsin - Madison, WI, 53717, USA}
	\email{\texttt{bduan5@wisc.edu}}

	\author{Hanbaek Lyu}
	\address{Hanbaek Lyu, Department of Mathematics, University of Wisconsin - Madison, WI, 53717, USA}
	\email{\texttt{hlyu@math.wisc.edu}}

	\author{William Powell}
	\address{William Powell, Department of Mathematics, University of Wisconsin - Madison, WI, 53717, USA}
	\email{\texttt{wgpowell@wisc.edu}}

	\keywords{Random matrices, matrix scaling, Sinkhorn algorithm, Schr\"{o}dinger bridge, entropic optimal transport, stability, scaling limit}
	\subjclass[2020]{60B20, 60F05, 49Q22, 94A17}

	\begin{abstract}
		We analyze the asymptotic behavior and scaling limits of large random matrices rescaled via the Sinkhorn algorithm to match prescribed row and column margins. For a random matrix with independent sub-exponential entries, we show that its Sinkhorn rescaling concentrates around the rescaling of its mean matrix, both at the level of the Schr\"odinger potentials and as random measures on the unit square, with explicit non-asymptotic rates. As the dimensions grow, the rescaled random matrix converges to the continuous static Schr\"odinger bridge (SSB) determined by the limiting margins and reference density. Around this scaling limit we develop a fluctuation theory: bulk rigidity for the empirical spectral distribution of the associated sample covariance matrix, and a central limit theorem for the empirical Schr\"odinger potentials of the rescaled empirical mean.
		
		Our analysis is driven by a new quantitative stability theory for the SSB, developed in three forms: Lipschitz continuity in the Hellinger distance under perturbations of the reference measure (kernel stability); H\"older-$1/2$ continuity in the Hellinger distance under $L^1$ perturbations of the margins (margin stability); and $L^\infty$ stability of the discrete Schr\"odinger potentials under margin perturbation (potential stability). Translated to the discrete random-matrix setting, these bounds yield the concentration and scaling-limit results, while a local law for random Gram matrices with a non-uniform variance profile drives the bulk rigidity. Our SSB stability theory may be of independent interest.
	\end{abstract}
	
	\maketitle
	
	\let\cleardoublepage\clearpage
	\vspace{-0.5cm}
	%\tableofcontents

	\section{Introduction}

	\subsection*{Matrix scaling and Sinkhorn's algorithm}
	
	%Below we describe how the rescaled model is constructed. While a typical realization of $\X$ need not belong to $\T(\r,\c)$, we can force it via either rescaling. 
	
	Whether a given matrix with positive entries can be rescaled via diagonal matrices to satisfy prescribed row and column margins is a classical problem known as %the 
	\textit{matrix scaling} \cite{sinkhorn1964relationship, sinkhorn1967concerning}. In most cases such a rescaled matrix cannot be written as a simple function of the input matrix and the target margins, but one can compute it (if it exists) via an iterative algorithm known as the \textit{Sinkhorn algorithm} (a.k.a. iterated proportional fitting), which is to simply rescale the rows and the columns of the input matrix alternatively to match the target row and column sums until convergence (see a recent survey \cite{idel2016review}).
	
	More precisely, given a nonnegative $m\times n$ matrix $\bLambda$ and target row and column margins $(\r,\c)$, one seeks diagonal matrices $\D_{1},\D_{2}$ such that $\D_{1}\bLambda \D_{2}$ lies in the transportation polytope
	\begin{align}\label{eq:transport_polytope}
		\T(\r,\c) := \left\{ \x\in \R_{\ge 0}^{m\times n}\,:\,  r(\x)=\r, \,\, c(\x)=\c   \right\},
	\end{align}
	where $r(\x)$ and $c(\x)$ denote the row and column margin vectors of matrix $\x$. It is well known \cite{deming1940least, sinkhorn1964relationship, sinkhorn1967concerning, ireland1968contingency} that the rescaled matrix $\D_{1}\bLambda \D_{2}$, whenever it exists, is the unique solution to the following relative entropy minimization problem 
	\begin{align}\label{eq:SSB_Poisson_main}
		\bLambda^{\r,\c}:=\argmin_{\bZ \in \T(\r,\c)} \left[ D_{\textup{KL}}(\bZ\,\Vert \, \bLambda) = \sum_{ij}  \bZ_{ij}\log \left( \bZ_{ij}/\bLambda_{ij} \right)  \right]. 
	\end{align}
	A standard Lagrange multiplier argument shows that 
	\begin{align}
		\bLambda^{\r,\c}= \exp(\balpha\oplus \bbeta) \odot \bLambda = (e^{\balpha(i)}\bLambda_{ij} e^{\bbeta(j)})_{ij},
	\end{align}
	where $\oplus$ and $\odot$ denote the direct sum and the entry-wise product, respectively. Here, $\balpha\in \R^{m}$ and $\bbeta\in \R^{n}$ satisfying the above relation are called the \textit{Schr\"{o}dinger potentials}, %Hence $\bZ^{\r,\c;\bLambda}$ is indeed of the form $\D_{1}\bLambda \D_{2}$ with $\D_{1}$ is the diagonal matrix of diagonal entries $e^{\balpha}$ and similarly for $\D_{2}$. Conversely, 
	which also solve the following \textit{Kantorovich dual} problem 
	\begin{align}\label{eq:dual}
		\max_{\balpha,\bbeta} \,\, \langle \balpha, \r  \rangle + \langle \bbeta, \c \rangle - \langle   \exp(\balpha\oplus \bbeta), \bLambda \rangle.
	\end{align}
	%{\color{red} Will- Suggestion (We already introduced sinkhorn and IPF above): Sinkhorn's algorithm is implemented by maximizing the objective in \eqref{eq:dual} with respect to $\balpha$ and $\bbeta$ in an alternating fashion. The update can be written in closed form as:}
	Alternating maximization for \eqref{eq:dual} in $\balpha$ and $\bbeta$ gives the dual form of Sinkhorn's algorithm, which iteratively updates the potentials as follows:
	\begin{align}\label{eq:sinkhon}
		\begin{cases}
			\textit{For $1\le j \le n$, }	\bbeta_{k}(j) \leftarrow \log\left(   \c(j)\bigg/ \sum_{i=1}^{m} \exp(\balpha_{k-1}(i))  \bLambda_{ij}\right),\\
			\textit{For $1\le i \le m$, } \balpha_{k}(i) \leftarrow \log \left(  \r(i) \bigg/  \sum_{j=1}^{n} \exp(\bbeta_{k}(j)) \bLambda_{ij} \right). 
		\end{cases}
	\end{align}
	It is well-known that the above algorithm converges exponentially fast \cite{sinkhorn1967concerning, choudharylinear} to Schr\"{o}dinger potentials if they exist.

	\subsection*{Sinkhorn-rescaled random matrices and the scaling limit}
	
	Given a nonnegative $m\times n$ random matrix $\X$ with independent entries and mean $\bLambda$, as well as a margin $(\r,\c)$, we define the Sinkhorn-rescaled random matrix $\X^{\r,\c}$ by replacing the deterministic input matrix $\bLambda$ in \eqref{eq:SSB_Poisson_main} with the random matrix $\X$:
	\begin{align}
		\X^{\r,\c} =  \D_{1}(\X)\, \X \,\D_{2}(\X), 
	\end{align}
	which is well-defined whenever the realization of $\X$ can be rescaled to the margin $(\r,\c)$ (it is well-defined a.s. if $\X$ is a.s. a positive matrix). The diagonal scaling matrices $\D_{1}(\X)$ and $\D_{2}(\X)$ above depend on $\X$. The fundamental questions we investigate in this work are the following. 
	\begin{description}
		\item[Q1.]  What is the structure of $\X^{\r,\c}$ at fixed dimensions $m\times n$? 
		
		%{\color{red} Will- The statement of Q2 reads a little unclear to me. Maybe ``Suppose $\X_k \in \R^{m_k \times n_k}$ is a sequence of random matrices with mean $\bLambda_k$ and $(\r_k, \c_k) \in \R^{m_k} \times \R^{n_k}$ is a sequence of margins.  If the means $\bLambda_k$ and margins $(\r_k, \c_k$) converge in a suitable sense as the dimensions grow to infinity, what is the scaling limit of $\X_k^{\r_k, \c_k}$?" }
		\item[Q2.] %Given a sequence of margins $(\r_{k},\c_{k})$ and random matrix $\X_{k}$ with growing dimensions with the mean matrix $\E\X_{k}=\bLambda_{k}$ converging in a suitable sense, what is the scaling limit of $\X_{k}^{\r_{k},\c_{k}}$ as the dimensions grow to infinity?
		Suppose $\X_k \in \R^{m_k \times n_k}$ is a sequence of random matrices with mean $\bLambda_k$ and $(\r_k, \c_k) \in \R^{m_k} \times \R^{n_k}$ is a sequence of margins. If the means $\bLambda_k$ and margins $(\r_k, \c_k)$ converge in a suitable sense as the dimensions grow to infinity, what is the scaling limit of $\X_k^{\r_k, \c_k}$?
	\end{description}
	
	The key challenge in analyzing the rescaled random matrix $\X^{\r,\c}$ is that the random scaling factors $\D_{1}(\X), \D_{2}(\X)$ are nonlinear functions of the realization $\X$ without an explicit functional form. A natural idea then is to decouple the dependency between the scaling factors and $\X$ by using the deterministic scaling factors computed for the expectation $\bLambda=\E\X$. Namely, we introduce the following \textit{comparison model}
	\begin{align}\label{eq:comparison_model_intro}
		\hat{\X}^{\r,\c} := \D_{1}(\bLambda) \,\X\, \D_{2}(\bLambda). 
	\end{align}
	%Since $\E[\hat{\X}^{\r,\c}] = \D_{1}(\bLambda)\bLambda\D_{2}(\bLambda)$ is precisely the solution to the discrete SSB \eqref{eq:SSB_Poisson_main}, showing that $\X^{\r,\c} \approx \hat{\X}^{\r,\c}$ reduces the analysis to the concentration of independent random variables around their means. The following result justifies this approximation by showing that the random Schr\"{o}dinger potentials concentrate around their deterministic counterparts.
	Our analysis is roughly based on justifying the following approximation scheme:
	\begin{align}
		\X^{\r,\c} \approx \hat{\X}^{\r,\c} \approx \bLambda^{\r,\c}.
	\end{align}
	To this effect, we are motivated to study the \textit{stability} of the matrix scaling problem. 
	\begin{description}
		\item[Q3.] How do the rescaled matrix $\bLambda^{\r,\c}$, the Schr\"{o}dinger potentials $\balpha,\bbeta$, and the scaling factors $\D_{i}$, change  when perturbing the margin $(\r,\c)$ and the input matrix $\bLambda$?
	\end{description}
	%Namely, comparing $\X^{\r,\c}$ and $\hat{\X}^{\r,\c}$ amounts to bounding the difference between the random scaling factors $\D_{i}(\X)$ 

	\subsection*{Passage to limit theory: The Static Schr\"{o}dinger Bridge}

	A natural candidate for the limiting object for matrix scaling with growing dimensions is the \textit{static Schr\"{o}dinger bridge} (SSB). Suppose one seeks the joint probability distribution closest in relative entropy to a given reference measure among all couplings with prescribed marginals \cite{fortet1940resolution}. For marginal measures $\mu_{0}$ and $\mu_{1}$ and a reference measure $\mathcal{R}$, the SSB is defined as
	\begin{align}\label{eq:RM_min_SB}
		\pi^{\mu_{0},\mu_{1};\mathcal{R}} = \argmin_{\mathcal{H}\in \Pi(\mu_{0},\mu_{1})}   D_{\textup{KL}}(\mathcal{H}\, \Vert \, \mathcal{R}),
	\end{align}
	where $D_{\textup{KL}}(\mathcal{H}\, \Vert \, \mathcal{R})= \int \log  \frac{d\mathcal{H}}{d\mathcal{R}}  \, d\mathcal{H}$ if $\mathcal{H} \ll \mathcal{R}$ and $+\infty$ otherwise. When $\mathcal{R}\ll \mu_{0}\otimes \mu_{1}$, writing the Radon--Nikodym derivative as %{\color{red} Will- do we want to use $\kappa$ instead of $c$ here since we do so in the rest of the paper?}
	%\commHL{Let's use $\kappa$ throughout.}
	\begin{align}
		\frac{d \mathcal{R}}{d\mu_{0} d\mu_{1}}(x,y) = \exp(-\kappa(x,y)/\eps),
	\end{align}
	the SSB problem is equivalent to the \textit{entropic optimal transport} (EOT) problem
	\begin{align}\label{eq:EOT}
		\min_{\mathcal{H}\in \Pi(\mu_{0},\mu_{1})} \int \kappa(x,y)\, d\mathcal{H}(x,y) + \eps\, D_{\textup{KL}}(\mathcal{H}\,\Vert \, \mu_{0}\otimes \mu_{1}),
	\end{align}
	where $\kappa(\cdot,\cdot)$ is the cost function and $\eps>0$ is the entropic regularization parameter. The SSB thus provides a variational framework unifying entropic regularization in optimal transport \cite{villani2021topics, nutz2021introduction} with the classical theory of Schr\"{o}dinger bridges \cite{pavon2021data}. Clearly the finite dimensional matrix scaling problem \eqref{eq:SSB_Poisson_main} is a special instance of the SSB problem \eqref{eq:RM_min_SB}. Indeed, without loss of generality, we can assume that $\r,\c$, and $\bLambda$ are probability measures on $[m]$, $[n]$, and $[m]\times [n]$, respectively. 
	
	The technical core of this work establishes stability theory for SSBs, addressing the question \textbf{Q3} on the stability of matrix scaling in the general setting of SSBs:
	\begin{description}
		\item[Q4.] How do the SSB, the Schr\"{o}dinger potentials, and the scaling factors change  when perturbing the margin and the reference measures? 
	\end{description}
	While stability of the entropic optimal transport plan or Schr\"{o}dinger potentials with respect to the margin have been investigated recently \cite{carlier2020differential, eckstein2022quantitative, deligiannidis2024quantitative, nutz2023stability}, their stability with respect to the input matrix $\bLambda$ for matrix scaling or the reference measure $\mathcal{R}$ for SSB are less well-developed.

	\subsection{Organization}
	The remainder of this paper is organized as follows. Section \ref{sec:results} provides the formal statements of our main results, sequentially covering the continuous stability estimates (Section \ref{sec:SSB_stability_statements}), the non-asymptotic discrete concentration bounds (Section \ref{sec:concentration_statements}), the scaling limit theorem (Section \ref{sec:scaling_limit_statements}), and the fluctuation analysis including bulk rigidity and the Central Limit Theorem (Section \ref{sec:fluctuation_statements}). Section \ref{sec:related_work} situates our contributions within the broader literature on entropic optimal transport, matrix scaling, biwhitening, and random contingency tables. 
	
	The proofs of our main theorems are divided into two sections. Section \ref{sec:SSB_stability} is devoted exclusively to the stability theory for the static Schr\"{o}dinger bridges and potentials. Finally, Section \ref{sec:rescaled_proofs} contains all proofs for the finite-dimensional Sinkhorn-rescaled random matrices, sequentially establishing the discrete concentration of potentials, the scaling limit, and the empirical fluctuation limits.

	\section{Statement of results}
	\label{sec:results}

	\subsection{Overview of results}
	\label{sec:results_roadmap}

	The aim of this paper is to develop a quantitative theory of Sinkhorn-rescaled random matrices. Our main results, all stated formally in this section, fall into four groups.

	\smallskip
	\noindent\emph{Concentration of the rescaled random matrix.} For $\X$ an $m\times n$ random matrix with independent sub-exponential entries and mean $\bLambda$, scaled to a target margin $(\r,\c)$, the random Schr\"{o}dinger potentials $(\balpha_{\X}, \bbeta_{\X})$ concentrate in $\ell^{\infty}$ around their deterministic counterparts $(\balpha,\bbeta)$ for $\bLambda$ (Theorem \ref{thm:concentration_margins_rs}). Consequently the rescaled matrix $\X^{\r,\c}$ is close to a comparison model in operator norm and to the deterministic bridge $\bLambda^{\r,\c}$ against bounded test functions, with explicit non-asymptotic rates (Theorems \ref{thm:first_approximation_rescaled} and \ref{thm:bound_X^rs-Z}).

	\smallskip
	\noindent\emph{Scaling limit.} Under a histogram-convergence assumption on the margins and the prior mean matrices, the sequence of rescaled random matrices converges, with an explicit rate, to the unique \emph{continuous static Schr\"{o}dinger bridge} (SSB) determined by the limiting margins and reference density (Theorem \ref{thm:RM_rescaled_limit}). The deterministic limit alone is treated separately as a quantitative Hellinger-distance bound on the convergence of discrete bridges to the continuous SSB (Theorem \ref{thm:SSB_deterministic_limit}).

	\smallskip
	\noindent\emph{Bulk rigidity of the sample covariance matrix.} The fluctuation matrix $\A$ extracted from $\X^{\r,\c}$ has a non-uniform variance profile $\mathbf{S}$, and the eigenvalue distribution of $\A\A^T$ is governed by the Dyson equation associated with $\mathbf{S}$. We prove bulk rigidity: the eigenvalues of $\A\A^T$ concentrate around the classical locations determined by this Dyson equation, with sharp rates and spectral confinement (Theorem \ref{thm:ESD}).

	\smallskip
	\noindent\emph{Central limit theorem for empirical Schr\"{o}dinger potentials.} For the Sinkhorn potentials computed from the empirical mean of $M$ i.i.d.\ copies of $\X$, the rescaled fluctuations around the deterministic potentials converge in distribution to a multivariate Gaussian as $M\to\infty$ (Theorem \ref{thm: CLT_potentials_fixed_dim}).

	\smallskip
	\noindent\emph{Stability of the SSB and its role.} All four groups of random-matrix results rest on a common deterministic input: a quantitative stability theory for the continuous and discrete static Schr\"{o}dinger bridge. We develop this theory in three forms — Lipschitz stability of the bridge in Hellinger distance under perturbations of the reference measure (Theorem \ref{thm:continuous_kernel_stability}), H\"older-$1/2$ stability under $L^1$ perturbations of the margins (Theorem \ref{thm:continuous_margin_stability}), a combined total stability bound (Theorem \ref{thm:total_stability}), and an $L^{\infty}$ stability bound for the discrete Schr\"{o}dinger potentials with explicit constants (Theorem \ref{thm:potential_stability}). These stability results are of independent interest. Because the random-matrix theorems quote the stability constants in their statements, we present the stability theory first (\S\ref{sec:SSB_stability_statements}) and then build on it to state the random-matrix results in \S\ref{sec:concentration_statements}--\S\ref{sec:fluctuation_statements}.

	\subsection{Stability of the static Schr\"{o}dinger bridge and potentials}
	\label{sec:SSB_stability_statements}

	In this section, we state three quantitative stability estimates for the static Schr\"{o}dinger bridge and Schr\"{o}dinger potentials  with respect to perturbations of the kernel and the margins. 
	
	We first formulate our two stability results for the static Schr\"{o}dinger bridge in a measure-theoretic setting. For a measure space $(\cX, \cF, \mu)$, we write $\mathcal{P}_{\ge 0}(\cX, \mu)$ for the set of probability measures that have density with respect to $\mu$ on $\cX$, and $\mathcal{P}_{>0}(\cX, \mu)$ for the strictly positive subset, i.e., the set of probability measures equivalent to $\mu$. 
	
	We fix marginal probability spaces $(\Omega_{\r}, \cF_{\r}, \mu)$ and $(\Omega_{\c}, \cF_{\c}, \nu)$. To formulate our stability results, we consider the family of static Schr\"{o}dinger bridges  %$\pi^{\rho_{\r},\rho_{\c};\varphi}$
	$\pi^{\mu_0, \mu_1; \mathcal{R}}$
	with marginals %$\rho_{\r}
	$\mu_0 \in \mathcal{P}_{\ge 0}(\Omega_{\r}, \mu)$, %$\rho_{\c} 
	$\mu_1 \in \mathcal{P}_{\ge 0}(\Omega_{\c}, \nu)$, and reference measure $\mathcal{R} \in \mathcal{P}_{>0}(\Omega_{\r} \otimes \Omega_{\c}, \mu_0 \otimes \mu_1)$. 
	We then provide bounds on the change in the bridges as we perturb the reference measure or the marginal densities.

	In the discrete case of the matrix scaling problem (see \eqref{eq:SSB_Poisson_main}), several necessary and sufficient conditions for the existence of $\pi^{\mu_0, \mu_1; \mathcal{R}}$ are known (see, e.g., Menon and Schneider \cite{menon1969spectrum} and Proposition \ref{prop:matrix_scaling_equiv}), and the uniqueness follows from the existence. In the general continuous setting, a sufficient condition for existence and uniqueness of the bridge is that $\mathcal{R}$ is measure-theoretically equivalent to $\mu_0 \otimes \mu_1$ and there is at least one coupling $\pi \in \Pi(\mu_0,\mu_1)$ with finite KL-divergence to $\mathcal{R}$ (see \cite[Thm. 2.1]{nutz2021introduction}).
	
	Recall that for two measures $\mu, \mu'$ both absolutely continuous with respect to a base measure $\sigma$, the Hellinger distance is
	\begin{align}
		d_{H}(\mu, \mu') = \left( \int \left(\sqrt{\frac{d\mu}{d\sigma}} - \sqrt{\frac{d\mu'}{d\sigma}}\right)^2 d\sigma \right)^{1/2}
	\end{align}
	and is independent of the choice of $\sigma$. For our purposes, it will be convenient to use either ${\sigma} = \mu \otimes \nu$ or $\sigma = \mu_0 \otimes \mu_1$.  We also remind the reader that the Hellinger metric induces the same topology as the total variation/$L_1$ metric through the relation
	\begin{align}\label{eq:TV_Hellinger}
		d_{H}(\mu, \mu')^2 \leq d_{TV}(\mu, \mu') \leq 2\sqrt{2}d_{H}(\mu, \mu')
	\end{align}
	where $d_{TV}(\mu, \mu') = \int |\frac{d\mu}{d\sigma} - \frac{d\mu'}{d\sigma}|d\sigma$. 
	%$d\sigma = \rho_{\r} d\mu\otimes \rho_{\c} d\nu$.

	Our first result establishes that, for fixed margins, the map from the %kernel $\varphi$ 
	reference measure $\mathcal{R}$ to the bridge $\pi^{\mu_0, \mu_1; \mathcal{R}}$ is Lipschitz continuous in the Hellinger distance.

	\begin{theorem}[Stability of SSB w.r.t.\ reference measure]
		\label{thm:continuous_kernel_stability}
		Fix marginal probability measures $\mu_0 \in \mathcal{P}_{\ge 0}(\Omega_{\r}, \mu)$ and $\mu_1 \in \mathcal{P}_{\ge 0}(\Omega_{\c}, \nu)$.
		Fix reference probability measures $\mathcal{R},\mathcal{R}'$ where $d\mathcal{R}= e^{-\kappa} d(\mu_0\otimes \mu_{1})$ and $d\mathcal{R}'= e^{-\kappa'} d(\mu_0\otimes \mu_{1})$ for some
		$\kappa, \kappa' \in L^{\infty}(\mu_0 \otimes \mu_1)$. Then we have
		\begin{align}
			d_{H}(\pi^{\mu_0,\mu_1; \mathcal{R}}, \pi^{\mu_0,\mu_1; \mathcal{R}'}) \leq \exp\left(\tfrac{3}{2}\norm{\kappa^{+}}_{\infty} \vee \norm{(\kappa')^{+}}_{\infty}\right)d_{H}(\mathcal{R}, \mathcal{R}').
		\end{align}
	\end{theorem}
	
	A few aspects of the result above are worth noting. Since we assume $\mathcal{R}, \mathcal{R}' \sim \mu_0\otimes\mu_1$, the bridges $\pi^{\mu_0,\mu_1;\mathcal{R}}$ and $\pi^{\mu_0,\mu_1;\mathcal{R}'}$ share the same support. Comparing them in the Hellinger metric --- a strong topology on densities --- is therefore natural, in contrast to the weaker topologies (e.g., Wasserstein) %or total variation)
	used in much of the literature.
	Furthermore, the factor $\exp\bigl(\tfrac{3}{2}\|\kappa^{+}\|_{\infty} \vee \|(\kappa')^{+}\|_{\infty}\bigr)$ cannot be improved in general: Example \ref{ex:kernel_stability_tightness} exhibits a family of reference measures for which the Lipschitz constant must diverge as the cost grows.
	
	Since perturbing the reference measure $d\mathcal{R} = e^{-\kappa}d(\mu_0 \otimes \mu_1)$ is equivalent to perturbing the cost function $\kappa$, Theorem \ref{thm:continuous_kernel_stability} can be viewed as a stability result with respect to the cost under fixed marginals. A detailed comparison with existing stability results in the literature is deferred to Section \ref{sec:rw_stability}.

	Our second result controls the bridge under perturbations of the marginals when the kernel is fixed. 
	
	\begin{theorem}[Stability of SSB w.r.t.\ margins]
		\label{thm:continuous_margin_stability}
		Fix pairs of marginal probability measures $(\mu_{0},\mu_{1})$  and $(\mu_{0}',\mu_{1}')$, where $\frac{d\mu_{0}}{d\mu}=\rho_{\r}$, $\frac{d\mu_{0}'}{d\mu}=\rho_{\r}'$, $\frac{d\mu_{1}}{d\nu}=\rho_{\c}$, and $\frac{d\mu_{1}'}{d\nu}=\rho_{\c}'$ for some positive probability densities $\rho_{\r},\rho_{\r}'$, $\rho_{\c}$, and $\rho_{\c}'$. Fix a reference probability measure $\mathcal{R}$ and suppose $\mathcal{R}\sim \mu_{0}\otimes \mu_{1}$  and $\mathcal{R}\sim \mu_{0}'\otimes \mu_{1}'$. Write cost functions $\kappa,\kappa'$ as $ e^{-\kappa} =  \frac{d\mathcal{R}}{d(\mu_{0}\otimes \mu_{1})}$ and $ e^{-\kappa'} =  \frac{d\mathcal{R}}{d(\mu_{0}'\otimes \mu_{1}')}$. 
		Then we have 
		\begin{align}
			d_{H}(\pi^{\mu_{0},\mu_{1};\mathcal{R}}, \pi^{\mu_{0}',\mu_{1}';\mathcal{R}})^2 \leq 4 (\| \kappa \|_{\infty}\lor \| \kappa' \|_{\infty}) \left( \norm{\rho_{\r} - \rho_{\r}'}_{L^1(\mu)} + \norm{\rho_{\c} - \rho_{\c}'}_{L^1(\nu)} \right).
		\end{align}
	\end{theorem}
	
	Two features of this bound are worth noting. First, the left-hand side is the \textit{square} of the Hellinger distance, giving Hölder-$1/2$ regularity in the $L^1$ distance between the margins. Second, the constant depends only \textit{linearly} on the cost, milder than the exponential dependence $e^{\|\kappa\|_{\infty}}$ in the %kernel
	reference stability (Theorem \ref{thm:continuous_kernel_stability}). 
	
	By combining Theorems \ref{thm:continuous_kernel_stability} and \ref{thm:continuous_margin_stability} via the triangle inequality, we obtain total stability under simultaneous perturbation of both the kernel and the margins.
	
	\begin{theorem}[Total stability of SSB]
		\label{thm:total_stability}
		Make the same assumptions as in Theorems \ref{thm:continuous_kernel_stability} and \ref{thm:continuous_margin_stability}. Let $d\mathcal{R} = e^{-\kappa}d(\mu_0 \otimes \mu_1)$ and $d\mathcal{R}' = e^{-\kappa'}d(\mu_0' \otimes \mu_1')$. Further assume that
		\begin{align}
			e^{-M} \leq \frac{d(\mu_0 \otimes \mu_1)}{d(\mu_0' \otimes \mu_1')} \leq e^{M} \quad \mu \otimes \nu \text{ a.s. }
		\end{align}
		for some $M \ge 0$. Then 
		\begin{align}
		&	d_{H}(\pi^{\mu_{0}, \mu_{1}; \mathcal{R}}, \pi^{\mu_{0}', \mu_{1}'; \mathcal{R}'})^2\\
		&\qquad	\leq 8(\norm{\kappa}_{\infty} + M) \left( \norm{\rho_{\r} - \rho_{\r}'}_{L^1(\mu)} + \norm{\rho_{\c} - \rho_{\c}'}_{L^1(\nu)} \right) + 4e^{3(\norm{\kappa}_{\infty} + M) \vee \norm{(\kappa')}_{\infty}}\, d_{H}(\mathcal{R}, \mathcal{R}')^2.
		\end{align}
	\end{theorem}
	
	\begin{remark}[Role of the constant $M$]
		\label{rmk:total_stability_M}
		The constant $M$ in Theorem \ref{thm:total_stability} bounds the Radon--Nikodym derivative $d(\mu_0\otimes\mu_1)/d(\mu_0'\otimes\mu_1')$ and appears solely to handle a change-of-reference step in the triangle inequality combining Theorems \ref{thm:continuous_kernel_stability} and \ref{thm:continuous_margin_stability}. In the scaling limit theory of Section \ref{sec:scaling_limit_statements}, where $\delta$-smoothness (see Def. \ref{def:delta_smooth_margins}) makes marginal densities uniformly comparable, $M$ is absorbed into constants depending only on $\delta$, and does not affect the asymptotic convergence rates.
	\end{remark}
	
	This total stability estimate is the main tool in our scaling limit theory: it allows us to control the convergence of a sequence of discrete Schr\"{o}dinger bridges (with converging margins and kernels) to the continuous SSB.

	The kernel and margin stability results above (Theorems \ref{thm:continuous_kernel_stability} and \ref{thm:continuous_margin_stability}) control the coupling $\pi$ in the Hellinger metric. While these results are sufficient for establishing scaling limits, the analysis of finite-dimensional Sinkhorn-rescaled random matrices requires a complementary result that controls the discrete \textit{Schr{\"o}dinger potentials} (the logarithms of the scaling factors) in the $L^{\infty}$ norm, and that remains valid when the input matrix may have zero entries. The following theorem provides such a stability bound. The closest precedent is \cite[Lem.~9]{landa2022scaling}, which establishes an $\ell^{\infty}$ stability bound for the scaling factors under approximate scaling by extending the uniqueness argument of Sinkhorn and Knopp \cite{sinkhorn1967concerning}; however, the constants there scale with the ratio of the maximum to the minimum entry of the input matrix, and therefore degenerate whenever any entry vanishes. Our result instead relies on a \textit{row-alignment} parameter that measures, in aggregate across columns, the inner-product overlap between pairs of rows. This condition admits matrices with many zero entries, and ensures that the Sinkhorn scaling map remains a well-conditioned local diffeomorphism, allowing the potentials to be controlled via a Banach fixed-point argument on its linearization at the origin. This is the key technical ingredient for the concentration theory for rescaled random matrices stated in Section \ref{sec:concentration_statements}.

	\begin{theorem}[Stability of discrete Schr\"{o}dinger potentials]\label{thm:potential_stability}
		Let $\A \in \R^{m\times n}$ be a matrix with non-negative entries, and let $\r \in \R^m_{>0}$ and $\c \in \R^n_{>0}$ be target margin vectors with strictly positive entries such that $\|\r\|_1 = \|\c\|_1 = N$. 
		Let $\r(\A) = \A \mathbf{1}_n$ and $\c(\A) = \A^T \mathbf{1}_m$ denote the current row and column margins of $\A$. Assume that the target margin $(\r,\c)$ is close to the margin of $\A$ in the sense that 
		\begin{align}
			\r(\A) \in [(1-\eps) \r, (1+\eps)\r], \qquad \c(\A) \in [(1-\eps) \c, (1+\eps)\c]
		\end{align}
		entry-wise for some  $\epsilon \in (0, \epsilon_{\textup{max}})$, where 
		\begin{align}\label{eq:eps_max_def}
			\eps_{\max}:=\frac{1}{50C_{\A}^2}, \quad   C_{\A} := %1 + \frac{9 \max_i \r_i \max_j \c_j}{2 \rho_{\A} m n}
			1 + \frac{9 \max_i \r_i \max_j \c_j}{\rho_{\A} m n}, \quad  \rho_{\A}:= \min_{i_{1},i_{2}\in \{1,\dots,m\}} n^{-1} \sum_{j=1}^{n}\A_{i_1 j}\A_{i_2 j}.
		\end{align}
		Then there exist exact Schr\"{o}dinger potentials $\balpha_{\A} \in \R^m, \bbeta_{\A} \in \R^n$ such that $\diag(e^{\balpha_{\A}})\A\diag(e^{\bbeta_{\A}})$ has exactly the margins $(\r, \c)$. These potentials are unique up to the one-parameter gauge shift $(\balpha_{\A},\bbeta_{\A})\mapsto(\balpha_{\A}-s\mathbf{1}_m,\bbeta_{\A}+s\mathbf{1}_n)$ with $s\in\R$, and can be chosen (via any canonical gauge-fixing, e.g.\ the one in the proof that projects onto the orthogonal complement of the gauge direction) such that
		\begin{align}\label{eq:potential_bd_approximate_scaling}
			\|(\balpha_{\A}, \bbeta_{\A})\|_{\infty} \le 4C_{\A} \epsilon.
		\end{align}
	\end{theorem}

	The bound \eqref{eq:potential_bd_approximate_scaling} characterizes how the exact potentials deviate from zero when the current margins of $\A$ are already close to the target margins. The term $4C_{\A}\epsilon$ reflects two distinct sources of stability: the linear sensitivity to the relative margin error $\epsilon$, and the geometric stability of the scaling map captured by the constant $C_{\A}$.
	
	These stability results address different but complementary aspects of the perturbation theory. Theorems \ref{thm:continuous_kernel_stability} and \ref{thm:continuous_margin_stability} provide the "macroscopic" stability needed for scaling limits, while Theorem \ref{thm:potential_stability} provides the "microscopic" $L^\infty$ stability required for concentration analysis. Together, they provide a complete toolkit for analyzing Sinkhorn-rescaled random matrices.

	\subsection{Concentration of Sinkhorn-rescaled random matrices}
	\label{sec:concentration_statements}
	
	We now turn to the finite-dimensional analysis of Sinkhorn-rescaled random matrices. The results in this section provide the probabilistic machinery needed to pass from the deterministic stability theory (Section \ref{sec:SSB_stability_statements}) to the stochastic setting.
	
	Throughout this section, we fix a $\delta$-smooth margin $(\r,\c)$ in the following sense.
	
	\begin{definition}[Scalability]\label{def:scalability}
		Fix an $m\times n$ margin $(\r,\c)$. Let $\mathcal{S}(\r,\c)$ denote the set of all $m\times n$ nonnegative matrices $\A$ that are \textit{scalable to margin $(\r,\c)$}, that is, there exist positive diagonal matrices $\D_{1},\D_{2}$ such that $\D_{1}\A\D_{2}\in \T(\r,\c)$.
	\end{definition}
	
	Next, we introduce the following notion of ``smoothness'' of a margin, first used by Barvinok in \cite{barvinok2010does} in the study of random contingency tables. 
	
	\begin{definition}[$\delta$-smooth margins \cite{barvinok2010does}]\label{def:delta_smooth_margins}
		An $m\times n$ margin $(\r,\c)$ with total sum $N := \|\r\|_1 = \|\c\|_1$ is \textit{$\delta$-smooth} for some $\delta>0$ if 
		\begin{align}
			\frac{\delta N}{m} \leq \r(i) \leq \frac{N}{\delta m}, \qquad \frac{\delta N}{n} \leq \c(j) \leq \frac{N}{\delta n}
		\end{align}
		for all $i \in [m]$ and $j \in [n]$.
	\end{definition}
	
	\noindent Note that if $(\r,\c)$ is a positive margin, then it is $\delta$-smooth for some $\delta=\delta(\r,\c)>0$.

	The random matrices we consider satisfy the following standing assumptions.

	\begin{assumption}[Subexponential random matrix]\label{assump:random_matrix}
		Let $\X \in \R_{\geq 0}^{m \times n}$ be a random matrix with independent nonnegative entries and positive mean $\E[\X] = \bLambda\in \R_{> 0}^{m \times n}$. There exist constants $\sigma, R > 0$ such that for all $q \geq 2$,
		\begin{align}\label{eq:moment_growth}
			\E\left[|\X_{ij} - \bLambda_{ij}|^q\right] \leq \frac{q!}{2}\sigma^2 R^{q-2} \qquad \textup{for all $i,j$}.
		\end{align}
	\end{assumption}
	
	\begin{assumption}[Bounded cost]\label{assump:bounded_cost}
		Let $(\r,\c)$ be a $\delta$-smooth margin and $\bLambda \in \R_{>0}^{m\times n}$. There exists a constant $K \ge 0$ such that
		\begin{align}\label{eq:bounded_cost_K}
			\left| \log \frac{\r(i)\c(j)\,\|\bLambda\|_1}{\bLambda_{ij} N^2} \right| \le K \qquad \textup{for all $i\in [m], j\in [n]$}.
		\end{align}
	\end{assumption}

	Assumption \ref{assump:random_matrix} accommodates a wide class of random matrix ensembles, holding for Poisson, bounded, and general sub-exponential entry distributions. Assumption \ref{assump:bounded_cost} states that the normalized mean matrix $\bLambda/\|\bLambda\|_1$ has bounded cost relative to the product of the normalized margins; it serves as the discrete analogue of the bounded-cost condition in Section \ref{sec:SSB_stability_statements}.
	
	\begin{remark}[Scalability is not assumed]
		Note that scalability of $\X$ to the target margin $(\r,\c)$ is not imposed as a hypothesis in Assumptions \ref{assump:random_matrix}--\ref{assump:bounded_cost}: Theorem \ref{thm:concentration_margins_rs} will show that $\X \in \mathcal{S}(\r,\c)$ holds automatically on a high-probability event under these assumptions and $\delta$-smoothness of the margin. For ensembles outside the scope of Assumption \ref{assump:random_matrix}, a matrix-ensemble-agnostic bound on the scalability probability—purely in terms of the smoothness parameter and the maximum zero-probability of entries—can be derived from the Menon--Schneider characterization \cite{menon1969spectrum}; we record this separately as Lemma \ref{lem:scalability} in Section \ref{sec:rescaled_proofs}.
	\end{remark}

	We now apply the deterministic stability bound of Theorem \ref{thm:potential_stability} to the random matrix $\X$ with target margin $(\r,\c)$ to obtain concentration of the random Schr\"{o}dinger potentials around their deterministic counterparts. Specializing the random quantity $C_{\X}$ from \eqref{eq:eps_max_def} (with $\A=\X$) requires controlling the margin sizes and the row-alignment parameter. We use $\delta$-smoothness of $(\r,\c)$ to deduce $\max_i \r(i)\max_j \c(j)\le N^2/(\delta^2 mn)$, while we obtain a high-probability lower bound on $\rho_{\X}$ in terms of $\bLambda_{\min}$ and $K$, where $K$ is the bound on the cost function in  \eqref{eq:bounded_cost_K}.  Substituting these estimates into \eqref{eq:eps_max_def} yields the explicit deterministic constant $C_*$ defined in \eqref{eq:eps0_Cstar_def} below, which serves as a high-probability upper bound for $C_{\X}$. Crucially, under Assumptions \ref{assump:random_matrix}--\ref{assump:bounded_cost} and $\delta$-smoothness the parameters $\delta, K$, and $\bLambda_{\min}/\bLambda_{\max}$ are bounded away from zero, so $C_*=O(1)$; consequently the $4 C_{\X}\eps$ bound of Theorem \ref{thm:potential_stability} reduces to an order-$\eps$ control of the random potentials.
	
	\begin{theorem}[Concentration of random potentials]\label{thm:concentration_margins_rs}
		Let $(\r, \c)$ be an $m \times n$ $\delta$-smooth margin. Under Assumptions \ref{assump:random_matrix} and \ref{assump:bounded_cost}, define the minimum and maximum expected entries $\bLambda_{\min} := \min_{i,j} \bLambda_{ij}$ and $\bLambda_{\max} := \max_{i,j} \bLambda_{ij}$.  Let $(\balpha, \bbeta)$ denote the deterministic Schr\"{o}dinger potentials for $\bLambda$ with margin $(\r, \c)$. Fix $D > 0$ and let $\tau = (D+1)\log(m \vee n) + \log 4$. Define the approximate scaling error $\eps_0$ and the explicit stability constant $C_*$ by:
		\begin{align}\label{eq:eps0_Cstar_def}
			\eps_0 := \frac{e^{2K}(m\vee n)}{\delta \|\bLambda\|_1} \left( \sigma\sqrt{2(m\vee n)\tau} + 2R\tau \right), \qquad C_* := 1 + \frac{18\, e^{8K} \|\bLambda\|_1^2}{\delta^2 \bLambda_{\min}^2 m^2 n^2}.
		\end{align}
		Suppose $\eps_0 \le \frac{1}{50 C_*^2}$, and define the row-alignment rate
		\begin{align}\label{eq:Phi_def}
			\Phi_{\bLambda} := \min\!\left( n\, \frac{\bLambda_{\min}^4}{(\sigma + R + \bLambda_{\max})^4},\ \sqrt{n}\, \frac{\bLambda_{\min}}{\sigma + R + \bLambda_{\max}} \right).
		\end{align}
		Then there exists an absolute constant $c>0$ such that for the following event
		\begin{align}\label{eq:def_event_E1}
			\mathcal{E}_{1} := \{ \X\in \mathcal{S}(\r,\c) \} \cap \left\{ \exists t\in \R\,:\, \max\left( \left\|\balpha_{\X}+t\mathbf{1} - \balpha\right\|_{\infty},\; \left\|\bbeta_{\X}-t\mathbf{1} - \bbeta\right\|_{\infty} \right) \le 4C_* \eps_{0} \right\},
		\end{align}
		we have
		\begin{align}
			\P(\mathcal{E}_{1}) &\ge 1 - (m \vee n)^{-D} - m^2 \exp(-c\,\Phi_{\bLambda}).
		\end{align}
	\end{theorem}
	
	Although the bound in \eqref{eq:eps0_Cstar_def} looks involved, $\eps_0$ is a small quantity in the regime of interest. Treating $\delta, K, \sigma, R$ as $O(1)$ and choosing $D = O(1)$, $\tau = \Theta(\log(m\vee n))$ so the definition of $\eps_0$ simplifies to
	\begin{align}\label{eq:eps0_asymptotic}
		\eps_0 \;=\; \Theta\!\left(\, \frac{(m\vee n)^{3/2}\,\sqrt{\log(m\vee n)}}{\|\bLambda\|_1} \,\right).
	\end{align}
	Under Assumptions \ref{assump:random_matrix}--\ref{assump:bounded_cost} and $\delta$-smoothness of the margin, $C_* = O(1)$ (see \eqref{eq:eps0_Cstar_def} and the computation in the proof), so the hypothesis $\eps_0 \le 1/(50 C_*^2)$ is satisfied for all sufficiently large dimensions whenever
	\begin{align}
		\|\bLambda\|_1 \,\gg\, (m\vee n)^{3/2}\,\sqrt{\log(m\vee n)} \qquad \text{(up to an absolute constant)}.
	\end{align}
	This is the same mass-growth condition that will drive our scaling limit theory in Section \ref{sec:scaling_limit_statements}. Under this regime, the bound $\|\balpha_{\X} + t\mathbf{1} - \balpha\|_\infty, \|\bbeta_{\X} - t\mathbf{1} - \bbeta\|_\infty \le 4 C_*\eps_0$ on $\mathcal{E}_1$ tends to zero at rate $(m\vee n)^{3/2}\sqrt{\log}/\|\bLambda\|_1$: the random Schr\"{o}dinger potentials concentrate sharply around their deterministic counterparts.
	
	The concentration of random potentials in Theorem \ref{thm:concentration_margins_rs} above allows us to compare $\X^{\r,\c}$ with the comparison model $\hat{\X}^{\r,\c}$ and with the deterministic bridge $\bLambda^{\r,\c}$.

	\begin{theorem}[Comparison model approximation]\label{thm:first_approximation_rescaled}
		Suppose Assumptions \ref{assump:random_matrix}--\ref{assump:bounded_cost} hold. % define the minimum and maximum expected entries $\bLambda_{\min} := \min_{i,j} \bLambda_{ij}$ and $\bLambda_{\max} := \max_{i,j} \bLambda_{ij}$.
		Fix $D>0$ and let $\tau$, $t_D$ be as in Lemma \ref{lem:spectral norm of centered poisson} with the sub-exponential parameters $(\sigma,R)$ of Assumption \ref{assump:random_matrix}. Let $\eps_0$ be the approximation error defined in \eqref{eq:eps0_Cstar_def} and let $C_*$ be the explicit stability constant defined in Theorem \ref{thm:concentration_margins_rs}. Assume the dimensions are sufficiently large such that $\eps_0 \le \frac{1}{50 C_*^2}$, and define the aggregate potential error $\eps_{\textup{pot}} := 16 C_* \eps_0 \le 1$.
		Define the two events
		\begin{align}\label{eq:def_event_E2}
			\mathcal{E}_{2} &:= \{ \X\in \mathcal{S}(\r,\c) \} \cap \left\{\tfrac{1}{N}\bigl\|\D(e^{\balpha_{\X}})(\X-\bLambda)\D(e^{\bbeta_{\X}})-\D(e^{\balpha})(\X-\bLambda)\D(e^{\bbeta})\bigr\|_2 \le \tfrac{3 e^{2K}\eps_{\textup{pot}} t_D}{\|\bLambda\|_1}\right\},\\
			\label{eq:def_event_E3}
			\mathcal{E}_{3} &:= \{ \X\in \mathcal{S}(\r,\c) \} \cap \left\{ \tfrac{1}{N}\left\|\X^{\r,\c} - \hat{\X}^{\r,\c}\right\|_1 \le 2\eps_{\textup{pot}} e^{2K} \right\}.
		\end{align}
		Then there exists an absolute constant $c>0$ such that
		\begin{align}
			\P(\mathcal{E}_{2}) &\ge 1 - 2(m \vee n)^{-D} - m^2 \exp(-c\,\Phi_{\bLambda}), \\
			\P(\mathcal{E}_{3}) &\ge 1 - (m \vee n)^{-D} - m^2 \exp(-c\,\Phi_{\bLambda}) - 2\exp\!\left(-\frac{\|\bLambda\|_1^2/2}{mn\sigma^2+R\|\bLambda\|_1}\right),
		\end{align}
		where $\Phi_{\bLambda}$ is the row-alignment rate from \eqref{eq:Phi_def}.
	\end{theorem}

	The right-hand sides of both bounds tend to zero provided $\|\bLambda\|_1 \gg (m \vee n)^{3/2}$ up to logarithmic factors. Under this condition, $\eps_0 \le 1/(50 C_*^2)$ holds for sufficiently large dimension, and $\eps_{\textup{pot}} = 16 C_* \eps_0 \to 0$ forces both the spectral and the $L^1$ error bounds to vanish. In the balanced case $m \asymp n$, the condition simplifies to $\|\bLambda\|_1 \gg n^{3/2}$. The spectral-norm bound $\mathcal{E}_2$ will be used for the fluctuation analysis in Section \ref{sec:fluctuation_statements}, while the $L^1$ bound $\mathcal{E}_3$ will drive the scaling limit theory in Section \ref{sec:scaling_limit_statements}.
	
	The $L^1$ bound in $\mathcal{E}_3$ compares the Sinkhorn-rescaled matrix $\X^{\r,\c}$ with the comparison model $\hat{\X}^{\r,\c}$. To pass from $\hat{\X}^{\r,\c}$ to the deterministic bridge $\bLambda^{\r,\c}$---the object that governs the scaling limit---we must also absorb the concentration of $\X$ around $\bLambda$, which contributes an additional Bernstein-type error term. The next theorem carries out this step at the level of integration against bounded test functions on the unit square, yielding the weak-norm control of $\varphi_{\X^{\r,\c}} - \varphi_{\bLambda^{\r,\c}}$ that feeds directly into the scaling limit results of Section \ref{sec:scaling_limit_statements}.
	
	\begin{theorem}[Approximation against test functions]\label{thm:bound_X^rs-Z}
		Suppose Assumptions \ref{assump:random_matrix}--\ref{assump:bounded_cost} hold. For any bounded test function $g \in L^{\infty}([0,1]^2)$ and any $D > 0$, denote $\tau = (D+1)\log(m \vee n) + \log 4$, and let $\eps_{\textup{pot}}$ be as defined in Theorem \ref{thm:first_approximation_rescaled} evaluated at this $\tau$. Assume the dimensions are sufficiently large such that $\eps_{\textup{pot}} \le 1$. Then there exists an absolute constant $c>0$ such that with probability at least
		\[
		1 - 2(m \vee n)^{-D} - m^2 \exp(-c\,\Phi_{\bLambda}) - 2\exp\!\left(-\frac{\|\bLambda\|_1^2/2}{mn\sigma^2+R\|\bLambda\|_1}\right),
		\]
		the integration error against $g$ is bounded by
		\begin{align}
			\left|\int_{[0,1]^2} g\!\left(\varphi_{\X^{\r,\c}} - \varphi_{\bLambda^{\r,\c}}\right) dx\, dy\right| \le 2\|g\|_{\infty} \eps_{\textup{pot}} e^{2K} + \frac{\|g\|_{\infty} e^{2K}}{\|\bLambda\|_1} \left( \sigma\sqrt{2mn\tau} + 2R\tau \right),
		\end{align}
		where $\varphi_{\X^{\r,\c}}(x,y) = \frac{mn}{N}\X^{\r,\c}(\lceil mx \rceil, \lceil ny \rceil)$ and $\varphi_{\bLambda^{\r,\c}}(x,y) = \frac{mn}{N}\bLambda^{\r,\c}(\lceil mx \rceil, \lceil ny \rceil)$ are the piecewise constant extensions to the unit square.
	\end{theorem}
	
	Theorem \ref{thm:bound_X^rs-Z} controls the integration error $\int_{[0,1]^2} g\bigl(\varphi_{\X^{\r,\c}} - \varphi_{\bLambda^{\r,\c}}\bigr)\, dx\, dy$ against any bounded test function $g$. The right-hand side consists of two terms with distinct origins: the first, $2\|g\|_\infty \eps_{\textup{pot}} e^{2K}$, inherits the potential-stability error from Theorem \ref{thm:first_approximation_rescaled} and reflects how far the random Schr\"{o}dinger potentials depart from their deterministic counterparts; the second, of order $\|g\|_\infty e^{2K} \sqrt{mn\tau}/\|\bLambda\|_1$, arises from a Bernstein-type concentration of the centered matrix $\X - \bLambda$ tested against $g$. Under the mass-growth condition $\|\bLambda\|_1 \gg (m\vee n)^{3/2}$, both terms vanish, so $\varphi_{\X^{\r,\c}}$ and $\varphi_{\bLambda^{\r,\c}}$ agree against any bounded test function with high probability. This weak control of $\varphi_{\X^{\r,\c}} - \varphi_{\bLambda^{\r,\c}}$ against $L^\infty$ observables is the input used in Section \ref{sec:scaling_limit_statements} to transfer the deterministic scaling limit of $\bLambda^{\r,\c}$ (obtained via the total stability theorem) to the random matrix $\X^{\r,\c}$.
	
	%\commHL{Add CLT statement here or somewhere}

	\subsection{Scaling limit of Sinkhorn-rescaled random matrices}
	\label{sec:scaling_limit_statements}
	
	We now state the scaling limit results, showing that Sinkhorn-rescaled random matrices converge to the continuous static Schr\"{o}dinger bridge as dimensions grow. The proof proceeds in two stages: first, the deterministic bridges $\bLambda_k^{\r_k,\c_k}$ converge to the continuous SSB via the total stability estimate (Theorem \ref{thm:total_stability}); second, the random matrices $\X_k^{\r_k,\c_k}$ concentrate around $\bLambda_k^{\r_k,\c_k}$ at each dimension via the results in Section \ref{sec:concentration_statements}.
	
	To embed the discrete problems into a common continuous framework, we introduce the following convergence assumption.

	\begin{assumption}[Convergence of margins and kernels]\label{assumption: limit_theory}
		Fix a sequence of $m_k \times n_k$ $\delta$-smooth margins $(\r_k, \c_k)$ with $\sum_{i} \r_k(i) = \sum_{j} \c_k(j) = N_k$, together with prior mean matrices $\bLambda_k$.
		Define the probability measures on $[0, 1]$,  $d\mu_0^k = \rho_{\r_k}dx$ and $d\mu_1^k = \rho_{\c_k}dx$ where
		\begin{align}
			\rho_{\r_k}(x) = \frac{m_k}{N_k} \r_k(\lceil m_k x \rceil), \quad \rho_{\c_k}(x) = \frac{n_k}{N_k} \c_k(\lceil n_k x \rceil).
		\end{align}
		Additionally define the reference measure $d\mathcal{R}_k = \varphi_{\bLambda_k} dxdy$ where
		\begin{align}
			\varphi_{\bLambda_k}(x, y) = \frac{n_km_k}{N_k} \bLambda_k(\lceil m_k x \rceil, \lceil n_k y \rceil). 
		\end{align}
		Then there exist probability measures $\mu_0 = \rho_{\r} dx$, $\mu_1 = \rho_{\c} dx$, and $\mathcal{R} = \varphi_{\bLambda}dxdy$  
		\begin{align}
			\rho_{\r_k} \to \rho_{\r} \;\textup{in } L^1, \qquad \rho_{\c_k} \to \rho_{\c} \;\textup{in } L^1, \qquad \varphi_{\bLambda_k} \to \varphi_{\bLambda} \;\textup{in } L^1.
		\end{align}
		Moreover, Assumption \ref{assump:bounded_cost} holds uniformly along the sequence: there is a constant $K > 0$, independent of $k$, such that for all $k$ sufficiently large the triple $(\r_k, \c_k, \bLambda_k)$ satisfies \eqref{eq:bounded_cost_K} with the same $K$.
	\end{assumption}
	
	In terms of the histogram densities introduced above, the uniform bounded-cost condition \eqref{eq:bounded_cost_K} is equivalent to the uniform bound on the Radon--Nikodym derivative $\frac{d\mathcal{R}_k}{d(\mu_0^k \otimes \mu_1^k)}$:
	\begin{align}\label{eq:limit_theory_density_ratio}
		e^{-K} \leq \frac{\varphi_{\bLambda_k}}{\rho_{\r_k} \otimes \rho_{\c_k}} \leq e^K \qquad (\mu_0^k \otimes \mu_1^k)\textup{-a.e.}
	\end{align}
	This uniform bound on the density ratio 
	ensures that the cost functions remain uniformly bounded along the sequence.
	
	\begin{theorem}[Scaling limit of deterministic Schr\"{o}dinger bridges]\label{thm:SSB_deterministic_limit}
		Let Assumption \ref{assumption: limit_theory} hold. For each $k$, let
		\begin{align}
			\pi_k \in \argmin_{\pi \in \Pi(\mu_{0}^k, \mu_1^{k})} D_{\textup{KL}}(\pi \Vert \mathcal{R}_k)
		\end{align}
		be the static Schr\"{o}dinger bridge with margins $\mu_0^k, \mu_1^k$ and reference measure $\mathcal{R}_k$, and let $\pi^{\r,\c;\bLambda}$ denote the SSB with the limiting margins $\mu_0, \mu_1$ and reference measure $\mathcal{R}$. Then
		\begin{align}
			d_{H}(\pi_k, \pi^{\r, \c; \bLambda})^2 \leq 8(K + 4\log\delta^{-1}) \left( \|\rho_{\r_k} - \rho_{\r}\|_{L^1} + \|\rho_{\c_k} - \rho_{\c}\|_{L^1} \right) + 4\, e^{3K}\,\delta^{-12}\, d_{H}(\mathcal{R}_k, \mathcal{R})^2.
		\end{align}
		In particular, $d_{H}(\pi_k, \pi^{\r, \c; \bLambda}) \to 0$ as $k \to \infty$, so the sequence of discrete bridges converges to the continuous SSB in the Hellinger metric.
	\end{theorem}
	
	Combining Theorem \ref{thm:SSB_deterministic_limit} with the concentration results from Section \ref{sec:concentration_statements}, we obtain the full scaling limit for random matrices.
	
	\begin{theorem}[Scaling limit of Sinkhorn-rescaled random matrices] \label{thm:RM_rescaled_limit}
		Let $(\r_k, \c_k)$ and $\bLambda_k$ satisfy Assumption \ref{assumption: limit_theory}, and for each $k$ let $\X_k$ be a random matrix satisfying Assumption \ref{assump:random_matrix} with mean $\bLambda_k$ and target margin $(\r_k, \c_k)$. Let $\varphi_{\X_k^{\r_k,\c_k}}$ denote the piecewise constant extension of $\X_k^{\r_k,\c_k}$ to $[0,1]^2$ as in Theorem \ref{thm:bound_X^rs-Z}, and let $W^{\r,\c;\bLambda} := d\pi^{\r,\c;\bLambda}/(dx\,dy)$ be the Lebesgue density of the limiting SSB from Theorem \ref{thm:SSB_deterministic_limit}. Assume
		\begin{enumerate}[label=\textup{(\roman*)}]
			\item $\sup_k \sigma_k \le \sigma < \infty$ and $\sup_k R_k \le R < \infty$;
			\item $\liminf_k (m_k/n_k) \wedge (n_k/m_k) \ge \gamma > 0$;
			\item $(m_k \vee n_k)^{3/2} \sqrt{\log(m_k \vee n_k)} = o_{K,\delta,\gamma,\sigma,R}(\|\bLambda_k\|_1)$.
		\end{enumerate}
		Fix $D > 0$. Then for any $g \in L^{\infty}([0,1]^2)$,
		\begin{align}
			&\left|\int_{[0,1]^2} g\!\left(\varphi_{\X_k^{\r_k,\c_k}} - W^{\r,\c;\bLambda}\right) dx\, dy\right| \nonumber \\
			&\qquad = O_{K, \delta, \|g\|_{\infty}, \sigma, R}\!\left( \frac{(m_k \vee n_k)^{3/2} \sqrt{D\log(m_k \vee n_k)}}{\|\bLambda_k\|_1} + \sqrt{\norm{\rho_{\r} - \rho_{\r_k}}_{L^1} + \norm{\rho_{\c} - \rho_{\c_k}}_{L^1}} + d_H(\varphi_{\bLambda_k}, \varphi_{\bLambda}) \right)
		\end{align}
		with probability at least 
		\begin{align}
			1 - 2(m_k \vee n_k)^{-D} - m_k^2 \exp(-c\,\Phi_{\bLambda_k}) - 2\exp\left(-\frac{\|\bLambda_k\|_1^2}{2(m_k n_k \sigma^2 + R\|\bLambda_k\|_1)}\right)
		\end{align}
		for an absolute constant $c > 0$, where $\Phi_{\bLambda_k}$ is the row-alignment rate in \eqref{eq:Phi_def}. 
	\end{theorem}
	
	Conditions (i)--(iii) have clear interpretations: (i) fixes uniform sub-exponential noise parameters along the sequence, (ii) forces $m_k$ and $n_k$ to grow comparably, and (iii) is the mass-growth condition ensuring that the total expected mass $\|\bLambda_k\|_1$ outpaces the stochastic noise.
	
	The three terms on the right-hand side correspond precisely to the three distinct layers of approximation. The first term isolates the stochastic fluctuation $\X_k^{\r_k,\c_k} \approx \bLambda_k^{\r_k,\c_k}$ governed by the discrete concentration theory; it vanishes asymptotically because the total expected mass $\|\bLambda_k\|_1$ is assumed to strictly overpower the dimensional growth and variance parameters. The second and third terms isolate the deterministic macroscopic limits—the margin convergence $(\rho_{\r_k}, \rho_{\c_k}) \to (\rho_{\r}, \rho_{\c})$ and the %kernel
	reference convergence $\varphi_{\bLambda_k} \to \varphi_{\bLambda}$, respectively—governed by the continuous SSB stability theory. Driven jointly by the explicit expected mass growth condition and the function space convergence in Assumption \ref{assumption: limit_theory}, all three error bounds rigorously tend to zero as $k \to \infty$. This establishes the final scaling limit: the discrete Sinkhorn-rescaled random matrices converge to the continuous static Schr\"{o}dinger bridge.

	\subsection{Fluctuation of Sinkhorn-rescaled random matrices}
	\label{sec:fluctuation_statements}
	
	Having established the average behavior of the rescaled random matrix, we now characterize the fluctuations of the exact rescaled matrix $\X^{\r,\c}$ around the rescaled mean $\bLambda^{\r,\c}$. Specifically, we establish the bulk rigidity of the empirical spectral distribution for the associated sample covariance matrix.
	
	We first introduce the fluctuation matrices. Let $(\balpha, \bbeta)$ denote the deterministic Schr\"{o}dinger potentials for the mean matrix $\bLambda$, and let $(\balpha_{\X}, \bbeta_{\X})$ denote the (random) potentials scaling $\X$ to margin $(\r,\c)$, defined on the event $\{\X \in \mathcal{S}(\r,\c)\}$. Define the flatness parameter
	\begin{align}\label{eq:def_s_star}
		s_{\textup{max}} := \sup_{i\in[m],\, j\in[n]} e^{2\balpha(i)+2\bbeta(j)}\,\operatorname{Var}(\X_{ij}),
	\end{align}
	which is finite and positive under Assumption~\ref{assump:random_matrix} and Assumption~\ref{assump:bounded_cost} since $\operatorname{Var}(\X_{ij})>0$ and $e^{\balpha(i)+\bbeta(j)}>0$ for all $(i,j)$. 
	Define the \emph{rescaled fluctuation matrix} $\A$ and the \emph{comparison fluctuation matrix} $\check{\A}$ by
	\begin{align}\label{eq:def_check_A_and_A}
		\qquad \A := \frac{1}{\sqrt{(m+n)s_{\textup{max}}}}\, \D(e^{\balpha_{\X}})\,(\X - \bLambda)\,\D(e^{\bbeta_{\X}}), \quad \,\,  \check{\A} := \frac{1}{\sqrt{(m+n)s_{\textup{max}}}}\, \D(e^{\balpha})\,(\X - \bLambda)\,\D(e^{\bbeta}).
	\end{align}
	Both matrices lie in $\R^{m\times n}$. Our object of interest is the $m\times m$ Gram (sample covariance) matrix $\A\A^T$, which we will control by comparison with $\check{\A}\check{\A}^T$. The symmetric prefactor $\bigl((m+n)s_{\textup{max}}\bigr)^{-1/2}$ follows the convention of \cite[\S1.7.3]{lyu2024large} for inhomogeneous Wishart fluctuations: it is chosen so that the variance profile of $\check{\A}$ satisfies the flatness bound $\mathbf{S}_{ij}\le 1/(m+n)$ required by the Alt--Erd\H{o}s--Kr\"{u}ger local law for random Gram matrices \cite{alt2017local}, placing $\check{\A}\check{\A}^T$ directly in the scope of that theory. On the high-probability event $\mathcal{E}_2$ of Theorem \ref{thm:first_approximation_rescaled}, $\A$ and $\check{\A}$ are close in operator norm, and Weyl's inequality then transfers the rigidity and spectral confinement from $\check{\A}\check{\A}^T$ to $\A\A^T$.

	Before stating the bulk rigidity theorem, we characterize the deterministic approximation for the eigenvalue distribution of the comparison Gram matrix $\check{\A}\check{\A}^T$. The entries of $\check{\A}$ are independent and centered, with variance profile
	\begin{align}\label{eq:variance_profile_def}
		\mathbf{S}_{ij} := \operatorname{Var}(\check{\A}_{ij}) = \frac{1}{(m+n)\,s_{\textup{max}}}\, e^{2\balpha(i) + 2\bbeta(j)}\, \operatorname{Var}(\X_{ij}).
	\end{align}
	By definition of $s_{\textup{max}}$ in \eqref{eq:def_s_star}, this profile satisfies the uniform flatness bound $\mathbf{S}_{ij}\le 1/(m+n)$. Note that $\mathbf{S}$ is an $m \times n$ matrix of per-entry variances and should not be confused with the $m \times m$ Gram matrix $\check{\A}\check{\A}^T$.
	
	The deterministic approximation for the eigenvalue distribution of $\check{\A}\check{\A}^T$ is the probability measure $\nu$ on $\mathbb{R}_{\ge 0}$ determined by the Dyson equation associated with $\mathbf{S}$: for $z \in\mathbb{H}$,
	\begin{align}\label{eq:dyson_system}
		-\frac{1}{m_i(z)} = z - \sum_{k=1}^n \mathbf{S}_{ik} \frac{1}{1 + \sum_{j=1}^m \mathbf{S}_{jk} m_j(z)}, \qquad i = 1, \dots, m,
	\end{align}
	which admits a unique holomorphic solution $(m_i)_{i=1}^{m}:\mathbb{H}\to\mathbb{H}^m$, and the averaged Stieltjes transform $\bar m(z) := \tfrac{1}{m}\sum_{i=1}^m m_i(z)$ is the Stieltjes transform of $\nu$. Standard results on the Dyson equation \cite{alt2017local} further give the decomposition $d\nu = \pi_*\delta_0 + \pi(\omega)\mathbf{1}(\omega>0)\,d\omega$ with $\pi_*\in[0,1]$ a dimensionally-determined atom at zero and $\pi$ locally H\"older-continuous on $(0,\infty)$. The following theorem quantifies the finite-$n$ concentration of the eigenvalues of $\A\A^T$ around the classical locations $i(\tau) := \lceil m\int_0^{\tau}\nu(\omega)\,d\omega \rceil$ determined by $\nu$; Figure \ref{fig:ESD} illustrates the predicted limit shape against the empirical histogram. 
	
	\begin{theorem}[Bulk Rigidity of the Sample Covariance Matrix]\label{thm:ESD}
		Suppose Assumptions \ref{assump:random_matrix}--\ref{assump:bounded_cost} hold and the aspect ratio is bounded: $\gamma \le m/n \le \gamma^{-1}$ for some $\gamma \in (0,1]$.  Let $\xi := \min_{i,j,k,\ell} \frac{\Var(X_{ij})}{\Var(X_{k\ell})}>0$.  
		Let $\nu$ be the measure from the Dyson equation \eqref{eq:dyson_system} for the variance profile $\mathbf{S}$ in \eqref{eq:variance_profile_def}, with atom $\pi_*=\max\{0, 1-n/m\}$ and $\operatorname{supp}(\nu)\subseteq[0,4]$. 
		Fix $\eps,\Delta > 0$ and $\eps_* > 0$. Let $\eps_{\textup{cov}}$ be the covariance deviation error defined in Lemma \ref{lem:rescaled_approx} and let $\mathcal{E}_{2}$ be as in Theorem \ref{thm:first_approximation_rescaled}.
		For any $D > 0$, there exist constants $C_3=C_{3}(D,\eps, \Delta,\eps_*, \xi, K)$ and $C_4=C_{4}(D,\eps,\eps_*, \xi, K)$ such that the following hold:
		\begin{description}
			\item[(i)] $m \ne n$ (rectangular): 
			\begin{align}\label{eq:rigidity_rectangular}
				\P\left( \{ \X \in \mathcal{S}(\r,\c) \} \cap \left\{ \exists \tau \in (\Delta,4] : \pi(\tau) \ge \eps_*, |\lambda_{i(\tau)}({\A}{\A}^T) - \tau| \ge \frac{m^{\eps}}{m} + \eps_{\textup{cov}} \right\} \right) \le \P(\mathcal{E}_2^c) + \frac{C_3}{m^D}.
			\end{align}
			
			\item[(ii)] $m = n$ (square):
			\begin{align}\label{eq:rigidity_square}
				\P\left( \{ \X \in \mathcal{S}(\r,\c) \} \cap \left\{ \exists \tau \in (0,4] : \pi(\tau) \ge \eps_*, |\lambda_{i(\tau)}({\A}{\A}^T) - \tau| \ge \frac{m^{\eps}}{m}\left(\sqrt{\tau} + \frac{1}{m}\right) + \eps_{\textup{cov}} \right\} \right) \le \P(\mathcal{E}_2^c) + \frac{C_4}{m^D}.
			\end{align}
			
			\item[(iii)] For $m$ sufficiently large, there are asymptotically no eigenvalues far outside the support of $\nu$:
			\begin{align}\label{eq:no_eigenvalues_outside}
				\P\left( \{ \X \in \mathcal{S}(\r,\c) \} \cap \left\{ \sigma({\A}{\A}^T) \cap \left\{ \tau : \text{dist}(\tau, \text{supp}(\nu)) \ge \eps_* + \eps_{\textup{cov}} \right\} \ne \emptyset \right\} \right) \le \P(\mathcal{E}_2^c) + \frac{1}{m^D}.
			\end{align}
		\end{description}
	\end{theorem}

	\begin{figure}[h!]
		\centering
		\includegraphics[width=\linewidth]{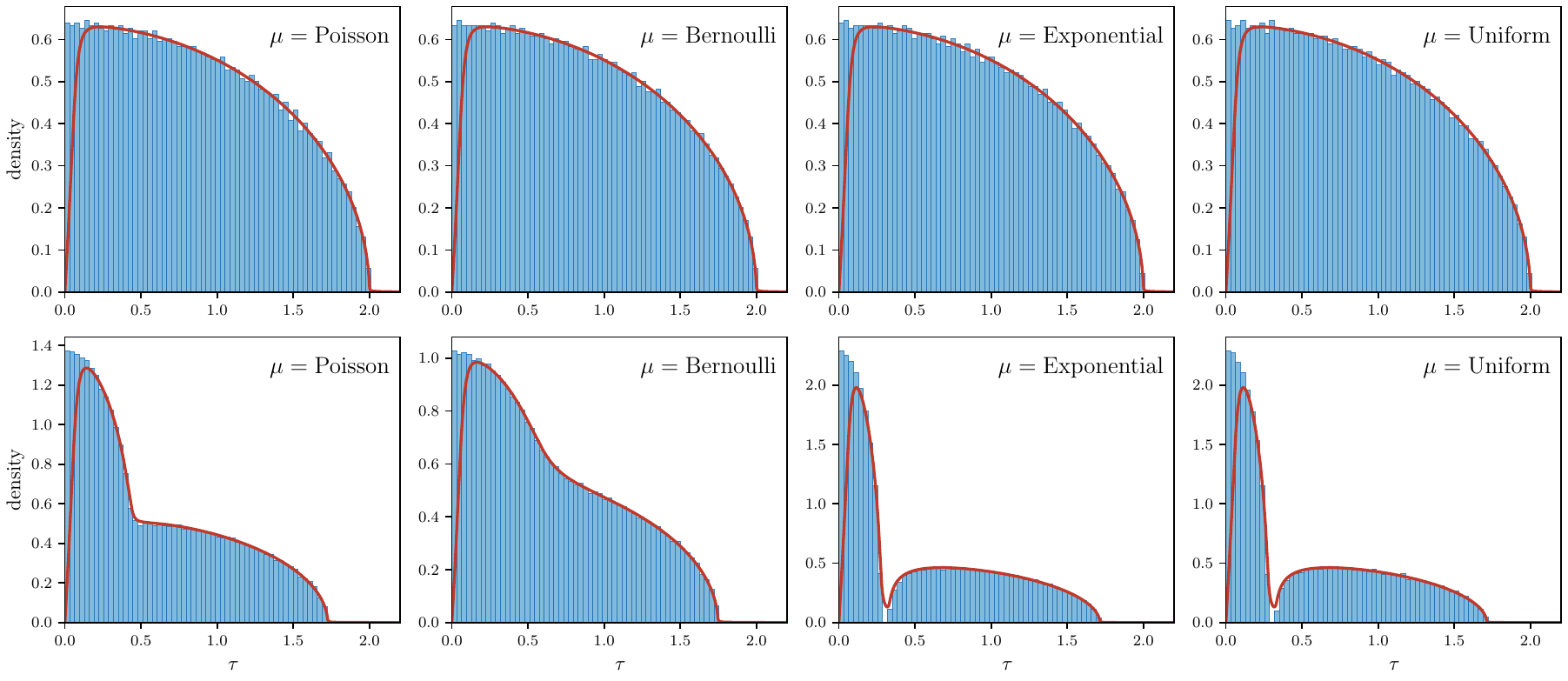}
		\caption{Empirical singular value distribution of the rescaled fluctuation matrix $\A$ at $m=n=5000$, for four entry distributions $\mu\in\{\mathrm{Poisson},\mathrm{Bernoulli},\mathrm{Exponential},\mathrm{Uniform}\}$. Top row: uniform margin $\r=\c=0.3n\cdot\mathbf{1}_n$ with homogeneous mean $\bLambda_{ij}\equiv 0.4$. Bottom row: row-block margin $\r_i\in\{0.1n, 0.5n\}$, uniform column margin $\c=0.3n\cdot\mathbf{1}_n$, row-block mean $\bLambda_{ij}\in\{0.2, 0.6\}$. Red curve: theoretical limit density of the singular values, obtained by solving the Dyson equation \eqref{eq:dyson_system} for the eigenvalue density of $\check{\A}\check{\A}^T$ and pushing forward under $\lambda\mapsto\sqrt{\lambda}$. See the discussion following Theorem \ref{thm:ESD} for the specialization to the Marchenko--Pastur quarter-circle in the homogeneous row and the non-trivial two-block shapes in the inhomogeneous row.}
		\label{fig:ESD}
	\end{figure}
	
	An important special case where the limiting ESD can be identified explicitly is the fully homogeneous setting: take $m=n$, $\r = \c = a\mathbf{1}_{n}$ for some $a>0$, $\bLambda_{ij} \equiv \lambda$ for some $\lambda>0$, and $\operatorname{Var}(\X_{ij}) \equiv \sigma^2$ for some $\sigma^2>0$ (e.g., $\X_{ij}$ i.i.d.\ Poisson$(\lambda)$). By symmetry the Schr\"{o}dinger potentials are then given by $\balpha = \bbeta = \tfrac{1}{2}\log\bigl(a/(n\lambda)\bigr)\mathbf{1}_{n}$, and together with the constant variance this makes the profile $\mathbf{S}_{ij}$ in \eqref{eq:variance_profile_def} uniformly equal to $1/n$. The solution to the Dyson equation \eqref{eq:dyson_system} then reduces to the Stieltjes transform of the Marchenko--Pastur law \cite{marchenko1967distribution}: the eigenvalue density of $\A\A^T$ converges to $\sqrt{4-\tau}/(2\pi\sqrt{\tau})$ on $(0,4]$, equivalently the singular value density of $\A$ converges to the quarter-circle density $\sqrt{4-\tau^2}/\pi$ on $[0,2]$. Theorem \ref{thm:ESD} specializes to bulk rigidity of $\A\A^T$ around this universal limit, independent of the entry distribution of $\X$. If instead $\bLambda$ is inhomogeneous, the variance profile $\mathbf{S}$ is no longer constant, the limiting density is no longer Marchenko--Pastur (equivalently, its singular value density is no longer the quarter-circle), and it depends genuinely on the entry distribution of $\X$ through $\operatorname{Var}(\X_{ij})$. Figure \ref{fig:ESD} illustrates both phenomena at the level of singular values.

Our last result concerns the asymptotic normality of the Schr\"{o}dinger potentials for rescaling random matrices to a given margin. Fix target margins $\r \in \R_{>0}^m, \c \in \R_{>0}^n$ and a strictly positive mean matrix $\bLambda \in \R_{>0}^{m \times n}$. Let $\{\X^{(\ell)}\}_{\ell =1}^{M}$ be a collection of $M$ independent and identically distributed random matrices satisfying Assumption \ref{assump:random_matrix} with mean $\E[\X^{(\ell)}] = \bLambda$. Define the empirical mean matrix:
\begin{align}
	\bar{\X}_M = \frac{1}{M} \sum_{\ell = 1}^{M} \X^{(\ell)}.
\end{align}
By the strong law of large numbers, $\bar{\X}_M \to \bLambda$ almost surely as $M \to \infty$. Since $\bLambda$ is strictly positive, $\bar{\X}_M \in \R_{>0}^{m \times n}$ with probability approaching one for sufficiently large $M$. Consequently, $\bar{\X}_M$ is exactly scalable to the target margins $(\r, \c)$ on this high-probability event.

Let $(\balpha, \bbeta)$ be the unique Sinkhorn scaling potentials for the mean matrix $\bLambda$ subject to the normalization condition $\langle \bbeta, \c \rangle = 0$. Similarly, let $(\balpha_{\bar{\X}_M}, \bbeta_{\bar{\X}_M})$ be the exact scaling potentials for the empirical mean $\bar{\X}_M$ satisfying the same normalization, such that:
\begin{align}
	\bar{\X}_M^{\textup{rs}} := \D(e^{\balpha_{\bar{\X}_M}}) \bar{\X}_M \D(e^{\bbeta_{\bar{\X}_M}}) 
\end{align}
has exact margins $(\r, \c)$. The next theorem establishes that the scaled fluctuations of these empirical potentials converge weakly to a multivariate Gaussian distribution.

\begin{theorem}[CLT for empirical Sinkhorn potentials] \label{thm: CLT_potentials_fixed_dim}
	Let $\X$ be a random matrix satisfying Assumption \ref{assump:random_matrix} with mean $\bLambda$ subject to a margin $(\r,\c)$ meeting Assumption \ref{assump:bounded_cost}, and let $\bSigma = \text{Cov}(\textup{vec}(\X))$ denote its covariance matrix. Let $\bar{\X}_M$ be the empirical mean of $M$ independent copies of $\X$. 
	
	Let $\bLambda^{\r,\c} = \D(e^{\balpha}) \bLambda \D(e^{\bbeta})$ denote the exact scaled mean matrix. Define the symmetric block matrix $\mathbf{L} \in \R^{(m+n) \times (m+n)}$ and the linear operator $\mathbf{H} \in \R^{(m+n) \times mn}$ by:
	\begin{align}
		\mathbf{L} := \begin{pmatrix}
			\D(\r) & \bLambda^{\r,\c} \\
			(\bLambda^{\r,\c})^\top & \D(\c)
		\end{pmatrix}, 
		\qquad 
		\mathbf{H} \, \textup{vec}(\Delta) := \begin{pmatrix}
			\big(\D(e^{\balpha}) \Delta \D(e^{\bbeta})\big) \mathbf{1}_n \\
			\big(\D(e^{\balpha}) \Delta \D(e^{\bbeta})\big)^\top \mathbf{1}_m
		\end{pmatrix}.
	\end{align}
	Let $\mathbf{L}^{\dagger}$ denote the Moore-Penrose pseudoinverse of $\mathbf{L}$. As $M \to \infty$, the scaled empirical potentials converge in distribution:
	\begin{align}
		\sqrt{M}\begin{pmatrix}
			\balpha_{\bar{\X}_M} - \balpha\\
			\bbeta_{\bar{\X}_M} - \bbeta
		\end{pmatrix}
		\xrightarrow{d}
		\mathcal{N}\left(0,  \mathbf{L}^{\dagger} \mathbf{H} \bSigma \mathbf{H}^{\top} (\mathbf{L}^{\dagger})^\top \right).
	\end{align}
\end{theorem}

\begin{remark}[Derivation of the Jacobian]
	The form of the covariance matrix follows from the Delta method, where the Jacobian of the map that takes an input matrix $\bLambda$ to the centered Schr\"{o}dinger potential for a fixed margin $(\r,\c)$ can be computed as  $J_{\bLambda} = -\mathbf{L}^{\dagger} \mathbf{H}$. Here, $\mathbf{L}$ is the partial derivative of the margin constraints with respect to the potentials, and $\mathbf{H}$ is the partial derivative with respect to the matrix entries. The pseudoinverse $\mathbf{L}^\dagger$ handles the translation invariance of the potentials within the subspace of normalized potentials. 
\end{remark}

\section{Background and related work}
\label{sec:related_work}

In this section, we give a detailed account of the background and related work. We organize the discussion into five threads: the Schrödinger bridge and matrix scaling (Section \ref{sec:rw_sb}), quantitative stability of entropic optimal transport (Section \ref{sec:rw_stability}), random contingency tables and matrices with given margins (Section \ref{sec:rw_contingency}), scaling of random matrices and biwhitening (Section \ref{sec:rw_biwhitening}), and the recent limit theory for conditioned random matrices (Section \ref{sec:rw_lyu_mukherjee}).

\subsection{Static Schrödinger bridge, matrix scaling, and Sinkhorn's algorithm}
\label{sec:rw_sb}

The static Schrödinger bridge problem \eqref{eq:RM_min_SB} traces back to Schrödinger \cite{schrodinger1931umkehrung}, who proposed it to model the most likely evolution of particle clouds. The modern treatment on general Polish spaces, including existence and duality, is covered in \cite{nutz2021introduction}. Its finite-dimensional counterpart is the matrix scaling problem: finding diagonal matrices $\D_1, \D_2$ such that $\D_1 \bLambda \D_2 \in \T(\r, \c)$. This problem is equivalent to a relative entropy minimization \eqref{eq:SSB_Poisson_main} \cite{deming1940least, ireland1968contingency}. 

The Sinkhorn algorithm \eqref{eq:sinkhon} computes these scaling factors via alternating renormalizations. Its convergence for strictly positive matrices is well-understood through the contraction properties of the Hilbert projective metric \cite{sinkhorn1967concerning, franklin1989scaling, chen2016entropic}. In the optimal transport community, the algorithm was popularized by Cuturi \cite{cuturi2013sinkhorn} as a regularization technique, sparking extensive research into the statistical properties of entropically regularized transport \cite{peyre2019computational}.

\subsection{Quantitative stability of entropic optimal transport}
\label{sec:rw_stability}

The stability of the Schrödinger bridge with respect to its defining data (margins and reference measure) has recently become a central topic in entropic optimal transport. We outline the closest related results and discuss how our bounds compare.

%\textit{Differential and qualitative approaches.} 
Carlier and Laborde \cite{carlier2020differential} established local invertibility of the Schrödinger system via an implicit-function-theorem argument, yielding Lipschitz continuity of the Schrödinger potentials with respect to $L^2$ and $L^\infty$ perturbations of the margins. Their result controls the potentials rather than the bridge itself, and the cost dependence of the Lipschitz constant is not made explicit. Ghosal et al.\ \cite{ghosal2022stability} and Nutz and Wiesel \cite{nutz2023stability} provided qualitative stability results under weak convergence, covering unbounded costs but without explicit rates.

%\textit{Quantitative stability via Wasserstein and total variation.}
Eckstein and Nutz \cite{eckstein2022quantitative} established quantitative stability bounds for the SSB; two of their results are directly comparable to ours.

\emph{(a) Reference-measure (cost) perturbation.} For \emph{fixed} marginals and bounded costs $c, \tilde{c}$, their Proposition~3.12 gives
\begin{align}\label{eq:EN_prop312}
	\|\pi^{*} - \tilde{\pi}^{*}\|_{\textup{TV}} \le \tfrac{1}{2}\,a^{1/(p+1)}\,\|c - \tilde{c}\|_{L^{p}(P)}^{p/(p+1)}, \qquad a = \exp(N\|c\|_\infty) + \exp(N\|\tilde{c}\|_\infty),
\end{align}
which is $\tfrac{p}{p+1}$-H\"older in the cost for $p<\infty$ (and degenerates to Lipschitz in the limiting case $p=\infty$). This is the direct analogue of our kernel-stability result (Theorem~\ref{thm:continuous_kernel_stability}).

\emph{(b) Marginal perturbation.} Their Theorems~3.11 and 3.13 bound $W_q(\pi^*, \tilde\pi^*)$ by a power of the Wasserstein distance between the marginals, with H\"older exponent $1/(2p)$ for general costs and $1/(p+1)$ for bounded costs (Lipschitz in the limit $p=\infty, q=1$). This is the analogue of our margin-stability result (Theorem~\ref{thm:continuous_margin_stability}).

%\textit{Comparison with our results.} 
Our stability theory differs from \cite{eckstein2022quantitative} in three respects.

(1) \textit{Hellinger control of densities.} Our scaling-limit theorem (Theorem~\ref{thm:RM_rescaled_limit}) requires control of SSB \emph{densities}, which Wasserstein bounds do not directly provide. Under the $\delta$-smoothness and bounded-cost assumptions, the discrete bridges $\pi_k$ in Theorem~\ref{thm:SSB_deterministic_limit} share a common support with the continuous SSB, and the Hellinger metric $d_H$---which is the $L^{2}$ distance between square-root densities and induces the same topology as total variation---is therefore the natural setting for stability. Its Hilbert-space structure is what enables the linear (Lipschitz, not H\"older) bound below via operator-theoretic arguments on square-root densities. By the comparison \eqref{eq:TV_Hellinger}, every bound we prove in $d_H$ also yields a TV bound at the same rate, up to an absolute factor.

(2) \textit{Lipschitz (not H\"older) kernel stability.} Theorem~\ref{thm:continuous_kernel_stability} shows that $d_H(\pi^{\mathcal{R}}, \pi^{\mathcal{R}'})$ is Lipschitz in $\|\kappa - \kappa'\|_{L^{2}(\mu_0\otimes\mu_1)}$, with constant $\exp\bigl(\tfrac{3}{2}\,\|\kappa^+\|_\infty \vee \|(\kappa')^+\|_\infty\bigr)$. Combined with the TV--Hellinger comparison \eqref{eq:TV_Hellinger}, as well as the mean value theorem estimate
\begin{align}
	d_{H}(\mathcal{R}, \mathcal{R}')\leq \frac{1}{2}e^{\tfrac{1}{2}\|\kappa^+\|_\infty \vee \|(\kappa')^+\|_\infty} \norm{\kappa - \kappa'}_{L^2(\mu_{0} \otimes \mu_1)}
\end{align}
this implies
\begin{align}\label{eq:our_TV_kernel}
	\|\pi^{\mathcal{R}} - \pi^{\mathcal{R}'}\|_{\textup{TV}} \;\le\; \sqrt{2}\, e^{2\|\kappa^+\|_\infty \vee \|(\kappa')^+\|_\infty}\,\|\kappa - \kappa'\|_{L^{2}(\mu_0\otimes\mu_1)},
\end{align}
which sharpens the H\"older exponent $p/(p+1) = 2/3$ in \eqref{eq:EN_prop312} to Lipschitz at $p=2$. The constant still grows exponentially in $\|\kappa^+\|_\infty$, as in \cite{eckstein2022quantitative}; Example~\ref{ex:kernel_stability_tightness} shows this dependence is unavoidable. The mechanism behind the improvement is the square-root density map
\begin{align}
	\Gamma : L^2(\mu_0 \otimes \mu_1) \to L^2(\mu_0 \otimes \mu_1), \qquad \sqrt{\frac{d\mathcal{R}}{d(\mu_0 \otimes \mu_1)}} \mapsto \sqrt{\frac{d\pi^{\mu_0,\mu_1;\mathcal{R}}}{d(\mu_0 \otimes \mu_1)}},
\end{align}
whose derivative factors into a multiplication operator and an orthogonal projection and is therefore bounded as a linear map. The argument is in spirit similar to the differential approach of \cite{carlier2020differential}, but the operator-theoretic perspective yields global (non-infinitesimal) bounds rather than a local-invertibility statement. By contrast, the sub-linear H\"older exponent in \eqref{eq:EN_prop312} arises from combining Pinsker's inequality with H\"older's inequality on $D_{\textup{KL}}$, which is unavoidable in a KL-based argument.

(3) \textit{Linear cost dependence for margin stability.} Theorem~\ref{thm:continuous_margin_stability} has merely \emph{linear} cost dependence for margin perturbations, milder than the exponential kernel case and of the same order as the Wasserstein bounds of \cite{eckstein2022quantitative} (Theorems~3.11, 3.13), though the metrics and proof techniques differ. Our argument is inspired by Lyu and Mukherjee \cite{lyu2024large}, who established Lipschitz continuity in $L^1$ of the typical table and the discrete Schr\"odinger potentials with respect to the margins for general $f$-divergences in place of the KL divergence but with uniform reference measure.

\subsection{Contingency tables and maximum entropy principle}
\label{sec:rw_contingency}

A \emph{contingency table} is an $m\times n$ matrix of non-negative integer entries with prescribed row sums $\r$ and column sums $\c$; we write $\mathrm{CT}(\r,\c)$ for the set of all such tables. Contingency tables are fundamental objects in statistics, combinatorics, and graph theory \cite{diaconis1995rectangular, diaconis1998algebraic, barvinok2009asymptotic}, and their exact counting and uniform sampling are computationally hard \cite{dyer1997sampling}, motivating a long line of approximate enumeration \cite{barvinok2009asymptotic, barvinok2010approximation, canfield2010asymptotic, lyu2022number} and sampling \cite{chen2005sequential, wang2020fast, besag1989generalized, diaconis1995rectangular}.

A unifying theme in this literature is the \textit{maximum entropy principle}. Classically \cite{good1963maximum}, the ``typical'' table with margins $(\r,\c)$ and total sum $N$ is the unique maximizer of the Shannon entropy $H(\X)=\sum_{ij}(\X_{ij}/N)\log(N/\X_{ij})$ on the transportation polytope $\T(\r,\c)$: the \textit{independence table} $\Y = (\r\otimes\c)/N$, whose $(i,j)$-entry is $\r(i)\c(j)/N$. This is the rank-one specialization of our discrete Schr\"{o}dinger bridge \eqref{eq:SSB_Poisson_main}: when the base matrix factorizes as $\bLambda = \mathbf{a}\otimes\mathbf{b}$, then for any $\X\in\T(\r,\c)$,
\begin{align*}
	D_{\mathrm{KL}}(\X\,\Vert\,\bLambda) = \sum_{ij}\X_{ij}\log\X_{ij} - \sum_i \r(i)\log\a(i) - \sum_j \c(j)\log\b(j),
\end{align*}
and the last two terms are constant on $\T(\r,\c)$; hence $\bLambda^{\r,\c} = \Y$ regardless of the factors $(\a,\b)$. The same matrix $\Y$ is produced by the Deming--Stephan iterative proportional fitting procedure \cite{deming1940least}, reinterpreted by Ireland and Kullback \cite{ireland1968contingency} as the $I$-projection onto the independence model---i.e.\ the Sinkhorn iteration.

Barvinok \cite{barvinok2010does, barvinok2010number} and Barvinok and Hartigan \cite{barvinok2010maximum, barvinok2012asymptotic, barvinok2013number} refined the maximum-entropy principle for the problem of \emph{counting} contingency tables. Their starting point is the probabilistic idea that a uniformly random element of $\mathrm{CT}(\r,\c)$ should be closely approximated by a deterministic matrix $\bZ \in \T(\r,\c)$---the \textit{typical table} for margin $(\r,\c)$---which maximizes the \emph{geometric-entropy functional}
\begin{align*}
	g(\X) = \sum_{ij} f(\X_{ij}), \qquad f(x) = (x+1)\log(x+1) - x\log x,
\end{align*}
where $f(x)$ is the entropy of the geometric distribution on $\mathbb{Z}_{\ge 0}$ with mean $x$. The rationale: the product measure that best approximates the uniform law on $\mathrm{CT}(\r,\c)$ should have expected margins $(\r,\c)$ while maximizing the total entropy of its entries under that first-moment constraint, and the geometric distribution is the maximum-entropy law on $\mathbb{Z}_{\ge 0}$ with a given mean. The resulting typical table $\bZ$ then satisfies a two-sided bound $g(\bZ) - O((m+n)\log N) \leq \log|\mathrm{CT}(\r,\c)| \leq g(\bZ)$ \cite{barvinok2009asymptotic,barvinok2012asymptotic}, sharpening Good's heuristic to a rigorous asymptotic count.

\subsection{Random matrices conditioned on margins}
\label{sec:rw_lyu_mukherjee}

Recently, Lyu and Mukherjee \cite{lyu2024large} extended the maximum-entropy principle from the uniform measure on $\mathrm{CT}(\r,\c)$ to random matrices $\X$ with i.i.d.\ entries $\X_{ij}\sim\mu$ conditioned on their margins $(\r,\c)$. They showed that the product measure best approximating such a conditional ensemble has independent but non-identically distributed entries, each drawn from the exponential family generated by $\mu$, with the law of $\X_{ij}$ determined by a tilt parameter of the form $\balpha(i)+\bbeta(j)$ depending only on its row and column. The row- and column-tilts $(\balpha,\bbeta)$ can be obtained either by relative-entropy minimization or by maximum-likelihood estimation, and they establish a Kantorovich-type duality between these two formulations. Moreover, $(\balpha,\bbeta)$ coincide with the Schr\"{o}dinger potentials of an $f$-divergence static Schr\"{o}dinger bridge with counting base measure, where the divergence $f$ is the KL divergence from $\mu$ to its exponential tilt.

Our discrete Schr\"{o}dinger bridge \eqref{eq:SSB_Poisson_main} sits inside this framework: $\bLambda^{\r,\c}$ is precisely the typical table for a Poisson random matrix with independent entries of means $\bLambda_{ij}$, conditioned on its margin being $(\r,\c)$. Interestingly, our main results show that a different but closely related random matrix model---obtained by Sinkhorn-\emph{rescaling}, rather than conditioning, an unconstrained random matrix to the target margin $(\r,\c)$---is also closely approximated by the same typical table $\bLambda^{\r,\c}$. The analytic mechanism behind our results is different: where \cite{lyu2024large} exploits the exponential-family / sufficient-statistic structure of the conditioned ensemble, our concentration and scaling-limit theorems are driven by a structural stability of the Sinkhorn mapping itself (Theorem~\ref{thm:potential_stability}).

\subsection{Scaling of random matrices and biwhitening}
\label{sec:rw_biwhitening}

A direction closer in methodology to the present paper concerns Sinkhorn rescaling of positive random matrices. Landa \cite{landa2022scaling} considers a random matrix $\X$ with strictly positive entries and mean $\bLambda = \E[\X]$, and establishes a concentration inequality for the random scaling factors of $\X$ around the deterministic scaling factors of $\bLambda$. Combined with standard vector and matrix concentration tools, this yields operator-norm convergence of the rescaled random matrix $\X^{\r,\c}$ to $\bLambda^{\r,\c}$ at rate $O(\sqrt{\log N/N})$ in the doubly-stochastic setting. A central deterministic ingredient is \cite[Lem.~9]{landa2022scaling}, an $\ell^{\infty}$ stability bound for the scaling factors under margin perturbations, whose proof extends Sinkhorn and Knopp's uniqueness argument \cite{sinkhorn1967concerning} by a direct manipulation of the scaling equations. Both the main theorem and this lemma assume that the entries of $\X$ lie in a bounded interval $[a,b]$ with $a>0$, and the resulting constants depend on the ratio $b/a$; in particular, they degenerate whenever any entry of $\bLambda$ vanishes.

In a closely related applied direction, Landa, Zhang, and Kluger \cite{landa2022biwhitening} use matrix scaling for rank estimation in corrupted count-data matrices arising in single-cell RNA sequencing. Their key observation is that, for Poisson (or more generally quadratic-variance) noise, one can scale the rows and columns of an observed data matrix so that the spectrum of the scaled \emph{noise} matrix matches the standard Marchenko--Pastur law on $[0,4]$; the Marchenko--Pastur upper edge then provides a principled threshold for rank selection. They refer to this procedure as \textit{biwhitening} and establish its asymptotic validity in a growing-mean regime $\min_{ij}\bLambda_{ij}\to\infty$, which ensures that the random scaling factors remain high-probability well-conditioned.

The present paper operates in a complementary regime along two axes. Analytically, our potential-stability theorem (Theorem~\ref{thm:potential_stability}) replaces the strict-positivity / bounded-ratio hypothesis of \cite{landa2022scaling} by a row-alignment condition expressed through the parameter $\rho_{\A}$, which measures the aggregate inner-product overlap between pairs of rows of $\A$. The stability constants depend on $\rho_{\A}$ rather than on the min-to-max ratio $b/a$, so the Sinkhorn mapping remains a well-conditioned local diffeomorphism even when many individual entries of $\bLambda$ vanish. Probabilistically, our concentration and scaling-limit theorems (Theorems~\ref{thm:concentration_margins_rs}, \ref{thm:RM_rescaled_limit}) hold for random matrices with $O(1)$ entrywise means and sparse supports, a regime excluded by the growing-mean hypothesis of \cite{landa2022biwhitening}. Finally, where biwhitening describes the spectrum of the \emph{fluctuation} of the scaled random matrix, our scaling-limit theorem targets the scaled \emph{mean} and identifies its limit as the continuous static Schr\"odinger bridge---a connection absent from the scaling / biwhitening literature.

\section{Proof of Stability for Static Schr\"{o}dinger Bridges}
\label{sec:SSB_stability}

In Section \ref{sec: sb_preliminaries}, we will briefly cover the background of the Schr\"{o}dinger bridge problem. Section \ref{sec: Kernel_stability} is for the stability under kernel perturbations and Section \ref{sec: margin_stability} covers stability with respect to margins.

\subsection{Preliminaries}
\label{sec: sb_preliminaries}

First, we record the following theorem, found as Theorem 2.1 of \cite{nutz2021introduction}, which gives mild sufficient conditions on the margins and base measure for $\pi^{\mu_0, \mu_1; \mathcal{R}}$ to uniquely exist. In this case, it has the special form
\begin{align}
	W^{\mu_0, \mu_1; \mathcal{R}} := \frac{d\pi^{\mu_0, \mu_1; \mathcal{R}}}{d\mathcal{R}} =  \exp(\balpha \oplus \bbeta)
\end{align}
for some functions $(\balpha, \bbeta) \in L^1(\mu_0) \times L^1(\mu_1)$.

\begin{theorem}[Theorem 2.1 of \cite{nutz2021introduction}]
	\label{thm: Schrodinger_bridge_existence_uniqueness}
	Suppose that $\mathcal{R} \sim \mu_0 \otimes \mu_1$ and that $D_{\textup{KL}}(\mu_{0} \otimes \mu_{1} || \mathcal{R}) < \infty$. Then the following hold.
	\begin{description}[leftmargin=0.6cm]
		\item[(i)] The static Schr\"{o}dinger bridge exists and is unique. Furthermore, there are $\balpha \in L^1(\mu_0)$ and $\bbeta \in L^1(\mu_1)$, unique up to transformations of the form $(\balpha, \bbeta) \mapsto (\balpha + \lambda, \bbeta - \lambda)$ for $\lambda \in \R$, such that
		\begin{align}
			\frac{d\pi^{\mu_0, \mu_1; \mathcal{R}}}{d\mathcal{R}} =  \exp(\balpha \oplus \bbeta) \quad \mathcal{R}  \text{-a.s.}
		\end{align}
		
		\item[(ii)] Conversely, if $\pi \in \Pi(\mu_0, \mu_1)$ admits a density 
		\begin{align}
			\frac{d\pi}{d\mathcal{R}} = \exp (\balpha \oplus \bbeta) \quad \mathcal{R} \text{-a.s.}
		\end{align}
		for some measurable functions $\balpha,\bbeta$, then $\pi$ is the Schr\"{o}dinger bridge. Moreover, $\balpha \in L^1(\mu_0)$ and $\bbeta \in L^1(\mu_1)$. 
	\end{description}
\end{theorem}

\begin{remark}[Continuous vs.\ discrete scalability conditions]
	In Theorem \ref{thm: Schrodinger_bridge_existence_uniqueness}, the requirement that the reference measure $\mathcal{R}$ is equivalent to $\mu_0 \otimes \mu_1$ is only a sufficient condition to guarantee the existence of potentials and the factorization $\mathcal{R}$-a.s., but is not necessary. Indeed, in the discrete instance of matrix scaling \eqref{eq:SSB_Poisson_main}, a necessary and sufficient condition for the existence of potentials (or scaling factors) is known due to Menon and Schneider \cite{menon1969spectrum} (see Lemma \ref{lem:scalability}).
	
	In the continuous entropic optimal transport (EOT) literature, it is common to assume that $\mathcal{R} \sim \mu_0 \otimes \mu_1$ in order to guarantee the existence, global uniqueness, and measurability of the scaling potentials (see, e.g., \cite[Theorem 2.1]{nutz2021introduction}). This strict equivalence is used to bypass pathological configurations where the optimal coupling $\pi_*$ is forced to drop support relative to $\mathcal{R}$; if such support dropping occurs, the standard exponential decomposition $\frac{d\pi_*}{d\mathcal{R}} = \exp(\balpha \oplus \bbeta)$ could fail $\mathcal{R}$-a.s.\ because the required potentials diverge. Nutz provides the following discrete counterexample:
	\begin{align}
		\X = \begin{bmatrix}
			1 & 1 \\
			0 & 1
		\end{bmatrix},
		\qquad \mu=\nu=(1/2, 1/2).
	\end{align}
	The optimal coupling must be the diagonal matrix $(1/2)\I$, so the potentials cannot be finite.
	
	However, this assumption is overly restrictive for the discrete setting, as it is equivalent to requiring $\X$ to have strictly positive entries everywhere. In the discrete regime, a necessary and sufficient condition for strict scalability—which guarantees the existence of finite, strictly positive diagonal scaling matrices and finite potentials—is given by the  combinatorial consistency condition of Menon and Schneider \cite[Thm. 4.1]{menon1969spectrum}. Nutz's counterexample above fails this condition.
\end{remark}

If $\mathcal{R} \in \mathcal{P}_{>0}(\Omega_{\r} \otimes \Omega_{\c}, \mu_0 \otimes \mu_1)$, then we may write $\frac{d\mathcal{R}}{d(\mu_0 \otimes \mu_1)} = e^{-\kappa}$ for some cost function $\kappa: \Omega_{\r} \times \Omega_{\c} \to \R$. In this case, it is well-known that the potentials $\balpha, \bbeta$ are the solutions to the dual problem
\begin{align}
	\label{eq: Schr\"{o}dinger_dual_cnts}
	\sup_{\xi \in L^1(\mu_0), \eta \in L^1(\mu_1)} \left( \int_{\Omega_{\r}} \xi d\mu_0 + \int_{\Omega_{\c}} \eta d\mu_1 - \iint_{\Omega_{\r} \times \Omega_{\c}} e^{\xi \oplus \eta - \kappa }d(\mu_0 \otimes \mu_1)  + 1 \right)
\end{align}
and satisfy the Schr\"{o}dinger equations

\begin{align}
	\label{eq: Schr\"{o}dinger_cnts}
	\balpha(x) = -\log \left( \int_{\Omega_{\c}} e^{\bbeta(y) - \kappa(x, y)} d\mu_1(y) \right), \quad \bbeta(y) = -\log \left( \int_{\Omega_{\r}} e^{\balpha(x) - \kappa(x, y)} d\mu_0(y) \right)
\end{align}

We will need the following lemma showing boundedness of the potentials. The proof follows an argument from the notes \cite{nutz2021introduction} and we include it for completeness. Here, and in the remainder of this section, we will use the shorthand notation
\begin{align}
	\langle \balpha, \mu_0 \rangle = \int_{\Omega_{\r}} \balpha d\mu_0, \quad \langle \bbeta, \mu_1 \rangle = \int_{\Omega_{\c}} \bbeta d\mu_1. 
\end{align}

\begin{lemma}
	\label{lem: boundedness_continuous_potentials}
	Make the same assumptions as in Theorem \ref{thm: Schrodinger_bridge_existence_uniqueness}. Further, assume $\frac{d\mathcal{R}}{d(\mu_0 \otimes \mu_1)} = e^{-\kappa}$ for some $\kappa \in L^{\infty}(\mu_0 \otimes \mu_1)$. Then
	\begin{align}
		-(\norm{\kappa^{-}}_{\infty} + \norm{\kappa^{+}}_{\infty})  \leq \balpha(x) \leq \norm{\kappa^{+}}_{\infty} \quad \mu_0 \textup{-a.s.}
	\end{align}
	and
	\begin{align}
		-(\norm{\kappa^{-}}_{\infty} + \norm{\kappa^{+}}_{\infty})  \leq \bbeta(y) \leq \norm{\kappa^{+}}_{\infty} \quad \mu_1 \textup{-a.s.}
	\end{align}
	where $\balpha, \bbeta$ are the unique Schr\"{o}dinger potentials satisfying $\langle \bbeta, \mu_1 \rangle = 0$.
\end{lemma}

\begin{proof}
	Applying Jensen's inequality to the first equation in \eqref{eq: Schr\"{o}dinger_cnts} we have for $\mu_0$ a.e. $x$
	\begin{align}
		\balpha(x) \leq \int_{\Omega_{\c}} -(\bbeta(y) - \kappa(x, y))  d\mu_1(y) \leq - \langle \bbeta, \mu_1 \rangle + \norm{\kappa^{+}}_{\infty} = \norm{\kappa^{+}}_{\infty}.
	\end{align}
	Reasoning similarly for $\bbeta$ we obtain
	\begin{align}
		\label{eq: pf_bounded_potentials_1}
		\bbeta(y) \leq - \langle \balpha, \mu_0 \rangle +  \norm{\kappa^{+}}_{\infty} 
	\end{align}
	for $\mu_1$ a.e. $y$. 
	Recall that $\balpha$ and $\bbeta$ are solutions to the dual problem \eqref{eq: Schr\"{o}dinger_dual_cnts} and that $\iint e^{\balpha \oplus \bbeta - \kappa} d(\mu_0 \otimes \mu_1) = \iint e^{-\kappa}d(\mu_0 \otimes \mu_1) = 1$. Then by taking $\eta = \xi = 0$ in \eqref{eq: Schr\"{o}dinger_dual_cnts} it follows
	\begin{align}
		\langle \balpha, \mu_0 \rangle &= \langle \balpha, \mu_0 \rangle + \langle \bbeta, \mu_1 \rangle\\
		&= \langle \balpha, \mu_0 \rangle + \langle \bbeta, \mu_1 \rangle - \iint_{\Omega_{\r} \times \Omega_{\c}} e^{\balpha \oplus \bbeta - \kappa }d(\mu_0 \otimes \mu_1) + 1\\
		& \geq \langle \xi, \mu_0 \rangle + \langle \eta, \mu_1 \rangle - \iint_{\Omega_{\r} \times \Omega_{\c}} e^{\xi \oplus \eta - \kappa }d(\mu_0 \otimes \mu_1) + 1 = 0.
	\end{align}
	Combining this with \eqref{eq: pf_bounded_potentials_1} we prove the upper bound. 
	
	For the lower bound, observe that for $\mu_0$-a.e.\ $x$,
	\begin{align}
		\balpha(x) &\geq -\log \left( \int_{\Omega_{\c}} e^{\textup{ess-sup}_{y} [\bbeta(y) - \kappa(x, y)]} d\mu_1(y) \right)\\
		&= \textup{ess-inf}_{y}[\kappa(x,y) - \bbeta(y)]\\
		& \geq -\norm{\kappa^{-}}_{\infty} - \textup{ess-sup}_{y} \bbeta(y) \geq -(\norm{\kappa^{-}}_{\infty} + \norm{\kappa^{+}}_{\infty}).
	\end{align}
	Similarly, one may show that
	\begin{align}
		\bbeta(y) \geq -(\norm{\kappa^{-}}_{\infty} + \norm{\kappa^{+}}_{\infty}).
	\end{align}
	This finishes the proof of the lower bound.
\end{proof}

\subsection{Stability with respect to reference measure}
\label{sec: Kernel_stability}

In this section, we prove Theorem \ref{thm:continuous_kernel_stability}. The techniques follow a similar idea from \cite{carlier2020differential} establishing local invertibility of the Schr\"{o}dinger system. However, since our primary goal is stability with respect to the reference measure, there are a number of important differences. 

Before the proof, we include include an example showing that the bounded cost assumption is necessary.

\begin{example}
	\label{ex:kernel_stability_tightness}
	Fix some $s \in (0, 1/4)$ which will be sent to zero. Let $\gamma(s)$ be a function with $\gamma(s) = o(s)$. We consider $\Omega_{\r} = \Omega_{\c} = \{0, 1\}$ with margins $\mu_{0} = \mu_{1} = \frac{1}{2}\delta_{0} + \frac{1}{2}\delta_{1}$.
	
	Consider the following reference measures (represented as matrices) 
	\begin{align}
		\mathcal{R} = \begin{bmatrix}
			1/4  & 1/2-(s +\gamma(s))\\
			1/4 & s + \gamma(s) 
		\end{bmatrix}, \quad 
		\mathcal{R}' = \begin{bmatrix}
			1/4 & 1/2-(s - \gamma(s)) \\
			1/4 & s - \gamma(s).
		\end{bmatrix} 
	\end{align}
	Put another way, we consider the problem of re-scaling matrices $\mathcal{R}$ and $\mathcal{R}'$ to have uniform row and column margins. Observe that the entry attaining the smallest ratio with respect to the product measure, and thus highest cost, is $\mathcal{R}'(2, 2)$. So, Theorem \ref{thm:continuous_kernel_stability} asserts that the rescaled matrices $\pi, \pi'$ satisfy
	\begin{align}
		d_{H}(\pi, \pi') \lesssim  \left(\frac{1}{s- \gamma(s)}\right)^{3/2}d_{H}(\mathcal{R}, \mathcal{R}') \lesssim s^{-3/2} d_{H}(\mathcal{R}, \mathcal{R}')
	\end{align} 
	for sufficiently small $s$. 
	We will show in this example that
	\begin{align}
		d_{H}(\pi, \pi') \gtrsim s^{-1/4}d_{H}(\mathcal{R}, \mathcal{R}')
	\end{align}
	implying that the bounded cost assumption of Theorem \ref{thm:continuous_kernel_stability} is necessary. Furthermore, it shows that Theorem \ref{thm:continuous_kernel_stability} is tight up to dependence on the coefficients of $\norm{\kappa^{+}}_{\infty}$ and $\norm{(\kappa')^{+}}_{\infty}$. 
	
	To see this, we start by computing $\pi$ and $\pi'$ explicitly. Note that they must be of the form
	\begin{align}
		\pi = \begin{bmatrix}
			a & \frac{1}{2}-a\\
			\frac{1}{2}-a & a
		\end{bmatrix}, \quad 
		\pi' = \begin{bmatrix}
			b & \frac{1}{2}-b\\
			\frac{1}{2}-b & b
		\end{bmatrix} 
	\end{align}
	for some $a, b \in [0, 1/2]$. 
	
	Let $(\balpha, \bbeta)$ be the potentials for $\mathcal{R}$, so that $\pi(i, j) = \mathcal{R}(i, j)e^{\balpha(i) + \bbeta(j)}$. Then, from the first row of $\pi$ we have
	\begin{align}
		\frac{a}{1/2- a} = \frac{\pi(1, 1)}{\pi(1, 2)} =  \frac{e^{\balpha(1) + \bbeta(1)}}{4(1/2 - s-\gamma(s))e^{\balpha(1)+\bbeta(2)}}
	\end{align}
	while the second row gives
	\begin{align}
		\frac{a}{1/2-a} = \frac{\pi(2, 2)}{\pi(2, 1)} =  4(s+\gamma(s))\frac{e^{\balpha(2) + \bbeta(2)}}{e^{\balpha(2)+\bbeta(1)}}.
	\end{align}
	Multiplying the two together we find
	\begin{align}
		\left(\frac{a}{1/2-a} \right)^2 = \frac{s+\gamma(s)}{1/2 - s - \gamma(s)}
	\end{align}
	or equivalently,
	\begin{align}
		a = \frac{\sqrt{s + \gamma(s)}}{\sqrt{1/2-s-\gamma(s)} + \sqrt{s+\gamma(s)}}.
	\end{align}
	Applying similar reasoning to compute the entries of $\pi'$, we find
	\begin{align}
		b = \frac{\sqrt{s - \gamma(s)}}{\sqrt{1/2 - s + \gamma(s)} + \sqrt{s - \gamma(s)}}.
	\end{align}
	
	Let $f(x) = \frac{x^{1/4}}{(\sqrt{1/2 - x} + \sqrt{x})^{1/2}}$ so we may write
	\begin{align}
		d_{H}(\pi, \pi')^2 \geq (\sqrt{a} - \sqrt{b})^2 = \left(f(s+\gamma(s)) - f(s - \gamma(s)) \right)^2.
	\end{align}
	Calculus gives
	\begin{align}
		f'(x) = \frac{1}{8x^{3/4}\sqrt{1/2-x}(\sqrt{1/2-x}+\sqrt{x})^{3/2}}.
	\end{align}
	Thus, applying the mean value theorem to the difference $f(s+\gamma(s)) - f(s - \gamma(s))$ while observing that $\sqrt{1/2 - x}(\sqrt{1/2-x} + \sqrt{x})^{3/2} \to 1/4$ as $x \to 0$, we have
	\begin{align}
		d_{H}(\pi, \pi')^2 \gtrsim \frac{\gamma(s)^2}{s^{3/2}}
	\end{align}
	for $s$ sufficiently small. On the other hand, a similar Taylor expansion gives
	\begin{align}
		d_{H}(\mathcal{R}, \mathcal{R}')^2 = \left(\sqrt{s + \gamma(s)} - \sqrt{s-\gamma(s)}\right)^2 + \left(\sqrt{1/2 - s - \gamma(s)} - \sqrt{1/2 - s + \gamma(s)} \right)^2 \lesssim \frac{\gamma(s)^2}{s}
	\end{align}
	implying that
	\begin{align}
		d_{H}(\pi, \pi') \gtrsim s^{-1/4}d_{H}(\mathcal{R}, \mathcal{R}').
	\end{align}
	as desired. \hfill $\blacktriangle$ 
\end{example}

Let us define $\overline{L}^{\infty}(\mu_1)$ to be the set of all bounded functions $\eta$ with $\langle \eta, \mu_1 \rangle = 0$, which is a Banach space under the inherited norm. The main idea of the proof is as follows. As in \cite{carlier2020differential}, we consider $L^{\infty}_{>0}(\mu_0 \otimes \mu_1)$, the interior of the positive cone of $L^{\infty}(\mu_0 \otimes \mu_1)$ which contains all probability densities with respect to $\mu_0 \otimes \mu_1$ with bounded cost. 

We will use the implicit function theorem to identify construct a map \begin{align}
	B : L^{\infty}_{>0}(\mu_0 \otimes \mu_1) \to L^{\infty}(\mu_0) \times \overline{L}^{\infty}(\mu_1)
\end{align}
such that for any probability density $\phi$, $B(\phi)$ are the Schr\"{o}dinger potentials for reference measure $d\mathcal{R} = \phi d(\mu_0 \otimes \mu_1)$ and margins $\mu_0, \mu_1$. Then, define
\begin{align}
	\label{eq: def_Psi}
	&\Psi : L^{\infty}_{>0}(\mu_0 \otimes \mu_1) \times L^{\infty}(\mu_0) \times \overline{L}^{\infty}(\mu_1) \to L^{\infty}(\mu_0 \otimes \mu_1)\\
	& (\phi, \xi, \eta) \mapsto \phi e^{\xi \oplus \eta}
\end{align}
so that 
\begin{align}
	\Psi(\phi, B(\phi)) = W := \frac{d\pi^{\mu_0, \mu_1; \mathcal{R}}}{d(\mu_0 \otimes \mu_1)}.
\end{align}

To construct $B$ we first set
\begin{align}
	\mathcal{E} = \left\{ (f_{0}, f_{1}) \in L^{\infty}(\mu_0) \times L^{\infty}(\mu_1) : \langle f_{0}, \mu_0 \rangle = \langle f_{1}, \mu_1 \rangle \right\}.
\end{align}
We aim to apply to the implicit function theorem to
\begin{align}
	F : L^{\infty}(\mu_0 \otimes \mu_1) \times L^{\infty}(\mu_0) \times \overline{L}^{\infty}(\mu_1) \to \mathcal{E}
\end{align}
given by
\begin{align}
	F(\phi, \xi, \eta) = \left( \int_{\Omega_{\c}} \phi(\cdot, y)e^{\xi(\cdot) + \eta(y)}d\mu_1(y) - 1, \int_{\Omega_{\r}}\phi(x, \cdot)e^{\xi(x) + \eta(\cdot)}d\mu_0(x) - 1 \right).
\end{align}
From the Schr\"{o}dinger equations \eqref{eq: Schr\"{o}dinger_cnts}, if $\phi$ is a probability density (or at least a positive integrable function), then $F(\phi, \xi, \eta) = 0$ implies that $\phi e^{\xi \oplus \eta}$ is the Schr\"{o}dinger bridge for reference measure $\phi$

\begin{lemma}
	\label{lem: implicit_function_justification}
	Fix $\phi \in L^{\infty}_{>0}(\mu_0 \otimes \mu_1)$. Also, let $(\xi, \eta) \in L^{\infty}(\mu_0) \times \overline{L}^{\infty}(\mu_1)$. Then $F$ is continuously (Frechet) differentiable at $(\phi, \xi, \eta)$. Furthermore, let $\partial_1$ be the partial derivative in the first coordinate and $\partial_{2}$ be the partial derivative with respect to $F$'s second two coordinates. The derivatives are the linear operators
	\begin{align}
		\label{eq: kernel_to_margins_derivative}
		\partial_1F(\phi, \xi, \eta)\psi = \left( \int_{\Omega_{\c}} \psi(\cdot, y)e^{\xi(\cdot) + \eta(y)}d\mu_1(y), \int_{\Omega_{\r}}\psi(x, \cdot) e^{\xi(x) + \eta(\cdot)}d\mu_0(x) \right). 
	\end{align}
	and for $h = (h_0, h_1) \in L^{\infty}(\mu_0) \times \overline{L}^{\infty}(\mu_1)$, 
	\begin{multline}
		\label{eq: potentials_to_margins_derivative}
		\partial_{2}F(\phi, \xi, \eta)h\\
		= \left(\int_{\Omega_c}\phi(\cdot, y)e^{\xi(\cdot) + \eta(y)}(h_{0}(\cdot) + h_{1}(y))d\mu_1(y), \int_{\Omega_{\r}}\phi(x, \cdot) e^{\xi(x) + \eta(\cdot)}(h_{0}(x) + h_{1}(\cdot))d\mu_0(x) \right).
	\end{multline}
	Furthermore, $\partial_2F(\phi, \xi, \eta)$ is an isomorphism from $L^{\infty}(\mu_0) \times \overline{L}^{\infty}(\mu_1)$ to $\mathcal{E}$. 
\end{lemma}

\begin{proof}
	Since $\phi$ is integrable and $\xi$ and $\eta$ are bounded, it is straightforward to see that $\partial_1F(\phi, \xi, \eta)$ and  $\partial_{2}F(\phi, \xi, \eta )$ are bounded as linear operators from $L^1(\mu_0 \otimes \mu_1) \to \mathcal{E}$ and $L^{\infty}(\mu_0) \times \overline{L}^{\infty}(\mu_1) \to \mathcal{E}$ respectively.
	To show $\partial_1 F + \partial_2 F$ is the derivative, we first observe that from Fubini's theorem and the triangle inequality, the $L^1$ norm of both coordinates of the difference
	\begin{align}
		F(\phi + \psi, \xi + h_{0}, \eta + h_{1}) - F(\phi, \xi, \eta) - \partial_1F(\phi, \xi, \eta)\psi -  \partial_{2}F(\phi, \xi, \eta)h
	\end{align}
	are bounded above by
	\begin{multline}
		\label{eq: pf_Schr\"{o}dinger_invertibility_1}
		\int_{\Omega_{\r} \times \Omega_{\c}}|\psi||e^{(\xi + h_{0}) \oplus (\eta + h_{1})} - e^{\xi \oplus \eta}|d(\mu_0 \otimes \mu_1)\\
		+ \int_{\Omega_{\r} \times \Omega_{\c}}|\phi||e^{(\xi + h_{0}) \oplus (\eta + h_{1})} - e^{\xi \oplus \eta} - e^{\xi \oplus \eta}(h_{0} \oplus h_{1})| d(\mu_0 \otimes \mu_1).
	\end{multline}
	
	Equip $L^1(\mu_0 \otimes \mu_1) \times L^{\infty}(\mu_0) \times \overline{L}^{\infty}(\mu_1)$ with the product norm
	\begin{align}
		\norm{(\psi, h_{0}, h_{1})} := \norm{\psi}_1 + \norm{h_{0}}_{\infty} + \norm{h_{1}}_{\infty}.
	\end{align}
	Assume that $(\psi, h_{0}, h_{1})$ is a sequence with $\norm{(\psi, h_{0}, h_{1})} \to 0$, where we omit the index for clarity. By a pointwise application of the mean value theorem, there exists a sequence of functions $f$ (also not indexed), with $\norm{f - \xi \oplus \eta}_{\infty} \leq \norm{h_{0}}_{\infty} + \norm{h_{1}}_{\infty}$ such that 
	\begin{align}
		e^{(\xi + h_{0}) \oplus (\eta + h_{1})} - e^{\xi \oplus \eta} = e^{f}(h_{0} \oplus h_{1}).
	\end{align}
	In particular, the sequence $f$ is uniformly bounded and converges to $\xi \oplus \eta$ in $L^{\infty}(\mu_0 \otimes \mu_1)$ as $\norm{h_{0}}_{\infty} + \norm{h_{1}}_{\infty} \to 0$. 
	Hence, may rewrite the integrals in  \eqref{eq: pf_Schr\"{o}dinger_invertibility_1} and continue as
	\begin{align}
		&\int_{\Omega_{\r} \times \Omega_{\c}}|\psi|e^{f}|h_{0} \oplus h_{1}|d(\mu_0 \otimes \mu_1) + \int_{\Omega_{\r} \times \Omega_{\c}} |\phi||e^{f} - e^{\xi \oplus \eta}||h_{0} \oplus h_{1}| d(\mu_0 \otimes \mu_1)\\
		& \leq \norm{\psi}_{1}\norm{e^{f}}_{\infty}(\norm{h_{0}}_{\infty} + \norm{h_{1}}_{\infty}) + (\norm{h_{0}}_{\infty} + \norm{h_{1}}_{\infty})\int_{\Omega_{\r} \times \Omega_{\c}} |\phi||e^{f} - e^{\xi \oplus \eta}|d(\mu_0 \otimes \mu_1)\\
		& = o(\norm{(\psi, h_{0}, h_{1})})
	\end{align}
	where the last equality follows from an application of the dominated convergence theorem to the second integral. 
	
	This shows that $F$ is differentiable. Then, using the assumption that $\phi$ is bounded above and away from zero, it is not hard to see that the claim that $\partial_2F(\phi, \xi, \eta)$ is an isomorphism is equivalent to Proposition 3.1 of \cite{carlier2020differential}.
\end{proof}

From the implicit function theorem, if $\mathcal{R} \in \mathcal{P}_{>0}(\Omega_{\r} \times \Omega_{\c}, \mu_0 \otimes \mu_1)$ has density $\phi$ satisfying the conditions of Lemma \ref{lem: implicit_function_justification} and $F(\phi, \xi, \eta) = 0$, then there is a map $B$ defined in a neighborhood of $\phi$ with $F(\phi, B(\varphi)) = 0$. That is, $B(\phi)$ are the Schr\"{o}dinger potentials for reference measure $\mathcal{R}$ and margins $\mu_0, \mu_1$. The derivative of $B$ is locally defined by the identity
\begin{align}
	\partial B(\psi) = - \partial_2F(\psi, B(\psi))^{-1}\partial_1F(\psi, B(\psi)). 
\end{align}
Finally, by Lemma \ref{lem: boundedness_continuous_potentials}, the image of $L^{\infty}_{>0}(\mu_0 \otimes \mu_1)$ under $B$ is contained in $L^{\infty}(\mu_0) \times \overline{L}^{\infty}(\mu_1)$

Recall the definition of $\Psi$ from \eqref{eq: def_Psi} and set $\Phi(\psi) = \Psi(\psi, B(\psi))$.
Finally, consider
\[
\Gamma: \begin{tikzcd}
	\sqrt{\phi} \arrow[r] & \phi \arrow[r,"\Phi"] & W \arrow[r] & \sqrt{W}
\end{tikzcd}
\]
where the square root is defined pointwise.
The main idea will be to use the mean value theorem, and bound the $L^2 \to L^2$ operator norm of $\partial \Gamma$ along the line connecting $\sqrt{\phi}$ and $\sqrt{\phi'}$ for given probability densities $\phi, \phi'$.

\begin{remark}
	\label{rem: extension_of_B}
	Note that we do not require  $\varphi$ to be a probability density in Lemma \ref{lem: implicit_function_justification}. If $\int \phi \neq 1$ then the potentials that solve $F(\phi, \xi, \eta) = 0$ are simply a re-scaling of those that solve $F(\bar{\phi}, \xi, \eta) = 0$ where $\bar{\phi}$ is the normalization of $\phi$ to a probability density. 
	
	Additionally, for a fixed $\phi$, the implicit function theorem only implies the existence of a locally defined map $B$. However, observe that if both $\varphi$ and $\varphi'$ are bounded above and below, then so is $\phi_t := (t\sqrt{\phi} + (1-t)\sqrt{\phi'})^2$ with the bounds holding uniformly for $t \in [0, 1]$. So we may apply the implicit function theorem at every point along the line. Then, using continuity of $t \mapsto \phi_t$ by a standard compactness argument, we may extend the definition of $B$ to the entire set $(\varphi_{t})_{t \in [0, 1]}$ to apply the mean value theorem. %then $\varphi_t$ is not guaranteed to be the density of a probability measure. At first glance, this may seem to be an issue since the map $B$, according to Lemma \ref{lem: implicit_function_justification}, was originally constructed with domain $\mathcal{P}_{>0}(\Omega, \rho)$. However, this turns out only to be a matter of scaling.
	
	%Indeed, if $\norm{\varphi_t}_{1} \neq 1$, then we can consider its normalization $\bar{\varphi}_t$ and associated potentials $(\bar{\balpha}, \bar{\bbeta})$. If we set $(\balpha, \bbeta) = (\bar{\balpha} - \log \norm{\varphi_t}_1, \bar{\bbeta})$, then $\phi_t e^{\balpha \oplus \bbeta}$ is a probability density with the required margins. In this way, we may extend $B$ to the positive subset of $L^{\infty}(\rho)$.
\end{remark}

Here, we introduce some notation that will be used in the rest of the proof. For a function $f \in L^{\infty}(\mu_0 \otimes \mu_1)$, $\text{Diag}(f): L^2(\mu_0 \otimes \mu_1) \to L^2(\mu_0 \otimes \mu_1)$ is the operator $\text{Diag}(f)g(x) = f(x)g(x)$.  For $(\xi, \eta) \in L^{\infty}(\mu_0) \times \overline{L}^{\infty}(\mu_1)$ and $\phi \in L^{\infty}(\mu_0 \otimes \mu_1)$, we set 
\begin{align}
	\label{eq: def_XI_Upsilon}
	\Xi(\phi, \xi, \eta) = \text{Diag}(\phi e^{\xi \oplus \eta})
	\text{ and } \Upsilon(\xi, \eta) = \text{Diag}(e^{\xi \oplus \eta}).
\end{align}
We also use the notation $(\balpha_{\phi}, \bbeta_{\phi}) := B(\varphi)$. Lemma \ref{lem: Gamma_derivative_bound} is the key step in proving Theorem \ref{thm:continuous_kernel_stability}. 

\begin{lemma}
	\label{lem: Gamma_derivative_bound}
	Let $\phi, \phi' \in L^{\infty}_{>0}(\mu_0 \otimes \mu_1)$ be probability densities with cost functions $\kappa$ and $\kappa'$. For $t \in [0, 1]$, let
	\begin{align}
		\sqrt{\phi_t} = t\sqrt{\phi} + (1-t)\sqrt{\phi'}. 
	\end{align} 
	There exists $g_t \in L^{\infty}(\mu_0 \otimes \mu_1)$, with $\norm{g_t}_{\infty} \leq e^{\tfrac{3}{2}\norm{\kappa^{+}}_{\infty} \vee \norm{(\kappa')^{+}}_{\infty}}$, and an orthogonal projection $P_t: L^2(\mu_0 \otimes \mu_1) \to L^2(\mu_0 \otimes \mu_1)$ so that 
	\begin{align}
		\partial \Gamma(\sqrt{\varphi_t}) = P_t\text{Diag}(g_t).
	\end{align} 
	Consequently, 
	\begin{align}
		\norm{\partial \Gamma(\sqrt{\varphi_t})f}_{L^2(\mu_0 \otimes \mu_1)} \leq e^{\tfrac{3}{2}(\norm{\kappa^{+}}_{\infty} \vee \norm{(\kappa')^{+}}_{\infty})}\norm{f}_{L^2(\mu_0 \otimes \mu_1)}.
	\end{align}
	for any $f \in L^{\infty}(\mu_0 \otimes \mu_1)$. 
\end{lemma}

Before proving Lemma \ref{lem: Gamma_derivative_bound}, we record a few intermediate observations. First, the derivatives of some of the operators introduced above admit some helpful decompositions. 

\begin{lemma}
	\label{lem: operator_factorizations}
	Let $H: L^2(\mu_1) \times L^2(\mu_0) \to L^2(\mu_0 \otimes \mu_1)$, be the linear map $(u, v) \mapsto u \oplus v$, whose adjoint $H^{\top} : L^2(\mu_0 \otimes \mu_1) \to L^2(\mu_0) \times L^2(\mu_1)$ is given by
	\begin{align}
		H^{\top} \psi = \left( \int_{\Omega_{\c}} \psi(\cdot, y)d\mu_1(y), \int_{\Omega_{\r}} \psi(x, \cdot) d\mu_0(x) \right). 
	\end{align}
	Let $\psi \in L^{\infty}_{>0}(\mu_0 \otimes \mu_1)$. Then
	\begin{itemize}
		\item[(i)]  $\partial_2 F(\psi, B(\psi)) = H^{\top}\Xi(\psi, \balpha_{\psi}, \bbeta_{\psi})H$
		\item[(ii)] $\partial_1 F(\psi, B(\psi)) = H^{\top} \Upsilon(\balpha_{\psi}, \bbeta_{\psi})$
		\item[(iii)] $\partial_2 \Psi(\psi, B(\psi)) = \Xi(\psi, \balpha_{\psi}, \bbeta_{\psi})H$. 
	\end{itemize}
\end{lemma}

\begin{proof}
	Items (i) and (ii) follow immediately from the form of the derivatives given in \eqref{eq: kernel_to_margins_derivative} and \eqref{eq: potentials_to_margins_derivative}. 
	
	For (iii), it is straightforward to see that for $h = (h_{0}, h_{1})$, 
	\begin{align}
		\partial_2 \Psi(\psi, B(\psi))h = \psi e^{\balpha_{\psi} \oplus \bbeta_{\psi}}(h_{1} \oplus h_{0})
	\end{align}
	which is precisely the result of applying $\Xi(\psi, \balpha_{\psi}, \bbeta_{\psi})H$ to $h$. 
\end{proof}

\begin{lemma}
	\label{lem: jacobian_factorization}
	Let $\psi \in L^1(\mu_0 \otimes \mu_1)$ satisfy the assumptions of Lemma \ref{lem: implicit_function_justification}. Then for $\Xi = \Xi(\psi, \balpha_{\psi}, \bbeta_{\psi})$ and $\Upsilon = \Upsilon(\balpha_{\psi}, \bbeta_{\psi})$ we have
	\begin{align}
		\partial_2\Psi(\psi, B(\psi))\partial B(\psi) = -\Xi^{1/2}O\Xi^{-1/2}\Upsilon
	\end{align}
	where $O = O(\psi, \balpha_{\psi}, \bbeta_{\psi}) : L^{2}(\mu_0 \otimes \mu_1) \to L^{2}(\mu_0 \otimes \mu_1)$ is an orthogonal projection. 
\end{lemma}

\begin{proof}
	Recall that $\partial B(\psi) = -\partial_2F(\psi, B(\psi))^{-1}\partial_1F(\psi, B(\psi))$. Using the factorizations from Lemma \ref{lem: operator_factorizations} we can write
	\begin{align}
		\partial_2\Psi(\psi, B(\psi))\partial B(\psi) &= -\partial_2\Psi(\psi, B(\psi)) \partial_2F(\psi, B(\psi))^{-1}\partial_1F(\psi, B(\psi))\\
		&= -\Xi H (H^{\top}\Xi H)^{-1}H^{\top}\Upsilon\\
		&= -\Xi H(H^{\top} \Xi H)^{-1}H^{\top} \Xi^{1/2} \Xi^{-1/2}\Upsilon
	\end{align}
	Set $A = \Xi^{1/2}H$. Noting that $\Xi^{1/2}$ is self adjoint, the above implies
	\begin{align}
		\partial_2\Psi(\varphi, B(\varphi))\partial B(\varphi) =  -\Xi^{1/2}A(A^{\top}A)^{-1}A^{\top}\Xi^{-1/2}\Upsilon. 
	\end{align}
	Let $O = A (A^{\top}A)^{-1}A^{\top}$. We observe that $O$ is well defined and bounded on $L^{\infty}(\mu_0 \otimes \mu_1)$ and thus extends as a bounded linear operator to $L^2(\mu_0 \otimes \mu_1)$. Furthermore, $O$ is idempotent and self-adjoint with respect to the $L^2$ inner product. This implies that $O$ is an orthogonal projection and finishes the proof. 
\end{proof}

We are now ready to prove Lemma \ref{lem: Gamma_derivative_bound}.

\begin{proof}[\textbf{Proof of Lemma \ref{lem: Gamma_derivative_bound}}]
	Let $\phi_t$ be as in the statement of the Lemma. We will need an upper bound on $e^{\balpha_{\varphi_t} \oplus \bbeta_{\varphi_t}}$.
	
	First, let $\kappa_t$ be the cost of $\phi_t$, i.e. $\kappa_t = -\log \phi_t$. Since $x \mapsto -\log(x^2)$ is convex,
	\begin{align}
		-\log \phi_t = -\log \left( \left(t \sqrt{\phi} + (1-t)\sqrt{\phi'}\right)^2 \right) &\leq -t\log(\phi) - (1-t)\log(\phi')\\
		&= t\kappa + (1-t)\kappa'\\
		& \leq \kappa \vee \kappa'.
	\end{align}
	Since $\phi_t$ is only guaranteed to be a sub-probability measure, we first consider its normalization $\bar{\phi}_t = \frac{\phi_t}{\norm{\phi_t}_1}$ and its potentials, $(\bar{\balpha}_t, \bar{\bbeta}_t) = B(\bar{\phi}_t)$. Let $\bar{\kappa}_t$ be such that $\bar{\phi}_t = e^{-\bar{\kappa}_t}$. If $\kappa_t$ is the cost of $\phi_t$ then
	\begin{align}
		\bar{\kappa}_t = \kappa_t + \log \norm{\phi_t}_1 \leq \kappa_t
	\end{align}
	since Jensen's inequality implies $\norm{\phi_t}_1 \leq 1$. Since $\bar{\phi}_t$ is a probability density, we can apply Lemma \ref{lem: boundedness_continuous_potentials} to get
	\begin{align}
		\bar{\balpha}_t \oplus \bar{\bbeta}_t \leq 2\norm{\bar{\kappa}^{+}_t}_{\infty} \leq 2\norm{\kappa_t^{+}}_{\infty} \leq 2(\norm{\kappa^{+}}_{\infty} \vee \norm{(\kappa')^{+}}_{\infty}).
	\end{align}
	By simple rescaling, the potentials for $\phi_t$ are given by
	\begin{align}
		(\balpha_{\varphi_t}, \bbeta_{\varphi_t}) = (\bar{\balpha}_t - \log \norm{\varphi_t}_1, \bar{\bbeta}_t).
	\end{align}
	Now, using Jensen's inequality again
	\begin{align}
		-\log \norm{\phi_t}_1 = -\log \left( \int_{\Omega_{\r} \times \Omega_{\c}} \phi_t d(\mu_0 \otimes \mu_1) \right) &= -\log \left( \int_{\Omega_{\r} \times \Omega_{\c}} e^{-\kappa_t} d(\mu_0 \otimes \mu_1) \right)\\
		& \leq \int_{\Omega_{\r} \times \Omega_{\c}} \kappa_t \,d(\mu_0 \otimes \mu_1)\\
		& \leq \norm{\kappa_t^{+}}_{\infty}.
		%& \leq \norm{\kappa^{+}}_{\infty} \vee \norm{(\kappa')^{+}}_{\infty}.
	\end{align}
	Thus, it follows
	\begin{align*}
		\label{eq: dGamma_decomp_pf_1}
		e^{\balpha_{\phi_t} \oplus \bbeta_{\phi_t}} \le e^{3(\norm{\kappa^{+}}_{\infty} \vee \norm{(\kappa')^{+}}_{\infty})}.
	\end{align*}
	
	Now observe that, by the chain rule
	\begin{align}
		\partial \Gamma(\sqrt{\phi_t}) = \textup{Diag}(\Phi(\phi_t)^{-1/2
		})\partial\Phi(\phi_t)\textup{Diag}(\sqrt{\phi_t})
		)
	\end{align}
	Furthermore, $\partial \Phi$ decomposes as
	\begin{align}
		\partial \Phi(\phi_t) &= \partial_1 \Psi(\phi_t, B(\phi_t)) + \partial_2 \Psi(\phi_t, B(\phi_t))\partial B(\phi_t)\\
		&= \Upsilon_t - \Xi_t^{1/2}O_t\Xi_t^{-1/2}\Upsilon_t.
	\end{align}
	where we use the notation of Lemma \ref{lem: jacobian_factorization}.
	We observe that, from the definition of $\Phi$
	\begin{align}
		\textup{Diag}(\Phi(\phi_t)^{-1/2}) = \Xi_t^{-1/2}.
	\end{align}
	Therefore, 
	\begin{align}
		\partial \Gamma(\sqrt{\phi_t}) = \Xi_t^{-1/2} \Upsilon_t \textup{Diag}(\sqrt{\phi_t}) - O_t \Xi_t^{-1/2} \Upsilon_t \textup{Diag}(\sqrt{\phi_t}) = (I - O_t) \textup{Diag}(g_t)
	\end{align}
	where $g_t = e^{\tfrac{1}{2}\balpha_{\phi_t} \oplus \bbeta_{\phi_t}}$. Since $O_t$ is an orthogonal projection, so is $I - O_t$.  From \eqref{eq: dGamma_decomp_pf_1}, $\norm{g_t}_{\infty} \leq e^{\tfrac{3}{2}(\norm{\kappa^{+}}_{\infty} \vee \norm{(\kappa')^{+}}_{\infty}}$. The proof is complete. 
\end{proof}

\begin{proof}[\textbf{Proof of Theorem \ref{thm:continuous_kernel_stability}}]
	Let
	\begin{align}
		d\pi^{\mu_0, \mu_1; \mathcal{R}} = Wd(\mu_0 \otimes \mu_1), \quad d\pi^{\mu_0, \mu_1; \mathcal{R}'} = W'd(\mu_0 \otimes \mu_1)
	\end{align}
	By Lemma \ref{lem: Gamma_derivative_bound} and the mean value theorem
	\begin{align}
		d_{H}(\pi^{\mu_0, \mu_1; \mathcal{R}}, \pi^{\mu_0, \mu_1; \mathcal{R}'}) &= \norm{\sqrt{W} - \sqrt{W'}}_{L^2(\mu_0 \otimes \mu_1)}\\
		&\leq \sup_{t \in [0, 1]} \norm{\partial \Gamma(\sqrt{\phi_t})(\sqrt{\phi} - \sqrt{\phi'})}_{L^2(\mu_0 \otimes \mu_1)}\\
		& \leq e^{\tfrac{3}{2}\norm{\kappa^{+}}_{\infty} \vee \norm{(\kappa')^{+}}_{\infty}}\norm{(\sqrt{\phi} - \sqrt{\phi'})}_{L^2(\mu_0 \otimes \mu_1)}\\
		&= e^{\tfrac{3}{2}\norm{\kappa^{+}}_{\infty} \vee \norm{(\kappa')^{+}}_{\infty}}d_{H}(\mathcal{R}, \mathcal{R}').
	\end{align}
	This shows the assertion. 
\end{proof}

\subsection{Stability with respect to margins}
\label{sec: margin_stability}

Recall that in this section, we suppose that all marginals have density with respect to some common reference measures $\mu$ and $\nu$. 

We will recycle some notation from the previous section. %Let $\overline{L}^{\infty}(\nu)$ be the set of functions $f \in L^{\infty}(\nu)$ with $\langle f,\nu \rangle = 0$. 
Since we are considering stability with respect to perturbations of the margins only, we consider a fixed reference measure $\mathcal{R} \in \mathcal{P}_{>0}(\Omega_{\r} \times \Omega_{\c}, \mu \otimes \nu)$ whose density with respect to the fixed base measure $\mu \otimes \nu$ will be denoted $\varphi$. Define the map
\begin{align}
	\Psi : L^{\infty}(\mu) \times L^{\infty}(\nu) \to L^{\infty}(\mu \otimes \nu), \qquad (\xi, \eta) \mapsto \varphi e^{\xi \oplus \eta}.
\end{align}
Thus, if $\balpha, \bbeta$ are the Schrodinger potentials for Schrodinger bridge $\pi$, $\Psi(\balpha, \bbeta)$ returns $\frac{d\pi}{d(\mu \otimes \nu)}$. Now let $F(\xi, \eta)$ be the margins of $\varphi e^{\xi \oplus \eta}$. That is,
\begin{align}
	F: L^{\infty}(\mu) \times L^{\infty}(\nu) \to L^{\infty}(\mu) \times L^{\infty}(\nu), \qquad (\xi, \eta) \mapsto H^{\top}\Psi(\xi, \eta).
\end{align}
Here $H$ is the same operator from Lemma \ref{lem: operator_factorizations}, but now viewed as acting on $L^2(\mu) \times L^2(\nu)$.

By calculations similar to the previous section, we observe that 
\begin{align}
	\partial \Psi(\xi, \eta) = \Xi(\varphi, \xi, \eta) H.
\end{align}
Thus, from the chain rule, it follows 
\begin{align}
	\partial F(\xi, \eta) = H^{\top}\Xi(\varphi, \xi, \eta)H.
\end{align}

\begin{proof}[\textbf{Proof of Theorem \ref{thm:continuous_margin_stability}}]
	Let us observe that, since we assume $\varphi > 0$, the map
	\begin{align}
		\Delta : L^{\infty}(\mu) \times L^{\infty}(\nu) \to L^{\infty}(\mu \otimes \nu),\qquad (\xi, \eta) \mapsto \sqrt{\varphi e^{\xi \oplus \eta}}
	\end{align}
	is $C^1$ with derivative
	\begin{align}
		\partial \Delta(\xi, \eta) = \frac{1}{2} \textup{Diag}(\Psi(\xi, \eta))^{-1/2} \partial \Psi(\xi, \eta)  = \frac{1}{2}\Xi(\varphi, \xi, \eta)^{1/2}H,
	\end{align}
	which extends as a bounded linear operator
	
	Let $(\balpha, \bbeta), (\balpha', \bbeta') \in L^{\infty}(\mu) \times L^{\infty}(\nu)$ be the unique potentials for margins $\mu_0, \mu_1$ and $\mu_0, \mu_1'$ respectively satisfying the normalization condition
	\begin{align}
		\langle \bbeta, \mu_1 \rangle = \langle \bbeta', \mu_1' \rangle = 0.
	\end{align}
	
	Set $(\balpha_t, \bbeta_t) = t(\balpha, \bbeta) + (1-t)(\balpha', \bbeta')$ and note that, under the assumptions of the theorem, Lemma \ref{lem: boundedness_continuous_potentials} implies $\sup_{t \in [0, 1]} \norm{\balpha_t \oplus \bbeta_t}_{\infty} < \infty$. Thus, from the fundamental theorem of calculus we have
	\begin{align}
		\sqrt{\frac{d\pi^{\mu_0, \mu_1; \mathcal{R}}}{d(\mu \otimes \nu)}} - \sqrt{\frac{d\pi^{\mu_0', \mu_1'; \mathcal{R}}}{d(\mu \otimes \nu)}} &= \int_0^1 \partial \Delta(\balpha_t, \bbeta_t)((\balpha, \bbeta) - (\balpha', \bbeta'))dt\\
		&= \int_0^1 \Xi(\varphi, \balpha_t, \bbeta_t)^{1/2}H((\balpha, \bbeta) - (\balpha', \bbeta'))dt
	\end{align}
	where the integral is understood in the $L^2$ sense. Recall that the densities of the margins $\mu_0$ and $\mu_1$ are denoted $\rho_{\r}$ and $\rho_{\c}$, respectively. We have
	\begin{align}
		\label{eq: pf_margin_stability_1}
		(\rho_{\r}, \rho_{\c}) - (\rho_{\r}', \rho_{\c}') &= F(\balpha, \bbeta) - F(\balpha', \bbeta')\\
		&= \int_0^1H^{\top} \Xi(\varphi, \balpha_t, \bbeta_t)H((\balpha, \bbeta) - (\balpha', \bbeta')) dt
	\end{align}
	where the integral is understood in $L^2(\mu) \times L^2(\nu)$.  
	
	Now we observe by Jensen's inequality 
	\begin{align}
		\left \lVert \sqrt{\frac{d\pi^{\mu_0,\mu_1; \mathcal{R}}}{d(\mu_0 \otimes \mu_1)}} - \sqrt{\frac{d\pi^{\mu_0',\mu_1'; \mathcal{R}}}{d(\mu_0 \otimes \mu_1)}} \right \rVert_{L^2(\mu_0 \otimes \mu_1)}^2 &= \left \lVert \int_0^1 \Xi(\varphi, \balpha_t, \bbeta_t)^{1/2}H((\balpha, \bbeta) - (\balpha', \bbeta'))dt \right \rVert_{L^2(\mu \otimes \nu)}^2\\
		& \leq \int_0^1 \norm{\Xi(\varphi, \balpha_t, \bbeta_t)^{1/2}H((\balpha, \bbeta) - (\balpha', \bbeta'))}_{L^2(\mu \otimes \nu)}^2 dt.
	\end{align}
	Observe that expanding the norm inside the integral and using the definitions of $\Xi$ and $H$, 
	\begin{align}
		&\hspace{-1cm} \int_0^1 \norm{\Xi(\varphi, \balpha_t, \bbeta_t)^{1/2}H((\balpha, \bbeta) - (\balpha', \bbeta'))}_{L^2(\mu \otimes \nu)}^2 dt\\
		&= \int_0^1 \Big\langle \Xi(\varphi, \balpha_t, \bbeta_t)^{1/2}H((\balpha, \bbeta) - (\balpha', \bbeta')), \Xi(\varphi, \balpha_t, \bbeta_t)^{1/2}H((\balpha, \bbeta) - (\balpha', \bbeta')) \Big\rangle_{L^2(\mu \otimes \nu)}\\
		&= \int_0^1 \Big\langle H^{\top}\Xi(\varphi, \balpha_t, \bbeta_t)H\bigl((\balpha, \bbeta) - (\balpha', \bbeta')\bigr),\, (\balpha, \bbeta) - (\balpha', \bbeta')\Big\rangle_{L^{2}(\mu) \times L^{2}(\nu)}\,dt \\
		&= \left\langle\int_0^1 H^{\top}\Xi(\varphi, \balpha_t, \bbeta_t)H\bigl((\balpha, \bbeta) - (\balpha', \bbeta')\bigr)\,dt,\, (\balpha, \bbeta) - (\balpha', \bbeta')\right\rangle_{L^{2}(\mu) \times L^{2}(\nu)}\\
		&= \langle \rho_{\r} - \rho_{\r}',\, \balpha - \balpha' \rangle_{L^{2}(\mu)} + \langle \rho_{\c} - \rho_{\c}',\, \bbeta - \bbeta' \rangle_{L^{2}(\nu)}\\
		& \leq \norm{\rho_{\r} - \rho_{\r}'}_{L^1(\mu)}\norm{\balpha - \balpha'}_{L^{\infty}(\mu)} + \norm{\rho_{\c} - \rho_{\c}'}_{L^1(\nu)}\norm{\bbeta - \bbeta'}_{L^{\infty}(\nu)}, 
	\end{align}
	where the third and the fourth equalities  are due to Fubini's theorem and  \eqref{eq: pf_margin_stability_1}, respectively. 
	
	% The assumption $\varphi \in \mathcal{C}_{M}$ is equivalently stated as
	% \begin{align}
		%     \norm{\kappa}_{L^{\infty}(\mu \otimes \nu)}, \norm{\kappa'}_{L^{\infty}(\mu \otimes \nu)} \leq \log M
		% \end{align}
	It then follows from Lemma \ref{lem: boundedness_continuous_potentials} and the triangle inequality that
	\begin{align}
		\norm{\balpha - \balpha'}_{L^{\infty}(\mu)} \leq 4 \norm{\kappa}_{L^{\infty}(\mu)} \vee \norm{\kappa'}_{L^{\infty}(\mu)}.
	\end{align}
	% where $\kappa, \kappa'$ are cost functions satisfying
	% \begin{align}
		%     \frac{d\varphi}{d\rho} = e^{-\kappa}, \quad \frac{d\varphi}{d\rho'} = e^{-\kappa'}.
		% \end{align}
	Since the same inequality holds for $\norm{\bbeta - \bbeta'}_{L^{\infty}(\nu)}$ the proof is complete. 
\end{proof}

\begin{proof}[\textbf{Proof of Theorem \ref{thm:total_stability}}]
	From the triangle inequality and Cauchy-Schwartz
	\begin{align}
		d_{H}(\pi^{\mu_0, \mu_1; \mathcal{R}}, \pi^{\mu_0', \mu_1'; \mathcal{R}'})^2 \leq 2d_{H}(\pi^{\mu_0, \mu_1; \mathcal{R}}, \pi^{\mu_0', \mu_1'; \mathcal{R}})^2 + 2d_{H}(\pi^{\mu_0', \mu_1'; \mathcal{R}}, \pi^{\mu_0', \mu_1'; \mathcal{R}'})^2.
	\end{align}
	Due to the assumption $e^{-M} \leq \frac{d(\mu_0 \otimes \mu_1)}{d(\mu_0' \otimes \mu_1')} \leq e^{M}$,   it follows from the idenity
	\begin{align}
		\frac{d\mathcal{R}}{d(\mu_0' \otimes \mu_1')} = \frac{d\mathcal{R}}{d(\mu_0 \otimes \mu_1)} \frac{d(\mu_0 \otimes \mu_1)}{d(\mu_0' \otimes \mu_1')},
	\end{align}
	the cost of $\frac{d\mathcal{R}}{d(\mu_0' \otimes \mu_1')}$ has the bound
	\begin{align}
		\left \lVert \log \frac{d\mathcal{R}}{d(\mu_0' \otimes \mu_1')} \right \rVert_{\infty} \leq \norm{\kappa}_{\infty} + M.
	\end{align}
	The proof follows by applying Theorem \ref{thm:continuous_margin_stability} to the first term on the right hand side and Theorem \ref{thm:continuous_kernel_stability} to the second.
\end{proof}

Next, we prove Theorem \ref{thm:potential_stability}. We establish the existence of exact discrete potentials by linearizing the Sinkhorn scaling mapping $F$ at the origin. We use the row alignment condition to bound the spectral gap of an associated Markov chain, ensuring the Jacobian is well-conditioned in the $\ell_\infty$ norm. A unique solution is then found via the Banach fixed-point theorem within an $\ell_\infty$ ball of radius $4C_{\mathbf{A}}\epsilon$.

\begin{proof}[\textbf{Proof of Theorem \ref{thm:potential_stability}}]
	Define the scaling function $F : \R^m \times \R^n \to \R^{m+n}$ by:
	\[
	F(\balpha, \bbeta) = \begin{bmatrix}
		(e^{\balpha\oplus\bbeta}\odot \A) \mathbf{1}_n - \r \\
		(e^{\balpha\oplus\bbeta}\odot \A)^T \mathbf{1}_m - \c
	\end{bmatrix},
	\]
	where $(e^{\balpha\oplus\bbeta})_{ij} = e^{\balpha_i + \bbeta_j}$. A root $\xi_{\A} = [\balpha_{\A}, \bbeta_{\A}]^T$ of $F(\xi) = \mathbf{0}$ corresponds to the exact scaling potentials. Note that $F(0) = [\r(\A)-\r, \c(\A)-\c]^T$. 
	
	We establish the existence of $\xi_{\A}$ by linearizing $F$ at the origin. Expanding $F$ at zero gives $F(\xi) = F(0) + DF|_{(0,0)}\xi + R(\xi)$. Let $D_r = \diag(\r(\A))$, $D_c = \diag(\c(\A))$, and $D = \diag(D_r, D_c)$. The Jacobian at the origin is:
	\[
	DF|_{(0,0)} = \begin{bmatrix}
		D_r & \A \\
		\A^T & D_c
	\end{bmatrix} = D \begin{bmatrix}
		I & D_r^{-1}\A \\
		D_c^{-1}\A^T & I
	\end{bmatrix} := D \begin{bmatrix}
		I & P_{12} \\
		P_{21} & I
	\end{bmatrix} := D L.
	\]
	Notice that $P_{12}$ and $P_{21}$ are row-stochastic matrices. We aim to show that the linear operator $L$ has a well-behaved pseudo-inverse with respect to the $\ell_{\infty}$ norm when applied to $D^{-1}F(0)$ and $D^{-1}R(\xi)$. Consider the consistent linear system $L[x_1, x_2]^T = [b_1, b_2]^T$. Block substitution yields:
	\begin{align}\label{eq:prop_schur}
		(I - P_{12}P_{21})x_1 = b_1 - P_{12}b_2, \qquad x_2 = b_2 - P_{21}x_1.
	\end{align}
	Denote $P_1 = P_{12}P_{21}$. This is the transition matrix of a reversible, aperiodic Markov chain on the rows of $\A$, with stationary distribution $\pi_1 = \r(\A)/\|\r(\A)\|_1$. 
	
	Since $\epsilon \le 1/2$, the current margins are strictly bounded by $\r_i(\A) \le \frac{3}{2}\r_i$ and $\c_j(\A) \le \frac{3}{2}\c_j$. Applying the row alignment definition for $\rho_{\A}$, we lower bound the transition probabilities of $P_1$: 
	\[
	P_1(i_1, i_2) = \frac{1}{\r_{i_1}(\A)}\sum_{j=1}^{n}\frac{\A_{i_1 j}\A_{i_2 j}}{\c_j(\A)} \ge \frac{\rho_{\A} n}{\left(\frac{3}{2}\max_i \r_i\right)\left(\frac{3}{2}\max_j \c_j\right)} = \frac{4 \rho_{\A} n}{9 \max_i \r_i \max_j \c_j} := \mu > 0.
	\]
	Because every entry of $P_1$ is strictly bounded below by $\mu$, the transition matrix satisfies a uniform minorization condition. Consequently, the total variation distance between any two rows is tightly bounded by $1 - m\mu$. Defining the spectral gap parameter $\nu_* = m\mu > 0$, we have $\nu_* = \frac{4 \rho_{\A} m n}{9 \max_i \r_i \max_j \c_j}$. Let $\Pi_1 = \mathbf{1}\pi_1^T$. For any vector $v$ orthogonal to the stationary distribution (i.e., $\pi_1^T v = 0$), the geometric decay implies
	\[
	\|P_1^tv\|_{\infty}\le \text{sp}(P_1^tv)\le (1-\nu_*)^t\text{sp}(v)\le 2(1-\nu_*)^t\|v\|_{\infty},
	\]
	where $\text{sp}(v)$ is the difference between the maximum and minimum entry of $v.$

	Since the target margins and current margins both sum to $N$, the projection of $b_1 - P_{12}b_2$ onto $\pi_1$ vanishes.
	When $[b_1,b_2]^{\top}=D^{-1}F(0)$, $\pi_1^{\top}b_1=\pi_1^{\top}P_{12}b_2=\frac{\|F(0)\|_1}{\|\A\|_1}$, and when $[b_1,b_2]^{\top}=D^{-1}R(\xi)$, $\pi_1^{\top}b_1=\pi_1^{\top}P_{12}b_2=\frac{\|F(\xi)\|_1-\langle \xi,[\r(\A),\c(\A)]^{\top}\rangle}{\|\A\|_1}$. In both cases, the projection of $b_1 - P_{12}b_2$ onto $\pi_1$ vanishes. Thus, we can invert $(I - P_1)$ on this subspace to solve \eqref{eq:prop_schur} with an absolutely convergent series:
	\[
	x_1 = \sum_{t=0}^{\infty} (P_1^t - \Pi_1)(b_1 - P_{12}b_2).
	\]
	Summing the geometric bounds gives $\sum_{t=0}^{\infty} \|P_1^t - \Pi_1\|_{\infty} \le \sum_{t=0}^{\infty} 2(1-\nu_*)^t = \frac{2}{\nu_*}$.
	
	Because $P_{12}$ is row-stochastic ($\|P_{12}\|_{\infty} = 1$), we obtain $\|x_1\|_{\infty} \le \frac{2}{\nu_*}(\|b_1\|_{\infty} + \|b_2\|_{\infty}) \le \frac{4}{\nu_*}\|b\|_{\infty}$. Consequently, $\|x_2\|_{\infty} \le \|b_2\|_{\infty} + \|x_1\|_{\infty} \le (1 + \frac{4}{\nu_*})\|b\|_{\infty}$.
	Defining $L^{\dagger}$ as the pseudo-inverse mapping into $\text{Ker}(L)^{\perp}$, we have established that
	\begin{align}\label{eq:pf_pot_stab_L_dagger_bound}
		\|L^{\dagger}\|_{\infty} \le 1 + \frac{4}{\nu_*} = 1 + \frac{9 \max_i \r_i \max_j \c_j}{\rho_{\A}\, m n} =: C_{\A},
	\end{align}
	so that $\|L^{\dagger} v\|_\infty \le C_{\A}\|v\|_\infty$ for every $v$ in the range of $L^{\dagger}$.
	
	We now rewrite $F(\xi) = 0$ as a fixed-point problem. Expanding $F$ at zero gives $F(\xi) = F(0) + DF|_{(0,0)}\xi + R(\xi)$. Restricting $\xi \in \text{Ker}(L)^{\perp}$, the root-finding equation becomes:
	\[
	\xi = -L^{\dagger}D^{-1}F(0) - L^{\dagger}D^{-1}R(\xi) := T(\xi).
	\] 
	We bound the components of $T$. The constant term satisfies
	\begin{align}\|D^{-1}F(0)\|_{\infty} = \max\left( \max_i |\frac{\r_i(\A)}{\r_i}-1|, \, \max_j |\frac{\c_j(\A)}{\c_j}-1| \right) \le \epsilon.
	\end{align}
	Thus, $\|L^{\dagger}D^{-1}F(0)\|_{\infty} \le C_{\A}\epsilon$.
	
	The remainder $R(\xi)$ has block components formed by $\sum_{j} \A_{ij} (e^{\alpha_i+\beta_j} - 1 - (\alpha_i+\beta_j))$. Using the scalar inequality $|e^x - 1 - x| \le \frac{1}{2}x^2 e^{|x|}$, we have:
	\[
	\|D^{-1}R(\xi)\|_{\infty} \le 2\|\xi\|_{\infty}^2 e^{2\|\xi\|_{\infty}}.
	\]
	Let $B_{R_0}$ be the closed $\ell_{\infty}$ ball of radius $R_0 = 4C_{\A}\epsilon$. We must show $T$ maps $B_{R_0}$ into itself and is a strict contraction. For $\xi \in B_{R_0}$,
	\[
	\|T(\xi)\|_{\infty} \le C_{\A} \epsilon + C_{\A} \left(2 R_0^2 e^{2R_0}\right) = C_{\A}\epsilon + 32 C_{\A}^3 \epsilon^2 e^{8C_{\A}\epsilon}.
	\]
	We require $\|T(\xi)\|_{\infty} \le 4C_{\A}\epsilon$, which reduces to $16 C_{\A}^2 \epsilon e^{8C_{\A}\epsilon} \le 1$. By our hypothesis, $\epsilon \le \epsilon_{\textup{max}} = \frac{1}{50C_{\A}^2}$. Since $C_{\A} \ge 1$, we have $8C_{\A}\epsilon \le \frac{8}{50C_{\A}} \le 0.16$, which implies $e^{8C_{\A}\epsilon} \le e^{0.16}$. Consequently,
	\[
	16 C_{\A}^2 \epsilon e^{8C_{\A}\epsilon} \le 16 C_{\A}^2 \left(\frac{1}{50C_{\A}^2}\right) (1.174) \le 0.38 < 1.
	\]
	Thus, $T$ maps $B_{R_0}$ strictly into itself. Furthermore, by the mean value theorem, the Lipschitz constant of $D^{-1}R(\xi)$ on $B_{R_0}$ is bounded by $4 R_0 e^{2 R_0} = 16C_{\A}\epsilon e^{8C_{\A}\epsilon}$. The Lipschitz constant of the map $T$ is therefore bounded by $C_{\A}(16C_{\A}\epsilon e^{8C_{\A}\epsilon}) \le 16C_{\A}^2(\frac{1}{50C_{\A}^2})(1.174) \approx 0.4 < 1$, verifying that $T$ is a strict contraction.
	
	By the Banach Fixed-Point Theorem, $T$ admits a unique fixed point $\xi_{\A} \in B_{R_0} \cap \text{Ker}(L)^{\perp}$, establishing the existence of the exact potentials with $\|(\balpha_{\A}, \bbeta_{\A})\|_{\infty} \le 4C_{\A}\epsilon$.
\end{proof}

We can now easily deduce the scaling limit result for deterministic Schr\"{o}dinger bridges stated in Theorem \ref{thm:SSB_deterministic_limit}. The proof demonstrates that under the $L^1$ convergence of the margins and the kernel, the sequence of discrete Schrödinger bridges converges to the continuous static Schrödinger bridge. It leverages the total stability estimate for the continuous problem.

\begin{proof}[\textbf{Proof of Theorem \ref{thm:SSB_deterministic_limit}}]
	Write $\mathcal{P} = \rho_{\r} \otimes \rho_{\c}$ and $\mathcal{P}_k = \rho_{\r_k} \otimes \rho_{\c_k}$ for the densities of the product measures $\mu_0 \otimes \mu_1$ and $\mu_0^k \otimes \mu_1^k$ with respect to the Lebesgue measure on $[0,1]^2$. We verify the hypotheses of Theorem \ref{thm:total_stability}.
	
	\emph{Cost bound.} Assumption \ref{assumption: limit_theory} gives the uniform bound $e^{-K} \le \varphi_{\bLambda_k}/\mathcal{P}_k \le e^K$ a.e.\ for all sufficiently large $k$. Since $L^1$ convergence implies convergence in measure and the range $[e^{-K}, e^K]$ is closed, the limit satisfies $e^{-K} \le \varphi_{\bLambda}/\mathcal{P} \le e^K$ a.e.\ as well. Writing $\kappa_k := -\log(\varphi_{\bLambda_k}/\mathcal{P}_k)$ and $\kappa := -\log(\varphi_{\bLambda}/\mathcal{P})$, we therefore have $\|\kappa_k\|_\infty, \|\kappa\|_\infty \le K$.
	
	\emph{Product-marginal ratio bound.} The $\delta$-smoothness of the margins gives $\delta \le \rho_{\r_k} \le \delta^{-1}$, with the same bound holding for $\rho_{\r}$ (by $L^1$ convergence and closed range), and analogously for the column densities. Consequently,
	\begin{align}\label{eq:pf_thm3_product_ratio}
		\delta^4 \le \frac{\mathcal{P}_k}{\mathcal{P}} \le \delta^{-4} \quad \textup{a.e.},
	\end{align}
	which is the hypothesis of Theorem \ref{thm:total_stability} with $M = 4\log\delta^{-1}$.
	
	\emph{Applying the total stability bound.} Theorem \ref{thm:total_stability} (taking the discrete bridge as $\pi^{\mu_0, \mu_1; \mathcal{R}}$ and the continuous bridge as $\pi^{\mu_0', \mu_1'; \mathcal{R}'}$) gives
	\begin{align*}
		d_H(\pi_k, \pi^{\r,\c;\bLambda})^2 &\le 8(\|\kappa_k\|_\infty + M)\bigl(\|\rho_{\r_k} - \rho_\r\|_{L^1} + \|\rho_{\c_k} - \rho_\c\|_{L^1}\bigr) \\
		&\qquad + 4\,e^{3(\|\kappa_k\|_\infty + M)\vee\|\kappa\|_\infty}\, d_H(\mathcal{R}_k, \mathcal{R})^2\\
		&\le 8(K + 4\log\delta^{-1})\bigl(\|\rho_{\r_k} - \rho_\r\|_{L^1} + \|\rho_{\c_k} - \rho_\c\|_{L^1}\bigr) + 4\, e^{3K}\,\delta^{-12}\, d_H(\mathcal{R}_k, \mathcal{R})^2,
	\end{align*}
	using $(K + 4\log\delta^{-1})\vee K = K + 4\log\delta^{-1}$ and $e^{12\log\delta^{-1}} = \delta^{-12}$. By Assumption \ref{assumption: limit_theory}, both the marginal $L^1$ errors and $d_H(\mathcal{R}_k, \mathcal{R})$ tend to zero, so $d_H(\pi_k, \pi^{\r,\c;\bLambda}) \to 0$.
\end{proof}

\section{Proofs for Sinkhorn-rescaled random matrices}
\label{sec:rescaled_proofs}

\subsection{Proofs of concentration results}

We record here, for completeness, a matrix-ensemble-agnostic bound on the probability that a nonnegative random matrix fails to be scalable to a $\delta$-smooth margin. This result is not invoked in the proofs of our main theorems, but it may be of independent interest and applies to ensembles outside the scope of Assumption \ref{assump:random_matrix}. Our argument is based on the following classical characterization of scalability due to Menon and Schneider \cite{menon1969spectrum}.

\begin{prop}[An equivalent condition for matrix scalability; Thm. 4.1 in \cite{menon1969spectrum}]\label{prop:matrix_scaling_equiv}
	Let $\A\in \R_{\ge 0}^{m\times n}$ be an $m\times n$ nonnegative matrix and let $(\r,\c)$ be an $m\times n$ margin with strictly positive row and column sums. Then the following are equivalent:
	\begin{description}
		\item[(i)] There exist positive diagonal matrices $\D_{1},\D_{2}$ such that $\D_{1}\A \D_{2}$ has margin $(\r,\c)$.
		
		\item[(ii)] For any $I\subset \{1,\dots,m\}$, $J\subset\{1,\dots,n\}$ such that $\A_{I^c,J}=O$, we have $\sum_{i\in I}\r_i \ge \sum_{j\in J}\c_j$, with equality if and only if $\A_{I,J^{c}}=O$.
	\end{description}
\end{prop}

\begin{lemma}[Scalability of random matrices]\label{lem:scalability}
	Let $(\r,\c)$ be an $m\times n$ $\delta$-smooth margin and let $\X$ be an $m\times n$ nonnegative random matrix with independent entries. Denoting $p_{0}=\max_{i,j} \P(\X_{ij}=0)$,
	\begin{align}\label{eq:scalability_prob_bd1}
		\P(\X\notin \mathcal{S}(\r,\c)) \le 2^{m+n-2}\, p_0^{\delta^3(m \wedge n)/2}.
	\end{align}
	In particular, suppose there exist constants $\gamma\in(0,1]$ and $\eta>0$ such that $\tfrac{m\land n}{m\lor n} \ge \gamma$ and
	\begin{align}\label{eq:p_0_threshold}
		p_0 \le \exp\!\left( -\frac{2\log 2}{\delta^3}\,(1 + \gamma^{-1}) - \eta \right).
	\end{align}
	Then
	\begin{align}\label{eq:scalability_prob_exp}
		\P(\X\notin \mathcal{S}(\r,\c)) \le \frac{1}{4} \exp\!\left( -\frac{\eta \delta^3}{2}\,(m \wedge n) \right).
	\end{align}
\end{lemma}

In particular, if $p_0 = 0$ (so that $\X$ is a.s.\ positive), then $\X$ is a.s.\ scalable for any positive margin $(\r,\c)$. More generally, a bounded aspect ratio together with the zero-probability threshold \eqref{eq:p_0_threshold} yields exponential decay of the failure probability in $m\wedge n$.

\begin{proof}[\textbf{Proof of Lemma \ref{lem:scalability}}]
	Let $R_I = \sum_{i \in I} \r(i)$ and $C_J = \sum_{j \in J} \c(j)$ for any row subset $I \subset \{1,\dots, m\}$ and column subset $J \subset \{1, \dots, n\}$. Denote the maximum zero-probability by $p_0 := \max_{i,j} \P(\X_{ij} = 0)$ and define 
	\begin{align}
		\rho := \min_{i,j} \frac{-\log \P(\X_{ij}=0)}{\r(i)\c(j)}.
	\end{align}
	This guarantees that $\P(\X_{ij} = 0) \le e^{-\rho \r(i) \c(j)}$ for all entries. By the independence of the matrix entries, the probability of any block $\X_{I^c, J}$ being entirely zero is
	\begin{align*}
		\P(\X_{I^c, J} = O) &= \prod_{i \in I^c, j \in J} \P(\X_{ij} = 0) \le \exp\left( -\rho \sum_{i \in I^c, j \in J} \r(i) \c(j) \right) = \exp(-\rho \cdot R_{I^c} \cdot C_J).
	\end{align*}
	
	Consider the event  $\{\X\notin \mathcal{S}(\r,\c)\}$ that $\X$ cannot be scaled to the margin $(\r,\c)$. By Proposition \ref{prop:matrix_scaling_equiv}, this occurs only if there exist subsets $I \subsetneq \{1,\dots,m\}$ and $J \subsetneq \{1,\dots,n\}$ such that $\X_{I^c,J} = O$ and $R_I \le C_J$.
	Note that $R_I \le C_J$ is equivalent to $N - R_{I^c} \le C_J$, which means $R_{I^c} + C_J \ge N$. Let us define the set of such failure configurations as:
	\begin{align*}
		\mathcal{I} &= \left\{(I,J): I\subsetneq\{1,\dots,m\},\, J\subsetneq\{1,\dots,n\},\, R_{I^c} + C_J \ge N \right\}.
	\end{align*}
	Thus, $\{\X\notin \mathcal{S}(\r,\c)\} \subseteq \bigcup_{(I,J)\in \mathcal{I}} \{\X_{I^c,J}=O\}$. 
	
	We bound the probability of these events by decomposing $\mathcal{I} = \mathcal{I}_1 \cup \mathcal{I}_2$, where $\mathcal{I}_1 := \{(\emptyset, J) \in \mathcal{I}\}$ and $\mathcal{I}_2 := \mathcal{I} \setminus \mathcal{I}_1$.
	
	For any $(\emptyset, J) \in \mathcal{I}_1$, we have $I^c = \{1, \dots, m\}$, so $R_{I^c} = N$. Since $J$ is non-empty, $C_J \ge \c_{\min}$. Therefore,
	\begin{align*}
		\P(\X_{I^c,J}=O) \le \exp(-\rho R_{I^c}C_{J}) \le \exp(-\rho N \c_{\min}).
	\end{align*}
	
	For any $(I,J) \in \mathcal{I}_2$, $I$ is non-empty, meaning $1 \le |I^c| \le m-1$, which implies $\r_{\min} \le R_{I^c} \le N - \r_{\min}$. Using the constraint $C_J \ge N - R_{I^c}$, we have:
	\begin{align*}
		\P(\X_{I^c,J}=O) \le \exp(-\rho R_{I^c}C_{J}) \le \exp(-\rho R_{I^c}(N-R_{I^c})).
	\end{align*}
	The quadratic function $x(N-x)$ achieves its minimum on the interval $[\r_{\min}, N-\r_{\min}]$ at the boundaries, so:
	\begin{align*}
		R_{I^c}(N-R_{I^c}) \ge \r_{\min}(N-\r_{\min}) \ge \r_{\min}\left(N - \tfrac{N}{m}\right) = \tfrac{m-1}{m}N\r_{\min} \ge \tfrac{N \r_{\min}}{2},
	\end{align*}
	
	Combining both cases, for any $(I,J) \in \mathcal{I}$, the block probability is bounded by:
	\begin{align*}
		\P(\X_{I^c,J}=O) \le \exp\left(-\rho\left(\c_{\min}\wedge\r_{\min}\right)\frac{N}{2}\right).
	\end{align*}
	Taking a union bound over all $2^{m+n-2}$ valid subset pairs $(I,J)$, we obtain:
	\begin{align*}
		\P(\X\notin \mathcal{S}(\r,\c)) \le 2^{m+n-2}\exp\left(-\rho \left(\c_{\min}\wedge\r_{\min}\right)\frac{N}{2}\right).
	\end{align*}
	
	Next, we use the $\delta$-smoothness of the margin $(\r,\c)$ to bound this expression entirely in terms of $p_0$ and $\delta$. By Definition \ref{def:delta_smooth_margins}, we have:
	\begin{align*}
		\r_{\max}\c_{\max} \le \frac{N^2}{\delta^2 mn} \qquad \text{and} \qquad \c_{\min} \wedge \r_{\min} \ge \frac{\delta N}{m \vee n}.
	\end{align*}
	This allows us to uniformly lower-bound the rate parameter $\rho$:
	\begin{align*}
		\rho \ge \frac{-\log p_0}{\r_{\max}\c_{\max}} \ge \frac{-\delta^2 mn \log p_0}{N^2}.
	\end{align*}
	Substituting this into the exponent of our worst-case failure probability yields:
	\begin{align*}
		\rho \left(\c_{\min}\wedge\r_{\min}\right)\frac{N}{2} \ge \left( \frac{-\delta^2 mn \log p_0}{N^2} \right) \left(\frac{\delta N}{m \vee n}\right) \frac{N}{2} = \frac{-\delta^3 (m \wedge n) \log p_0}{2}.
	\end{align*}
	Applying this to the union bound, we conclude:
	\begin{align*}
		\P(\X\notin \mathcal{S}(\r,\c)) \le 2^{m+n-2} \exp\left( \frac{\delta^3 (m \wedge n) \log p_0}{2} \right) = 2^{m+n-2} p_0^{\frac{\delta^3(m \wedge n)}{2}}.
	\end{align*}
	%Setting this upper bound to $\eps_{\textup{scale}}$ completes the proof.
	This shows \eqref{eq:scalability_prob_bd1}.

	For the second part of the statement, we incorporate the aspect ratio bound $\frac{m \wedge n}{m \vee n} \ge \gamma$, which implies $m \vee n \le \gamma^{-1}(m \wedge n)$. Therefore, the sum of the dimensions is bounded by:
	\begin{align*}
		m + n = (m \wedge n) + (m \vee n) \le (1 + \gamma^{-1})(m \wedge n).
	\end{align*}
	It follows that 
	\begin{align}
		2^{m+n-2}\le \frac{1}{4}\exp\big( (m+n)\log 2 \big) \le  \frac{1}{4} \exp\left( (1 + \gamma^{-1})(m \wedge n) \log 2 \right),
	\end{align}
	so \eqref{eq:scalability_prob_bd1} yields 
	\begin{align*}
		\P(\X\notin \mathcal{S}(\r,\c)) \le \frac{1}{4} \exp\left( (m \wedge n) \left[ (1+\gamma^{-1})\log 2 + \frac{\delta^3}{2} \log p_0 \right] \right).
	\end{align*}
	Then \eqref{eq:scalability_prob_exp} follows immediately from the above and the hypothesis  \eqref{eq:p_0_threshold}. This completes the proof.
\end{proof}

The next result shows that when $\X$ is rescaled by the \textit{deterministic} scaling factors (those computed for the mean $\bLambda$), the resulting matrix $\hat{\X}^{\r,\c} = \D(e^{\balpha})\X\D(e^{\bbeta})$ has margins close to $(\r, \c)$. This is the key step justifying the comparison model \eqref{eq:comparison_model_intro}.

\begin{prop}[Approximate scaling of the comparison model] \label{prop:approximate_scaling}
	Let $(\r, \c)$ be an $m \times n$ positive margin. Under Assumption \ref{assump:random_matrix}, let $(\balpha, \bbeta)$ denote the deterministic Schr\"{o}dinger potentials scaling $\bLambda$ to $(\r, \c)$, and let $\hat{\X}^{\r,\c} = \D(e^{\balpha})\X\D(e^{\bbeta})$ be the comparison model. Define the variance and scale proxies
	\begin{align}\label{eq:V_r_V_c_B_def}
		V_{r} := n\sigma^2 e^{4K}\,\frac{N^2}{\|\bLambda\|_1^2}, \qquad V_{c} := m\sigma^2 e^{4K}\,\frac{N^2}{\|\bLambda\|_1^2}, \qquad B := e^{2K}\,\frac{N R}{\|\bLambda\|_1}.
	\end{align}
	Then, for every $\eps > 0$, $i \in [m]$, and $j \in [n]$,
	\begin{align}\label{eq:prop_approx_row}
		\P\!\left( \bigl|r_i(\hat{\X}^{\r,\c}) - \r(i)\bigr| \ge \eps\, \r(i) \right) &\le 2\exp\!\left( -\frac{\eps^2 \r(i)^2 / 2}{V_r + B\,\eps\, \r(i)} \right), \\
		\label{eq:prop_approx_col}
		\P\!\left( \bigl|c_j(\hat{\X}^{\r,\c}) - \c(j)\bigr| \ge \eps\, \c(j) \right) &\le 2\exp\!\left( -\frac{\eps^2 \c(j)^2 / 2}{V_c + B\,\eps\, \c(j)} \right).
	\end{align}
\end{prop}

In the proof of the above result, we will use the following elementary Bernstein's inequality for sub-exponential random variables. 

\begin{lemma}[Bernstein's inequality for sub-exponential variables]\label{lem:bernstein_RV}
	Let $Y_1, \dots, Y_n$ be independent centered random variables satisfying the moment bounds
	\[
	\E\left[|Y_i|^q\right] \le \frac{q!}{2}\sigma_i^2 R_i^{q-2} \quad \text{for all integers } q \ge 2. 
	\]
	Then, for any deterministic constants $a_1, \dots, a_n \in \R$ and any $t > 0$, we have 
	\[
	\P\left(\left|\sum_{i=1}^{n}a_i Y_i\right| > t\right) \le 2\exp\left(-\frac{t^2/2}{\sum_{i=1}^{n}a_i^2\sigma_i^2 + a_{\max}R_{\max}t}\right),
	\]
	where $a_{\max} = \max_{1 \le i \le n} |a_i|$ and $R_{\max} = \max_{1 \le i \le n} R_i$.
\end{lemma}

\begin{proof}
	Let $S = \sum_{i=1}^n a_i Y_i$. By the independence of the variables $Y_i$, the moment generating function of $S$ factors as $\E[e^{\tau S}] = \prod_{i=1}^n \E[e^{\tau a_i Y_i}]$. 
	
	Using the standard Taylor expansion and the given moment bounds, for any $0 < \tau < 1/(a_{\max} R_{\max})$, we can bound the moment generating function by:
	\begin{align*}
		\E[\exp(\tau S)] &\le \exp\left( \frac{\tau^2}{2} \sum_{i=1}^n \frac{a_i^2 \sigma_i^2}{1 - |a_i| R_i \tau} \right) \le \exp\left(\frac{\tau^2 \sum_{i=1}^{n}a_i^2\sigma_i^2}{2(1 - a_{\max}R_{\max}\tau)}\right).
	\end{align*}
	Applying Markov's inequality to $e^{\tau S}$, we obtain the upper tail bound:
	\[
	\P(S \ge t) = \P(e^{\tau S} \ge e^{\tau t}) \le \exp\left(\frac{\tau^2 \sum_{i=1}^{n}a_i^2\sigma_i^2}{2(1 - a_{\max}R_{\max}\tau)} - \tau t\right).
	\]
	Optimizing the bound by choosing 
	\[
	\tau = \frac{t}{\sum_{i=1}^{n}a_i^2\sigma_i^2 + a_{\max}R_{\max}t},
	\]
	which satisfies $0 < \tau < 1/(a_{\max} R_{\max})$, yields:
	\[
	\P(S \ge t) \le \exp\left(-\frac{t^2/2}{\sum_{i=1}^{n}a_i^2\sigma_i^2 + a_{\max}R_{\max}t}\right).
	\]
	Applying the identical argument to $-S$ and taking a union bound provides the two-sided inequality, completing the proof.
\end{proof}

We will also repeatedly use the following two-sided uniform bound on the deterministic scaling potentials.

\begin{lemma}\label{lem: boundedness_matrix_scaling_potentials}
	Let $\bLambda \in \R^{m \times n}_{\ge 0}$ be a  matrix with $\bLambda_{ij} = 0$ if and only if $\r(i)\c(j) = 0$. Let $\bar{\bLambda} := \bLambda/\|\bLambda\|_1$, $(\bar{\r}, \bar{\c}) := (\r/N, \c/N)$ with $N=\sum_{i} \r(i)=\sum_{j} \c(j)$, and $\bar{\kappa}(i,j) := \log(\bar{\r}(i)\bar{\c}(j)/\bar{\bLambda}_{ij})$ (with the convention $0/0=1$). 
	Then a canonical choice of matrix-scaling potentials $(\balpha, \bbeta)$ (fixing the gauge $\sum_j \bar{\bbeta}(j)\bar{\c}(j) = 0$ in the normalized problem below) satisfies
	\begin{align}\label{eq:mat_scaling_potentials_bounds}
		-2(\bar{\kappa}_{\textup{-}} + \bar{\kappa}_{\textup{+}}) + \log\!\left(\tfrac{N}{\|\bLambda\|_1}\right) \;\le\; \balpha(i) + \bbeta(j) \;\le\; 2\bar{\kappa}_{\textup{+}} + \log\!\left(\tfrac{N}{\|\bLambda\|_1}\right).
	\end{align}
	In particular, \begin{align}\label{eq:mat_scaling_potentials_bounds_K}
		\tfrac{e^{-4 \| \bar{\kappa}\|_{\infty} } N}{\|\bLambda\|_1} \;\le\; e^{\balpha(i)+\bbeta(j)} \;\le\; \tfrac{e^{2 \| \bar{\kappa}\|_{\infty} } N}{\|\bLambda\|_1} \qquad \text{for all } i,j.
	\end{align}
\end{lemma}

\begin{proof}
	View the normalized margins $\bar{\r}, \bar{\c}$ as probability measures $\mu_0 = \sum_{i=1}^{m}\bar{\r}(i)\,\delta_i$ on $[m]$ and $\mu_1 = \sum_{j=1}^{n}\bar{\c}(j)\,\delta_j$ on $[n]$, and the normalized prior $\bar{\bLambda}$ as the probability measure $\mathcal{R} = \sum_{i,j}\bar{\bLambda}_{ij}\,\delta_{(i,j)}$ on $[m]\times[n]$. The Radon--Nikodym derivative of $\mathcal{R}$ with respect to $\mu_0\otimes\mu_1$ is
	\begin{align*}
		\frac{d\mathcal{R}}{d(\mu_0\otimes\mu_1)}(i,j) \;=\; \frac{\bar{\bLambda}_{ij}}{\bar{\r}(i)\bar{\c}(j)} \;=\; e^{-\bar{\kappa}(i,j)},
	\end{align*}
	so $\bar{\kappa}\in L^{\infty}(\mu_0\otimes\mu_1)$ plays the role of the cost function in Lemma \ref{lem: boundedness_continuous_potentials}. Let $(\bar{\balpha}, \bar{\bbeta})$ be the unique Schr\"{o}dinger potentials scaling $\bar{\bLambda}$ to margins $(\bar{\r},\bar{\c})$ with the gauge $\sum_j \bar{\bbeta}(j)\bar{\c}(j)=\langle\bar{\bbeta},\mu_1\rangle = 0$. Applying Lemma \ref{lem: boundedness_continuous_potentials} to $(\mu_0,\mu_1,\mathcal{R})$ gives
	\begin{align*}
		-(\bar{\kappa}_{\textup{-}} + \bar{\kappa}_{\textup{+}}) \;\le\; \bar{\balpha}(i)+ \bar{\bbeta}(j) \;\le\; \bar{\kappa}_{\textup{+}},
	\end{align*}
	where $\bar{\kappa}_{\textup{+}}:=\|\bar{\kappa}^+\|_\infty$ and $\bar{\kappa}_{\textup{-}}:=\|\bar{\kappa}^-\|_\infty$. Summing the bounds for $\bar{\balpha}(i)$ and $\bar{\bbeta}(j)$ yields
	\begin{align}\label{eq:pf_mat_normalized_twosided}
		-2(\bar{\kappa}_{\textup{-}} + \bar{\kappa}_{\textup{+}}) \;\le\; \bar{\balpha}(i) + \bar{\bbeta}(j) \;\le\; 2\bar{\kappa}_{\textup{+}}.
	\end{align}
	
	\smallskip
	\textbf{Undoing the normalization.} Define the unnormalized potentials by $\balpha(i)+\bbeta(j) := \bar{\balpha}(i)+\bar{\bbeta}(j) + \log(N/\|\bLambda\|_1)$ (distributing the additive shift arbitrarily between $\balpha$ and $\bbeta$). Using $\bar{\bLambda}_{ij}=\bar{\r}(i)\bar{\c}(j)\,e^{-\bar{\kappa}(i,j)}$ and the Schr\"{o}dinger equation for $(\bar{\balpha},\bar{\bbeta})$,
	\begin{align*}
		\sum_j e^{\balpha(i)+\bbeta(j)}\bLambda_{ij} \;&=\; \tfrac{N}{\|\bLambda\|_1}\cdot\|\bLambda\|_1 \sum_j e^{\bar{\balpha}(i)+\bar{\bbeta}(j)}\bar{\bLambda}_{ij} \;=\; N\bar{\r}(i)\, e^{\bar{\balpha}(i)}\sum_j e^{\bar{\bbeta}(j)-\bar{\kappa}(i,j)}\bar{\c}(j) \\
		&=\; N\bar{\r}(i) \;=\; \r(i);
	\end{align*}
	the column identity $\sum_i e^{\balpha(i)+\bbeta(j)}\bLambda_{ij} = \c(j)$ follows by the symmetric computation. Hence $(\balpha,\bbeta)$ are matrix-scaling potentials for $(\bLambda,\r,\c)$, and inserting the additive shift $\log(N/\|\bLambda\|_1)$ into \eqref{eq:pf_mat_normalized_twosided} proves \eqref{eq:mat_scaling_potentials_bounds}. The specialization \eqref{eq:mat_scaling_potentials_bounds_K} follows from $\bar{\kappa}_{\textup{+}}, \bar{\kappa}_{\textup{-}}\le \|\bar{\kappa}\|_\infty$.
\end{proof}

Now we prove Proposition \ref{prop:approximate_scaling}. 

\begin{proof}[\textbf{Proof of Proposition \ref{prop:approximate_scaling}}]
	For the $(i,j)$-th entry of the comparison model, $\hat{\X}^{\r,\c}_{ij} = e^{\balpha(i)+\bbeta(j)}\X_{ij}$. Since the potentials $(\balpha, \bbeta)$ exactly scale the mean matrix $\bLambda$ to $(\r,\c)$, we have $\sum_{j=1}^n e^{\balpha(i)+\bbeta(j)}\bLambda_{ij} = \r(i)$. Hence for any fixed row $i$,
	\[
	\bigl|r_i(\hat{\X}^{\r,\c}) - \r(i)\bigr| = \left|\sum_{j=1}^{n} e^{\balpha(i)+\bbeta(j)} (\X_{ij} - \bLambda_{ij}) \right|.
	\]
	Set $a_{ij} := e^{\balpha(i)+\bbeta(j)}$. By Lemma \ref{lem: boundedness_matrix_scaling_potentials} and Assumption \ref{assump:bounded_cost}, the weights are uniformly bounded by $a_{ij} \le a_{\max} := e^{2K} N/\|\bLambda\|_1$. The variance proxy for the row sum is therefore bounded by
	\[
	\sum_{j=1}^n a_{ij}^2 \sigma^2 \le n \sigma^2 a_{\max}^2 = n\sigma^2 e^{4K}\,\frac{N^2}{\|\bLambda\|_1^2} = V_r,
	\]
	and the scale parameter by $a_{\max} R \le e^{2K}\,N R / \|\bLambda\|_1 = B$. Applying Bernstein's inequality (Lemma \ref{lem:bernstein_RV}) with threshold $t = \eps \r(i)$ yields \eqref{eq:prop_approx_row}. By the symmetric argument with the roles of $i$ and $j$ (and $n$, $m$) interchanged, the column bound \eqref{eq:prop_approx_col} follows.
\end{proof}

A central step in the proof of Theorem \ref{thm:concentration_margins_rs} is to verify that the random matrix $\X$ satisfies the row-alignment hypothesis of Theorem \ref{thm:potential_stability} with high probability. By independence across columns, each pairwise row inner product $\sum_{j=1}^n \X_{i_1 j}\X_{i_2 j}$ is a sum of $n$ independent non-negative summands whose mean is at least $n\bLambda_{\min}^2$ --- either by row independence when $i_1\ne i_2$, or by Jensen's inequality when $i_1 = i_2$. The lower tail of this sum is controlled by a sub-Weibull concentration inequality: products and squares of sub-exponential variables are sub-Weibull of order $1/2$, and the resulting tail bound is strong enough that a union bound over all $m^2$ ordered row pairs still leaves a polynomially-decaying failure probability. The lemma records the resulting uniform lower bound in a form used later in the proof of Theorem \ref{thm:concentration_margins_rs}.

\begin{lemma}[High-probability row alignment]\label{lem:row_overlap}
	Let $\X \in \R^{m\times n}$ be a random matrix with independent, non-negative entries. Assume there exists $\lambda_{\min} > 0$ such that $\E[\X_{ij}] = \bLambda_{ij} \ge \lambda_{\min}$ for all $i, j$. Suppose the entries have uniformly bounded sub-exponential norms, $\max_{i,j} \|\X_{ij}\|_{\psi_1} \le K_\psi$ (here $K_\psi$ denotes the Orlicz-$\psi_1$ scale of the entries, not the bounded-cost parameter $K$ of Assumption \ref{assump:bounded_cost}). Define the overlap threshold parameter $a = \frac{1}{2}\lambda_{\min}^2$. Then there exists an absolute constant $c > 0$ such that the row alignment condition holds uniformly for all pairs of rows (including $i_1 = i_2$) with probability at least:
	\begin{align}\label{eq:overlap_prob_bound}
		\P\left( \forall i_1, i_2 \in [m],\ \sum_{j=1}^n \X_{i_1 j}\X_{i_2 j} \ge a n \right) \ge 1 - m^2 \exp\left( - c \min\left( n \frac{\lambda_{\min}^4}{K_\psi^4}, \sqrt{n} \frac{\lambda_{\min}}{K_\psi} \right) \right).
	\end{align}
\end{lemma}

\begin{proof}
	Fix a pair of rows $i_1, i_2 \in [m]$ (possibly equal). For $j = 1, \dots, n$, define the random variables $Z_j = \X_{i_1 j} \X_{i_2 j}$. Because the entries of $\X$ are independent across columns, the sequence $Z_1, \dots, Z_n$ consists of independent, non-negative random variables (this holds whether $i_1 \ne i_2$ or $i_1 = i_2$, since $Z_j$ depends only on column $j$).
	
	When $i_1 \ne i_2$, independence of the rows gives
	\[
	\E[Z_j] = \E[\X_{i_1 j}] \E[\X_{i_2 j}] = \bLambda_{i_1 j} \bLambda_{i_2 j} \ge \lambda_{\min}^2.
	\]
	When $i_1 = i_2$, Jensen's inequality gives $\E[Z_j] = \E[\X_{ij}^2] \ge (\E \X_{ij})^2 = \bLambda_{ij}^2 \ge \lambda_{\min}^2$. Furthermore, because the product of two independent sub-exponential random variables, or the square of a sub-exponential random variable, is a sub-Weibull random variable of order $\alpha = 1/2$, there exists an absolute constant $C_1 > 0$ such that the sub-Weibull norm is bounded by:
	\[
	\|Z_j\|_{\psi_{1/2}} \le C_1 \|\X_{i_1 j}\|_{\psi_1} \|\X_{i_2 j}\|_{\psi_1} \le C_1 K_\psi^2 =: M.
	\]
	We wish to bound the probability that the sum $S = \sum_{j=1}^n Z_j$ drops below $an = \frac{1}{2}n\lambda_{\min}^2$. If $S < \frac{1}{2}n\lambda_{\min}^2$, its deviation from its mean is at least $t = \frac{1}{2}n\lambda_{\min}^2$.
	
	We apply the generalized concentration inequality for sums of independent sub-Weibull($1/2$) random variables. For any $t > 0$, there exists an absolute constant $c_0 > 0$ such that:
	\[
	\P\left( \sum_{j=1}^n (Z_j - \E[Z_j]) \le -t \right) \le \exp\left( -c_0 \min\left( \frac{t^2}{n M^2}, \left(\frac{t}{M}\right)^{1/2} \right) \right).
	\]
	Substituting our deviation $t = \frac{1}{2}n\lambda_{\min}^2$ into the bound yields:
	\begin{align*}
		\P\left( \sum_{j=1}^n Z_j < a n \right) &\le \exp\left( -c_0 \min\left( \frac{n^2 \lambda_{\min}^4 / 4}{n M^2}, \left(\frac{n\lambda_{\min}^2 / 2}{M}\right)^{1/2} \right) \right) \\
		&= \exp\left( -c_0 \min\left( n \frac{\lambda_{\min}^4}{4 M^2}, \sqrt{n} \frac{\lambda_{\min}}{\sqrt{2M}} \right) \right).
	\end{align*}
	Substituting $M = C_1 K_\psi^2$ and absorbing the fractional and absolute constants into a new universal constant $c > 0$, we obtain the tail bound:
	\[
	\P\left( \sum_{j=1}^n Z_j < a n \right) \le \exp\left( - c \min\left( n \frac{\lambda_{\min}^4}{K_\psi^4}, \sqrt{n} \frac{\lambda_{\min}}{K_\psi} \right) \right).
	\]
	
	This bound holds for any fixed (ordered) pair of rows $(i_1, i_2) \in [m]^2$, of which there are $m^2$. Taking a union bound over all such pairs yields the uniform probability lower bound in \eqref{eq:overlap_prob_bound}.
\end{proof}

\begin{proof}[\textbf{Proof of Theorem \ref{thm:concentration_margins_rs}}]
	We construct the high-probability event $\mathcal{E}_{1}$ by intersecting two favorable events: $E_{\textup{approx}}(\eps_0)$ and $E_{\textup{overlap}}$. As will become clear, their intersection already forces $\X \in \mathcal{S}(\r,\c)$, so the scalability event from Lemma \ref{lem:scalability} need not be union-bounded separately.
	
	\smallskip
	\textbf{1. Approximate Scaling ($E_{\textup{approx}}(\eps_0)$):} Let $E_{\textup{approx}}(\eps_0)$ be the event that $r_i(\hat{\X}^{\r,\c}) \in [(1-\eps_0)\r(i),(1+\eps_0)\r(i)]$ for all $i \in [m]$ and $c_j(\hat{\X}^{\r,\c}) \in [(1-\eps_0)\c(j),(1+\eps_0)\c(j)]$ for all $j \in [n]$. By Proposition \ref{prop:approximate_scaling} (with $V_r, V_c, B$ as in \eqref{eq:V_r_V_c_B_def}), for every $i$,
	\[
	\P\!\left(\bigl|r_i(\hat{\X}^{\r,\c}) - \r(i)\bigr| \ge \eps_0 \r(i)\right) \le 2\exp\!\left(-\frac{\eps_0^2 \r(i)^2/2}{V_r + B\,\eps_0\, \r(i)}\right),
	\]
	and analogously for columns via $V_c$. Using the standard fact that for a sub-exponential Bernstein bound $2\exp(-t^2/2/(V + B t))$, the choice $t \ge \sqrt{2V\tau} + 2B\tau$ makes it $\le 2\exp(-\tau)$, it suffices that $\eps_0 \r(i) \ge \sqrt{2V_r \tau} + 2 B \tau$ for every $i$, and analogously for columns. By $\delta$-smoothness, $\r(i) \ge \delta N/m$, so it is enough to require
	\[
	\eps_0 \ge \frac{m}{\delta N}\bigl(\sqrt{2 V_r \tau} + 2 B \tau\bigr) = \frac{m\, e^{2K}}{\delta \|\bLambda\|_1}\bigl(\sigma\sqrt{2 n \tau} + 2 R\tau\bigr),
	\]
	and symmetrically $\eps_0 \ge \frac{n e^{2K}}{\delta\|\bLambda\|_1}(\sigma\sqrt{2m\tau} + 2R\tau)$ for the columns. Replacing $m$ and $n$ by $(m\vee n)$ on the right-hand sides yields exactly the definition of $\eps_0$ in \eqref{eq:eps0_Cstar_def}. A union bound over $m+n \le 2(m \vee n)$ rows and columns gives
	\[
	\P(E_{\textup{approx}}(\eps_0)^c) \le 4(m \vee n)\exp(-\tau) = (m \vee n)^{-D}.
	\]
	
	\smallskip
	\textbf{2. Row alignment ($E_{\textup{overlap}}$):} Let $E_{\textup{overlap}}$ be the event that the row alignment satisfies $\sum_{j=1}^n \X_{i_1 j} \X_{i_2 j} \ge \frac{1}{2}n \bLambda_{\min}^2$ uniformly for all pairs of rows $i_1, i_2 \in [m]$ (including $i_1 = i_2$). To bound the probability of this event via the sub-Weibull concentration of the product, we require a uniform upper bound on the sub-exponential norm of the matrix entries.
	
	By Assumption \ref{assump:random_matrix}, the centered entries $\X_{ij} - \bLambda_{ij}$ satisfy the Bernstein moment condition. By standard Orlicz norm equivalences, there exists an absolute constant $C_0 > 0$ such that their sub-exponential norm is bounded by $C_0(\sigma + R)$. Furthermore, the $\psi_1$ norm of a constant is bounded by the constant scaled by $1/\ln 2$. Applying the triangle inequality gives:
	\[
	\|\X_{ij}\|_{\psi_1} \le \|\X_{ij} - \bLambda_{ij}\|_{\psi_1} + \|\bLambda_{ij}\|_{\psi_1} \le C_0(\sigma + R) + \frac{\bLambda_{\max}}{\ln 2}.
	\]
	Defining the absolute constant $C_1 = \max(C_0, 1/\ln 2)$, we obtain the explicit uniform bound $\|\X_{ij}\|_{\psi_1} \le C_1 (\sigma + R + \bLambda_{\max})$. Because $C_1$ is a universal constant, it can be absorbed into the final absolute constant $c$ in the exponent of the sub-Weibull tail bound. Thus, utilizing the global scale parameter $M = \sigma + R + \bLambda_{\max}$, Lemma \ref{lem:row_overlap} yields the failure probability:
	\[
	\P(E_{\textup{overlap}}^c) \le m^2 \exp(-c\,\Phi_{\bLambda}).
	\]
	
	\smallskip
	\textbf{3. Combining the events.} We set $\mathcal{E}_{1} = E_{\textup{approx}}(\eps_0) \cap E_{\textup{overlap}}$ and work on this intersection. Consider the approximately scaled matrix $\hat{\X}^{\r,\c} = \D(e^{\balpha})\X\D(e^{\bbeta})$. By definition of $E_{\textup{approx}}(\eps_0)$, its margins satisfy $\r(\hat{\X}^{\r,\c}) \in (1 \pm \eps_0)\r$ and $\c(\hat{\X}^{\r,\c}) \in (1 \pm \eps_0)\c$.
	Furthermore, Lemma \ref{lem: boundedness_matrix_scaling_potentials} gives a uniform lower bound on the deterministic potentials, yielding $\hat{\X}^{\r,\c}_{ij} \ge (e^{-4K} N / \|\bLambda\|_1) \X_{ij}$. Since $E_{\textup{overlap}}$ holds for all $i_1, i_2 \in [m]$, the row alignment for $\hat{\X}^{\r,\c}$ is bounded below by:
	\[
	\sum_{j=1}^n \hat{\X}^{\r,\c}_{i_1 j}\hat{\X}^{\r,\c}_{i_2 j} \ge \left(\frac{e^{-4K} N}{\|\bLambda\|_1}\right)^2 \sum_{j=1}^n \X_{i_1 j}\X_{i_2 j} \ge \frac{e^{-8K} N^2}{2 \|\bLambda\|_1^2} \bLambda_{\min}^2 n =: \hat{a} n.
	\]
	We now apply Theorem \ref{thm:potential_stability} to $\hat{\X}^{\r,\c}$ with overlap parameter $\hat{a}$. Because $(\r,\c)$ is $\delta$-smooth, the stability constant evaluates to:
	\[
	C = 1 + \frac{9 N^2}{\hat{a}\, \delta^2 m^2 n^2} = 1 + \frac{9 N^2}{\delta^2 m^2 n^2} \left( \frac{2 e^{8K} \|\bLambda\|_1^2}{\bLambda_{\min}^2 N^2} \right) = 1 + \frac{18\, e^{8K} \|\bLambda\|_1^2}{\delta^2 \bLambda_{\min}^2 m^2 n^2} = C_*.
	\]
	By assumption, $\eps_0 \le \frac{1}{50 C_*^2}$. Therefore, Theorem \ref{thm:potential_stability} produces corrective potentials $\Delta \balpha, \Delta \bbeta$ with $\|(\Delta \balpha, \Delta \bbeta)\|_{\infty} \le 4C_* \eps_0$ such that $\D(e^{\Delta \balpha})\hat{\X}^{\r,\c}\D(e^{\Delta \bbeta})$ has exact margins $(\r, \c)$. Equivalently, we have $\D(e^{\balpha+\Delta\balpha})\, \X\, \D(e^{\bbeta+\Delta\bbeta}) \in \T(\r,\c)$, so $\X \in \mathcal{S}(\r,\c)$ on $\mathcal{E}_1$ — no separate scalability union bound is needed.
	
	By uniqueness of the Sinkhorn scaling potentials (up to global shifts), there exists a scalar $t \in \R$ such that $\balpha_{\X} = \balpha + \Delta \balpha - t\mathbf{1}$ and $\bbeta_{\X} = \bbeta + \Delta \bbeta + t\mathbf{1}$. The bounds on the deviations of the exact potentials follow immediately. Taking the union bound $\P(\mathcal{E}_1^c) \le \P(E_{\textup{approx}}(\eps_0)^c) + \P(E_{\textup{overlap}}^c)$ gives the probability estimate stated in the theorem.
\end{proof}

The remaining ingredient for the fluctuation analysis of Section \ref{sec:fluctuation_statements} is a spectral-norm concentration bound for the centered noise matrix $\X - \bLambda$, which will be applied in the proofs of Theorem \ref{thm:first_approximation_rescaled} and Lemma \ref{lem:rescaled_approx}. We derive it by Girko-type Hermitian dilation followed by a matrix Bernstein inequality. We begin by recording the vector-form matrix Bernstein inequality we will use.

\begin{theorem}[Matrix Bernstein Inequality, Thm. 6.2 in \cite{tropp2012user}]\label{thm:matrix_ber}
	Let $\{\bZ_k\}_{k=1}^{N}$ be a collection of independent, self-adjoint, random $d\times d$ mean zero matrices. Suppose the moment growth
	\[
	\E[\bZ_k^q]\preccurlyeq\frac{q!}{2}R^{q-2}\bA_k^2 \quad \text{ for } q=2,3,4,...
	\]
	Denote $\sigma^2=\|\sum_{k=1}^{N}\bA_k^2\|,$ then
	\[
	\P\left(\left\|\sum_{k=1}^{N}\bZ_k\right\|>t\right)\le \begin{cases}
		d\exp(-t^2/(4\sigma^2)) & \text{if } t\le \sigma^2/R,\\
		d\exp(-t/(4R))  & \text{if } t\ge\sigma^2/R.
	\end{cases}
	\]
\end{theorem}

Applying Theorem \ref{thm:matrix_ber} to the $(m+n)\times(m+n)$ Hermitian dilation of $\X - \E[\X]$ with rank-one Bernstein coefficient matrices, and specializing the threshold $t$ to the value $t_D$ that makes the union-bound prefactor $(m+n)\exp(-\tau)$ decay polynomially in $m\vee n$, yields the following dimension-free spectral-norm bound that we use throughout the fluctuation analysis.

\begin{lemma}[Spectral concentration of centered random matrices]\label{lem:spectral norm of centered poisson}
	Let $\X = (\X_{ij})$ be an $m\times n$ matrix with independent real entries and finite means $\E[\X_{ij}]$. Assume the moment bounds
	\[
	\E\bigl[|\X_{ij}-\E[\X_{ij}]|^{q}\bigr] \le \frac{q!}{2}\sigma_{ij}^2 R_{ij}^{q-2} \qquad \text{for all } q \ge 2,
	\]
	and denote $\sigma = \max_{ij}\sigma_{ij}$, $R = \max_{ij} R_{ij}$. Then for any $t \ge 0$,
	\begin{align}\label{eq:Bernstein_spectral_noise}
		\P\bigl(\|\X-\E[\X]\|_2 \ge t\bigr) \;\le\; (m+n)\exp\!\left(-\frac{t^2/2}{(m\vee n)\sigma^2 + Rt}\right).
	\end{align}
	In particular, for any $D > 0$, set
	\begin{align}\label{eq:t_D_def}
		\tau \;:=\; (D+1)\log(m\vee n) + \log 4, \qquad t_D \;:=\; \sigma\sqrt{2(m\vee n)\tau} + 2R\tau.
	\end{align}
	Then we have 
	\begin{align}\label{eq:spec_prob_bd}
		\P\bigl( \|\X - \E[\X]\|_2 > t_D \bigr) \;\le\; (m\vee n)^{-D}.
	\end{align}
\end{lemma}

\begin{proof}
		Consider the dilation of $\X-\E[\X]$
		\[
		\Y=\begin{bmatrix}
			\mathbf{0} & \X-\E[\X]\\
			(\X-\E[\X])^T & \mathbf{0}
		\end{bmatrix}.
		\]
		Since $\|\Y\|=\|\X-\E[\X]\|$, we can apply Matrix Bernstein Inequality to $\Y$ to bound the spectral norm of $\X-\E[\X]$. For $1\le i\le m, 1\le j\le n$, denote $[(\X_{kl}-\E[\X_{kl}])\delta_{ij}]_{k,l}$ the $m\times n$ matrix with $\X_{ij}-\E[\X_{ij}]$ at $(i,j)$ position and zero everywhere else. And define adjoint matrix $\bZ^{(ij)}$ of dimension $m+n$ by
		\[
		\bZ^{(ij)}=\begin{bmatrix}
			\mathbf{0} & [(\X_{kl}-\E[\X_{kl}])\delta_{ij}]_{k,l}\\
			([(\X_{kl}-\E[\X_{kl}])\delta_{ij}]_{k,l})^T & \mathbf{0}
		\end{bmatrix}.
		\]
		Then clearly $\Y=\sum_{i,j}\bZ^{(ij)}.$ Notice that $\bZ^{(ij)}$ has only two nonzero entries, $\X_{ij}-\E[\X_{ij}]$, at positions $(i,j+m)$ and $(j+m,i)$. So, it has the following singular value decomposition
		\[
		\bZ^{(ij)}=(\X_{ij}-\E[\X_{ij}])[\mathbf{e}_{j+m}\ \mathbf{e}_{i}]\mathbf{I}_{2\times 2}[\mathbf{e}_{i}\ \mathbf{e}_{j+m}]^T,
		\]
		where $\mathbf{e}_i$ denotes the standard $i$-th basis vector in $\R^{m+n}$. Then one can check that for an even integer $q>0$
		\[
		(\bZ^{(ij)})^q=(\X_{ij}-\E[\X_{ij}])^q[\mathbf{e}_{i}\ \mathbf{e}_{j+m}]\mathbf{I}_{2\times 2}[\mathbf{e}_{i}\ \mathbf{e}_{j+m}]^T,
		\]
		namely $(\bZ^{(ij)})^q$ has two nonzero entries of size $(\X_{ij}-\E[\X_{ij}])^q$ at $(i,i)$ and $(j+m,j+m)$ positions. And when $q$ is an odd positive integer, 
		\[
		(\bZ^{(ij)})^q=(\X_{ij}-\E[\X_{ij}])^q[\mathbf{e}_{j+m}\ \mathbf{e}_{i}]\mathbf{I}_{2\times 2}[\mathbf{e}_{i}\ \mathbf{e}_{j+m}]^T,
		\]
		namely $(\bZ^{(ij)})^q$ has two nonzero entries of size $(\X_{ij}-\E[\X_{ij}])^q$ at $(i,j+m)$ and $(j+m,i)$ positions. Notice that in either case $\E[|\X_{ij}-\E[\X_{ij}]|^q]\,\D(\mathbf{e}_{i}+\mathbf{e}_{j+m}) - \E[(\bZ^{(ij)})^q]$ is diagonal dominant, thus PSD by Gershgorin circle theorem. From the assumption on the moment bounds, we have
		\[
		\E[(\bZ^{(ij)})^q]\preccurlyeq \frac{q!}{2}R_{ij}^{q-2}\bA_{ij}^2,
		\]
		where $\A_{ij}^2=\sigma_{ij}^2\D(\mathbf{e}_{i}+\mathbf{e}_{j+m}).$
		
		\[
		\sum_{i,j}\A_{ij}^2\preccurlyeq \sigma^2\begin{bmatrix}
			n\mathbf{I}_{m\times m} & \mathbf{0}\\
			\mathbf{0} & m\mathbf{I}_{n\times n}
		\end{bmatrix},
		\]
		where $\sigma=\max_{ij}\sigma_{ij}.$
		Then applying Thm. \ref{thm:matrix_ber}, we get
		\[
		\P(\|\X-\E[\X]\|_2\ge t)<(m+n)\exp\left(-\frac{t^2/2}{(m\vee n)\sigma^2+Rt}\right),
		\]
		which is \eqref{eq:Bernstein_spectral_noise}. For the specialization, set $V := (m\vee n)\sigma^2$ and $t_D = \sqrt{2V\tau} + 2R\tau$ as in \eqref{eq:t_D_def}. A direct expansion gives
		\[
		t_D^{2} \;=\; 2V\tau + 4R\tau\sqrt{2V\tau} + 4R^{2}\tau^{2}, \qquad V + Rt_D \;=\; V + R\sqrt{2V\tau} + 2R^{2}\tau,
		\]
		so $t_D^{2} - 2\tau(V + Rt_D) = 2R\tau\sqrt{2V\tau} \ge 0$. Hence $t_D^{2}/\bigl(2(V+Rt_D)\bigr)\ge \tau$, and \eqref{eq:Bernstein_spectral_noise} yields
		\[
		\P\bigl(\|\X-\E[\X]\|_2 > t_D\bigr) \;\le\; (m+n)e^{-\tau} \;\le\; 2(m\vee n)\cdot\tfrac{1}{4}(m\vee n)^{-(D+1)} \;=\; \tfrac{1}{2}(m\vee n)^{-D} \;\le\; (m\vee n)^{-D},
		\]
		using $m+n \le 2(m\vee n)$ and $e^{-\log 4} = 1/4$.
\end{proof}

\begin{proof}[\textbf{Proof of Theorem \ref{thm:first_approximation_rescaled}}]
	The two bounds defining $\mathcal{E}_2$ (spectral) and $\mathcal{E}_3$ ($L^1$) are established via a common concentration step for the potentials followed by a spectral path and an $L^1$ path.
	
	\textbf{Step 1: Potential concentration.}
	By Theorem \ref{thm:concentration_margins_rs}, on the event $\mathcal{E}_1$ the random potentials are well-defined and there exists a scalar shift $t \in \R$ such that
	\begin{align}\label{eq:pf_thm1_pot_bound}
		\|\balpha_{\X} + t\mathbf{1} - \balpha\|_\infty \;\le\; 4 C_* \eps_0, \qquad \|\bbeta_{\X} - t\mathbf{1} - \bbeta\|_\infty \;\le\; 4 C_* \eps_0.
	\end{align}
	Since $\eps_0 \le 1/(50 C_*^2)$ and $C_* \ge 1$ give $4 C_* \eps_0 \le 4/50 < 1/2$, the inequality $|e^x - 1| \le 2|x|$ for $|x|\le 1/2$ yields the per-potential relative bound
	\begin{align}\label{eq:pf_thm1_relative_bound}
		\max_i\left|\frac{e^{\balpha_{\X}(i)+t}}{e^{\balpha(i)}} - 1\right|,\;\;\; \max_j\left|\frac{e^{\bbeta_{\X}(j)-t}}{e^{\bbeta(j)}} - 1\right| \;\le\; 8 C_* \eps_0 \;=\; \tfrac{1}{2}\eps_{\textup{pot}},
	\end{align}
	and, adding the two exponents, the joint relative bound
	\begin{align}\label{eq:pf_thm1_joint_bound}
		\max_{i,j}\left|\frac{e^{\balpha_{\X}(i)+\bbeta_{\X}(j)}}{e^{\balpha(i)+\bbeta(j)}} - 1\right| \;\le\; 16 C_* \eps_0 \;=\; \eps_{\textup{pot}}.
	\end{align}
	By Lemma \ref{lem: boundedness_matrix_scaling_potentials}, $e^{\balpha(i)+\bbeta(j)} \le e^{2K} N/\|\bLambda\|_1$; setting $M_{\textup{dev}} := \max_{i,j}\left|e^{\balpha_{\X}(i)+\bbeta_{\X}(j)} - e^{\balpha(i)+\bbeta(j)}\right|$, \eqref{eq:pf_thm1_joint_bound} gives
	\begin{align}\label{eq:pf_thm1_M_bound}
		\frac{M_{\textup{dev}}}{N} \;\le\; \frac{\eps_{\textup{pot}}\, e^{2K}}{\|\bLambda\|_1}.
	\end{align}

	\textbf{Step 2: Mass concentration.}
	Applying Lemma \ref{lem:bernstein_RV} with weights $a_{ij} = 1$, centered variables $Y_{ij} = \X_{ij} - \bLambda_{ij}$, and threshold $t = \|\bLambda\|_1$ (so $\sum_{i,j} a_{ij}^2 \sigma^2 \le mn\sigma^2$ and $a_{\max} R = R$) yields
	\[
	\P\!\left(\Bigl|\sum_{i,j}(\X_{ij} - \bLambda_{ij})\Bigr| \ge \|\bLambda\|_1\right) \le 2\exp\!\left(-\frac{\|\bLambda\|_1^2/2}{mn\sigma^2 + R\|\bLambda\|_1}\right).
	\]
	Let $\mathcal{E}_{\textup{mass}} := \{\sum_{i,j}\X_{ij} \le 2\|\bLambda\|_1\}$.
	
	\textbf{Step 3: Spectral path (for $\mathcal{E}_2$).}
	Write $\mathbf{D}_{\alpha} := \D(e^{\balpha})$, $\mathbf{D}_{\beta} := \D(e^{\bbeta})$, and analogously $\mathbf{D}_{\alpha_{\X}}, \mathbf{D}_{\beta_{\X}}$. The algebraic identity
	\begin{align*}
		\mathbf{D}_{\alpha_{\X}} (\X-\bLambda) \mathbf{D}_{\beta_{\X}} - \mathbf{D}_{\alpha} (\X-\bLambda) \mathbf{D}_{\beta} = \;&(\mathbf{D}_{\alpha_{\X}} - \mathbf{D}_{\alpha})(\X-\bLambda)\mathbf{D}_{\beta} + \mathbf{D}_{\alpha}(\X-\bLambda)(\mathbf{D}_{\beta_{\X}} - \mathbf{D}_{\beta}) \\
		&\phantom{}+ (\mathbf{D}_{\alpha_{\X}} - \mathbf{D}_{\alpha})(\X-\bLambda)(\mathbf{D}_{\beta_{\X}} - \mathbf{D}_{\beta}),
	\end{align*}
	combined with sub-multiplicativity of the spectral norm and the fact that a diagonal matrix's spectral norm equals the $\ell^{\infty}$ norm of its diagonal, gives
	\begin{align}\label{eq:pf_rescaled_approx_diag}
		\|\mathbf{D}_{\alpha_{\X}} (\X-\bLambda) \mathbf{D}_{\beta_{\X}} - \mathbf{D}_{\alpha} (\X-\bLambda) \mathbf{D}_{\beta} \|_2 \le \|\X-\bLambda\|_2 \Big( &\|e^{\balpha_{\X}} - e^{\balpha}\|_{\infty} \|e^{\bbeta}\|_{\infty} + \|e^{\balpha}\|_{\infty} \|e^{\bbeta_{\X}} - e^{\bbeta}\|_{\infty} \nonumber\\
		& + \|e^{\balpha_{\X}} - e^{\balpha}\|_{\infty} \|e^{\bbeta_{\X}} - e^{\bbeta}\|_{\infty} \Big).
	\end{align}
	The definitions of $(\balpha_{\X}, \bbeta_{\X})$ in Theorem \ref{thm:concentration_margins_rs} and of $(\balpha, \bbeta)$ are phrased modulo the gauge $(\balpha, \bbeta)\mapsto(\balpha - s\mathbf{1}, \bbeta + s\mathbf{1})$, under which the left-hand side of \eqref{eq:pf_rescaled_approx_diag} is invariant. We absorb the scalar shift $t$ from \eqref{eq:pf_thm1_pot_bound} into $(\balpha_{\X}, \bbeta_{\X})$, under which \eqref{eq:pf_thm1_relative_bound} yields
	\[
	\|e^{\balpha_{\X}} - e^{\balpha}\|_{\infty} \;\le\; \tfrac{1}{2}\eps_{\textup{pot}}\,\|e^{\balpha}\|_{\infty}, \qquad \|e^{\bbeta_{\X}} - e^{\bbeta}\|_{\infty} \;\le\; \tfrac{1}{2}\eps_{\textup{pot}}\,\|e^{\bbeta}\|_{\infty}.
	\]
	A further gauge choice equalizes $\|e^{\balpha}\|_{\infty} = \|e^{\bbeta}\|_{\infty}$; under this normalization, Lemma \ref{lem: boundedness_matrix_scaling_potentials} gives
	\[
	\|e^{\balpha}\|_{\infty}\,\|e^{\bbeta}\|_{\infty} = \max_{i,j} e^{\balpha(i)+\bbeta(j)} \le e^{2K}\frac{N}{\|\bLambda\|_1}.
	\]
	Factoring $\|e^{\balpha}\|_{\infty}\|e^{\bbeta}\|_{\infty}$ out of the parenthesized factor in \eqref{eq:pf_rescaled_approx_diag} and using $\eps_{\textup{pot}}\le 1$ to bound the residual by $\eps_{\textup{pot}} + \eps_{\textup{pot}}^2/4 \le \tfrac{5}{4}\eps_{\textup{pot}} \le 3\eps_{\textup{pot}}$, we have
	\begin{align}\label{eq:pf_thm1_spectral_conclusion}
		\frac{1}{N}\,\|\mathbf{D}_{\alpha_{\X}} (\X-\bLambda) \mathbf{D}_{\beta_{\X}} - \mathbf{D}_{\alpha} (\X-\bLambda) \mathbf{D}_{\beta} \|_2 \;\le\; \frac{3\, e^{2K}\,\eps_{\textup{pot}}\, t_D}{\|\bLambda\|_1}
	\end{align}
	on the event $\mathcal{E}_{1}\cap\{\|\X-\bLambda\|_{2}\le t_{D}\}$, which establishes the bound defining $\mathcal{E}_2$. Hence
	\begin{align}\label{eq:pf_thm1_E2_inclusion}
		\mathcal{E}_{1}\cap\{\|\X-\bLambda\|_{2}\le t_{D}\} \;\subseteq\; \mathcal{E}_{2}.
	\end{align}
	
	\textbf{Step 4: $L^1$ path (for $\mathcal{E}_3$).}
	The entry-wise factorization
	\[
	[\X^{\r,\c} - \hat{\X}^{\r,\c}]_{ij} = \bigl(e^{\balpha_{\X}(i)+\bbeta_{\X}(j)} - e^{\balpha(i)+\bbeta(j)}\bigr)\,\X_{ij}
	\]
	and the nonnegativity of $\X$ give $\|\X^{\r,\c} - \hat{\X}^{\r,\c}\|_1 \le M_{\textup{dev}} \sum_{i,j}\X_{ij}$. On $\mathcal{E}_1 \cap \mathcal{E}_{\textup{mass}}$, using \eqref{eq:pf_thm1_M_bound},
	\begin{align}\label{eq:pf_thm1_L1_conclusion}
		\frac{1}{N}\|\X^{\r,\c} - \hat{\X}^{\r,\c}\|_1 \;\le\; \frac{M_{\textup{dev}}}{N}\sum_{i,j}\X_{ij} \;\le\; \frac{\eps_{\textup{pot}} e^{2K}}{\|\bLambda\|_1}\cdot 2\|\bLambda\|_1 \;=\; 2\eps_{\textup{pot}} e^{2K},
	\end{align}
	establishing the bound in $\mathcal{E}_3$.
	
	\textbf{Step 5: Probability bounds.}
	Because $\mathcal{E}_1$ already forces $\X \in \mathcal{S}(\r,\c)$ (proof of Theorem \ref{thm:concentration_margins_rs}), the scalability component in \eqref{eq:def_event_E2} and \eqref{eq:def_event_E3} is automatic on $\mathcal{E}_1\cap\{\|\X-\bLambda\|_2\le t_D\}$ and $\mathcal{E}_1\cap\mathcal{E}_{\textup{mass}}$, respectively. Combining \eqref{eq:pf_thm1_E2_inclusion} with Lemma \ref{lem:spectral norm of centered poisson} and the union bound gives
	\begin{align*}
		\P(\mathcal{E}_2^c) &\le \P(\mathcal{E}_1^c) + \P\bigl(\|\X-\bLambda\|_2 > t_D\bigr) \le 2(m \vee n)^{-D} + m^2 \exp(-c\,\Phi_{\bLambda}), \\
		\P(\mathcal{E}_3^c) &\le \P(\mathcal{E}_1^c) + \P(\mathcal{E}_{\textup{mass}}^c) \le (m \vee n)^{-D} + m^2 \exp(-c\,\Phi_{\bLambda}) + 2\exp\!\left(-\tfrac{\|\bLambda\|_1^2/2}{mn\sigma^2 + R\|\bLambda\|_1}\right),
	\end{align*}
	matching the probability bounds stated in the theorem.
\end{proof}

\begin{proof}[\textbf{Proof of Theorem \ref{thm:bound_X^rs-Z}}]
	We split the integration error via the triangle inequality:
	\begin{align}\label{eq:pf_thm2_split}
		\left|\int_{[0,1]^2}g(\varphi_{\X^{\r,\c}} - \varphi_{\bLambda^{\r,\c}})\, dx\,dy\right| \le \underbrace{\left|\int_{[0,1]^2}g(\varphi_{\X^{\r,\c}} - \varphi_{\hat{\X}^{\r,\c}})\, dx\,dy\right|}_{\text{Term 1}} + \underbrace{\left|\int_{[0,1]^2}g(\varphi_{\hat{\X}^{\r,\c}} - \varphi_{\bLambda^{\r,\c}})\, dx\,dy\right|}_{\text{Term 2}}.
	\end{align}
	
	Since $\varphi_{\X^{\r,\c}}$ and $\varphi_{\hat{\X}^{\r,\c}}$ are piecewise constant on the $mn$ grid cells, denoting $G_{ij} = \int_{A_{ij}} g(x, y) dxdy$, 
	\[
	\text{Term 1} = %\left| \frac{1}{mn} \sum_{i=1}^m \sum_{j=1}^n g\!\left(\tfrac{i}{m}, \tfrac{j}{n}\right) \frac{mn}{N} \bigl(\X^{\r,\c}_{ij} - \hat{\X}^{\r,\c}_{ij}\bigr) \right|
	\left| \sum_{i = 1}^{m} \sum_{j = 1}^{n} G_{ij} \frac{mn}{N} \bigl(\X^{\r,\c}_{ij} - \hat{\X}^{\r,\c}_{ij}\bigr) \right| 
	\le \frac{\|g\|_{\infty}}{N} \|\X^{\r,\c} - \hat{\X}^{\r,\c}\|_1.
	\]
	By Theorem \ref{thm:first_approximation_rescaled}, on the event $\mathcal{E}_3$ (evaluated at the current $\tau$), $\frac{1}{N}\|\X^{\r,\c} - \hat{\X}^{\r,\c}\|_1 \le 2\eps_{\textup{pot}} e^{2K}$; hence
	\[
	\text{Term 1} \le 2\|g\|_{\infty}\, \eps_{\textup{pot}}\, e^{2K}.
	\]
	
	Next, since $\bLambda^{\r,\c}_{ij} = e^{\balpha(i)+\bbeta(j)} \bLambda_{ij}$ and $\hat{\X}^{\r,\c}_{ij} = e^{\balpha(i)+\bbeta(j)} \X_{ij}$, Term 2 is a weighted sum of centered entries:
	\[
	\text{Term 2} = \left| \sum_{i,j} W_{ij} (\X_{ij} - \bLambda_{ij}) \right|, \qquad W_{ij} := \frac{mn}{N}\, e^{\balpha(i)+\bbeta(j)} \int_{A_{ij}} g(x,y)\, dx\, dy,
	\]
	where $A_{ij}$ is the grid cell of area $1/(mn)$. Using $\bigl|\int_{A_{ij}} g\bigr| \le \|g\|_\infty / (mn)$ and Lemma \ref{lem: boundedness_matrix_scaling_potentials} (giving $e^{\balpha(i)+\bbeta(j)} \le e^{2K} N/\|\bLambda\|_1$), we uniformly bound the weights by
	\[
	|W_{ij}| \le \frac{\|g\|_{\infty} e^{2K}}{\|\bLambda\|_1}.
	\]
	Applying Lemma \ref{lem:bernstein_RV} with variance proxy $V = \sum_{i,j} W_{ij}^2 \sigma^2 \le mn\sigma^2 (\|g\|_\infty e^{2K}/\|\bLambda\|_1)^2$ and scale $B = \max_{i,j}|W_{ij}|\, R \le \|g\|_\infty e^{2K} R/\|\bLambda\|_1$, for any $t > 0$,
	\[
	\P(\text{Term 2} \ge t) \le 2\exp\!\left(-\tfrac{t^2/2}{V + Bt}\right).
	\]
	The threshold $t = \sqrt{2V\tau} + 2B\tau$ makes this at most $2\exp(-\tau)$, and substituting the bounds on $V$ and $B$ yields
	\[
	t \le \frac{\|g\|_\infty e^{2K}}{\|\bLambda\|_1}\bigl(\sigma\sqrt{2mn\tau} + 2R\tau\bigr).
	\]
	Let $\mathcal{E}_{\textup{conc}}$ be the event that Term 2 is bounded by this threshold; then $\P(\mathcal{E}_{\textup{conc}}^c) \le 2\exp(-\tau) = \tfrac{1}{2}(m\vee n)^{-(D+1)}$.
	
	On $\mathcal{E}_3 \cap \mathcal{E}_{\textup{conc}}$, summing the bounds for Term 1 and Term 2 yields the stated error bound. A union bound on the failure probabilities gives
	\begin{align*}
		\P(\mathcal{E}_3^c) + \P(\mathcal{E}_{\textup{conc}}^c)
		&\le (m\vee n)^{-D} + m^2\exp(-c\,\Phi_{\bLambda}) + 2\exp\!\left(-\tfrac{\|\bLambda\|_1^2/2}{mn\sigma^2+R\|\bLambda\|_1}\right) + \tfrac{1}{2}(m\vee n)^{-(D+1)}\\
		&\le 2(m\vee n)^{-D} + m^2\exp(-c\,\Phi_{\bLambda}) + 2\exp\!\left(-\tfrac{\|\bLambda\|_1^2/2}{mn\sigma^2+R\|\bLambda\|_1}\right),
	\end{align*}
	where the $(m\vee n)^{-D}$ bound on $\P(\mathcal{E}_3^c)$'s polynomial part is obtained by applying Theorem \ref{thm:first_approximation_rescaled} (and thus Theorem \ref{thm:concentration_margins_rs}) with the current $\tau = (D+1)\log(m\vee n) + \log 4$.
\end{proof}

\subsection{Proofs of scaling limit results}

Here we prove Theorem \ref{thm:RM_rescaled_limit}. This proof integrates the stochastic approximation bounds with the deterministic limit. The key to the proof is rigorously evaluating the asymptotic order of the algebraic terms generated by Theorem \ref{thm:bound_X^rs-Z} to show that the $O(1)$ stability constants hold and the stochastic noise term decays exactly as stated.

\begin{proof}[\textbf{Proof of Theorem \ref{thm:RM_rescaled_limit}}]
	Within this proof we write $N_k := m_k \vee n_k$ for the larger dimension; we trust context to distinguish this from the total margin mass $\sum_i \r_k(i)=\sum_j \c_k(j)$ also denoted $N_k$ in Assumption \ref{assumption: limit_theory}. By Assumption \ref{assumption: limit_theory}, the discrete kernel $\varphi_{\bLambda_k}$ converges to a strictly positive, bounded density $\varphi_{\bLambda}$, which forces the matrix entries to scale uniformly: $\bLambda_{k,\min} = \Theta(\|\bLambda_k\|_1/N_k^2)$ and $\bLambda_{k,\max} = \Theta(\|\bLambda_k\|_1/N_k^2)$. Consequently the explicit stability constant satisfies $C_{*,k} = 1 + O(1) = O(1)$, and the global scale parameter $\sigma + R + \bLambda_{k,\max}$ stays bounded away from $0$ and $\infty$.
	
	We separate the integration error via the triangle inequality:
	\begin{align}\label{eq:pf_thm4_split}
		\left|\int_{[0,1]^2}g\bigl(\varphi_{\X_k^{\r_k,\c_k}} - W^{\r,\c;\bLambda}\bigr)\right| \le \underbrace{\left|\int_{[0,1]^2}g\bigl(\varphi_{\X_k^{\r_k,\c_k}} - \varphi_{\bLambda_k^{\r_k,\c_k}}\bigr)\right|}_{\text{Term 1}} + \underbrace{\left|\int_{[0,1]^2}g\bigl(\varphi_{\bLambda_k^{\r_k,\c_k}} - W^{\r,\c;\bLambda}\bigr)\right|}_{\text{Term 2}}.
	\end{align}
	
	First, we apply Theorem \ref{thm:bound_X^rs-Z} with $\tau = (D+1)\log N_k + \log 4 = \Theta(D\log N_k)$. Since $C_{*,k} = O(1)$,
	\[
	\eps_{\textup{pot}} \asymp \eps_0 \asymp \frac{N_k^{3/2}\sqrt{D\log N_k}}{\|\bLambda_k\|_1}.
	\]
	The mass-growth hypothesis $N_k^{3/2}\sqrt{\log N_k} = o(\|\bLambda_k\|_1)$ gives $\eps_{\textup{pot}} \to 0$, in particular $\eps_{\textup{pot}} \le 1$ for all sufficiently large $k$. The two terms in the bound of Theorem \ref{thm:bound_X^rs-Z} scale as
	\[
	2\|g\|_\infty \eps_{\textup{pot}} e^{2K} \asymp \frac{N_k^{3/2}\sqrt{D\log N_k}}{\|\bLambda_k\|_1}, \qquad \frac{\|g\|_\infty e^{2K}}{\|\bLambda_k\|_1}\bigl(\sigma\sqrt{2 m_k n_k \tau} + 2R\tau\bigr) \asymp \frac{N_k \sqrt{D\log N_k}}{\|\bLambda_k\|_1},
	\]
	where we used $m_k n_k \asymp N_k^2$ under the balanced-dimension hypothesis. The first term dominates, so
	\[
	\text{Term 1} = O_{K,\delta,\gamma,\|g\|_\infty,\sigma,R}\!\left(\frac{N_k^{3/2}\sqrt{D\log N_k}}{\|\bLambda_k\|_1}\right).
	\]
	The failure probability is $O(N_k^{-D})$ plus the sub-Weibull row-alignment term (which decays faster than any polynomial under the fixed kernel hypothesis) plus $2\exp(-\|\bLambda_k\|_1^2/(2(m_k n_k \sigma^2 + R\|\bLambda_k\|_1)))$ (also decaying faster than any polynomial since $\|\bLambda_k\|_1 \gg N_k^{3/2}$).
	
	Next, we bound Term 2. Bounding the integral by the $L^1$ distance, then by the Hellinger distance, we get 
	\[
	\text{Term 2} \le \|g\|_{\infty}\, \bigl\|\varphi_{\bLambda_k^{\r_k,\c_k}} - W^{\r,\c;\bLambda}\bigr\|_{L^1} \le 2\|g\|_{\infty}\, d_H\!\bigl(\varphi_{\bLambda_k^{\r_k,\c_k}},\, W^{\r,\c;\bLambda}\bigr).
	\]
	Applying Theorem \ref{thm:SSB_deterministic_limit} yields
	\[
	\text{Term 2} = O_{K,\delta,\|g\|_\infty}\!\left(\sqrt{\|\rho_{\r_k} - \rho_{\r}\|_{L^1} + \|\rho_{\c_k} - \rho_{\c}\|_{L^1}} + d_H(\varphi_{\bLambda_k}, \varphi_{\bLambda})\right).
	\]
	Summing the bounds for Term 1 and Term 2 yields the stated estimate.
\end{proof}

\subsection{Proofs of fluctuation results}

\begin{lemma}[Spectral concentration of the rescaled fluctuation matrix]\label{lem:rescaled_approx}
	Suppose Assumption \ref{assump:random_matrix} holds, and let $s_{\max}$ be the maximum variance defined in \eqref{eq:def_s_star}, $\A$ and $\check{\A}$ be the rescaled and comparison fluctuation matrices defined in \eqref{eq:def_check_A_and_A}. Let $t_D$, $\tau$, and $\eps_{\textup{pot}}$ be as in Theorem \ref{thm:first_approximation_rescaled}, and assume $\eps_{\textup{pot}} \le 1$. Then, on the event $\mathcal{E}_{2} \cap \{\|\X-\bLambda\|_2 \le t_D\}$,
	\begin{align}\label{eq:A_Acheck_spectral_norm_bd}
		%&{\color{blue}\|\A\A^T - \check{\A}\check{\A}^T\|_2 \le \eps_{\textup{cov}} := \frac{6\, e^{4K}\, N^2}{(m+n)\,s_{\textup{max}}\,\|\bLambda\|_1^2}\, \eps_{\textup{pot}}\, t_D^2\, (2 + 3\eps_{\textup{pot}}).}\\
		&\|\A\A^T - \check{\A}\check{\A}^T\|_2 \le \eps_{\textup{cov}} := \frac{3\, e^{4K}\, N^2}{(m+n)\,s_{\textup{max}}\,\|\bLambda\|_1^2}\, \eps_{\textup{pot}}\, t_D^2\, (2 + 3\eps_{\textup{pot}}).
	\end{align}
\end{lemma}

\begin{proof}	
	To bound the difference of the covariance matrices, we use the algebraic identity $\A\A^T - \check{\A}\check{\A}^T = (\A - \check{\A})\A^T + \check{\A}(\A^T - \check{\A}^T)$. Applying the triangle inequality and sub-multiplicativity gives:
	\[
	\|\A\A^T - \check{\A}\check{\A}^T\|_2 \le \|\A - \check{\A}\|_2 (\|\A\|_2 + \|\check{\A}\|_2).
	\]
	{
		By the definitions \eqref{eq:def_check_A_and_A} and %^Theorem \ref{thm:first_approximation_rescaled},
		\eqref{eq:def_event_E2}, on the event $\mathcal{E}_{2}$, 
		\begin{align}
			\|\A - \check{\A}\|_2 &= \frac{1}{\sqrt{(m+n)s_{\textup{max}}}}\bigl\|\D(e^{\balpha_{\X}})(\X-\bLambda)\D(e^{\bbeta_{\X}})-\D(e^{\balpha})(\X-\bLambda)\D(e^{\bbeta})\bigr\|_2 \\
			&\le \frac{3e^{2K}\, N}{\sqrt{(m+n)s_{\textup{max}}}\;\|\bLambda\|_1}\,\eps_{\textup{pot}}\, t_D.
		\end{align}
	}
	To bound the spectral norm of the comparison fluctuation matrix $\check{\A}$, we use the uniform bound $\|e^{\balpha}\|_{\infty}\|e^{\bbeta}\|_{\infty}\le e^{2K}N/\|\bLambda\|_1$ from Lemma \ref{lem: boundedness_matrix_scaling_potentials} (combined with a gauge choice $\|e^{\balpha}\|_\infty = \|e^{\bbeta}\|_\infty$ as in the proof of Theorem \ref{thm:first_approximation_rescaled}) and the spectral noise bound $\|\X-\bLambda\|_2 \le t_D$ from the lemma hypothesis:
		\[
		\|\check{\A}\|_2 \le \frac{1}{\sqrt{(m+n)s_{\textup{max}}}}\, \|e^{\balpha}\|_{\infty} \|e^{\bbeta}\|_{\infty}\, \|\X - \bLambda\|_2 \le \frac{e^{2K}\, N}{\sqrt{(m+n)s_{\textup{max}}}\;\|\bLambda\|_1}\, t_D.
		\]
	By the triangle inequality, the norm of the random fluctuation matrix is bounded by:
		\[
		\|\A\|_2 \le \|\check{\A}\|_2 + \|\A - \check{\A}\|_2 \le \frac{e^{2K}\, N}{\sqrt{(m+n)s_{\textup{max}}}\;\|\bLambda\|_1}\, t_D\,(1 + 3\eps_{\textup{pot}}).
		\]
	Substituting these bounds into the product yields:
	I think $3$ is a tigher constant
	\begin{align}
		\|\A\A^T - \check{\A}\check{\A}^T\|_2 &\le \left( \frac{3 e^{2K}\, N}{\sqrt{(m+n)s_{\textup{max}}}\;\|\bLambda\|_1}\, \eps_{\textup{pot}}\, t_D \right)\!\left( \frac{e^{2K}\, N}{\sqrt{(m+n)s_{\textup{max}}}\;\|\bLambda\|_1}\, t_D\,(2 + 3\eps_{\textup{pot}}) \right) \\
		&= \frac{3\, e^{4K}\, N^2}{(m+n)\,s_{\textup{max}}\,\|\bLambda\|_1^2}\, \eps_{\textup{pot}}\, t_D^2\, (2 + 3\eps_{\textup{pot}}),
	\end{align}
	which is the exact upper bound $\eps_{\textup{cov}}$, completing the proof.
\end{proof}

\begin{proof}[\textbf{Proof of Theorem \ref{thm:ESD}}]
	Our proof relies on the spectral concentration bounds from Lemma \ref{lem:rescaled_approx}, combined with the local law for random Gram matrices established by Alt, Erdős, and Krüger \cite{alt2017local}.
	
	Recall that the entries of the comparison fluctuation matrix $\check{\A}$ are given by \eqref{eq:def_check_A_and_A}. Since $\X$ has independent entries, the entries of $\check{\A}$ are independent and centered, with variance profile $\mathbf{S}$.
	
	To apply the local law and rigidity results from \cite{alt2017local} to $\check{\A}\check{\A}^T$, we verify that $\mathbf{S}$ satisfies the foundational assumptions (A), (F2), (C), (D) of that paper. Let $\sigma_{\max}^2 = \max_{i,j} \operatorname{Var}(\X_{ij})$ and $\sigma_{\min}^2 = \min_{i,j} \operatorname{Var}(\X_{ij}) > 0$.
	\begin{itemize}[leftmargin=0.3cm, itemsep=0.15cm]
		\item \textbf{(D) Comparable Dimensions:} By hypothesis of the theorem, $\gamma \le m/n \le \gamma^{-1}$, satisfying the dimensional-comparability requirement.
		\item \textbf{(A) Flatness (Upper Bound):}  By the definition of $s_{\textup{max}}$ in \eqref{eq:def_s_star}, $e^{2\balpha(i)+2\bbeta(j)}\operatorname{Var}(\X_{ij})\le s_{\textup{max}}$ for all $(i,j)$, so
		\[
		\mathbf{S}_{ij} \;=\; \frac{1}{(m+n)\,s_{\textup{max}}}\, e^{2\balpha(i)+2\bbeta(j)}\, \operatorname{Var}(\X_{ij}) \;\le\; \frac{1}{m+n},
		\]
		verifying the flatness assumption (A) of \cite{alt2017local}. Then \cite[Thm.~2.1]{alt2017local}  yields $\operatorname{supp}(\nu)\subseteq [0,4]$. 
	\item \textbf{(F2) Uniform Lower Bound:}  Let $s_{\min} := \min_{i,j} e^{2\balpha(i)+2\bbeta(j)}\,\operatorname{Var}(\X_{ij})$, which is strictly positive by Assumption \ref{assump:random_matrix}. Then directly from \eqref{eq:variance_profile_def} and the definition of $s_{\textup{max}}$ in \eqref{eq:def_s_star},
	\[
	\mathbf{S}_{ij} \;\ge\; \frac{\, s_{\min}}{(m+n)\,s_{\textup{max}}} \;=\; \frac{\varphi}{m+n},\qquad \varphi := \frac{ s_{\min}}{s_{\textup{max}}}.
	\]
	Using Lemma \ref{lem: boundedness_matrix_scaling_potentials} to bound the potentials ($e^{-2K}N/\|\bLambda\|_1 \le e^{\balpha(i)+\bbeta(j)}\le e^{2K}N/\|\bLambda\|_1$), this constant admits the explicit lower bound $\varphi \ge  e^{-8K} \xi>0$, depending only on $K$ and the variance ratio $\xi$. Since $\mathbf{S}_{ij}>0$ for every $(i,j)$, the underlying support matrix is the all-ones matrix, which is fully indecomposable and strictly connected; hence (F2) implies Assumption (B) of \cite{alt2017local} by their Remark 2.8.
	\item \textbf{(C) Bounded Moments:} Because $\X_{ij}$ has sub-exponential tails by Assumption \ref{assump:random_matrix}, the normalized entries $\check{\A}_{ij}$ possess uniformly bounded moments of all orders, satisfying the polynomial moment bounds required by the local law.
\end{itemize}

Because the regularity conditions are satisfied, the eigenvalues of the comparison model $\check{\A}\check{\A}^T$ tightly concentrate around their classical locations. Let $\mathcal{E}_{\textup{rig}}$ be the event on which both the optimal rigidity bounds and the spectral confinement hold for $\check{\A}\check{\A}^T$. Specifically, we define $\mathcal{E}_{\textup{rig}}$ as the event where the following two conditions are met simultaneously:
\begin{enumerate}
	\item $\sigma(\check{\A}\check{\A}^T) \subset \{\tau : \text{dist}(\tau, \text{supp}(\nu)) < \eps_*\}$.
	\item For all $\tau \in (0, 4]$ satisfying $\pi(\tau) \ge \eps_*$, the ordered eigenvalues satisfy:
	\[
	|\lambda_{i(\tau)}(\check{\A}\check{\A}^T) - \tau| \le \begin{cases} 
		\frac{m^{\eps}}{m} & \text{if } m \ne n, \\
		\frac{m^{\eps}}{m}\left(\sqrt{\tau} + \frac{1}{m}\right) & \text{if } m = n.
	\end{cases}
	\]
\end{enumerate}
By Theorem 2.9(ii) and (iii) for properly rectangular matrices, and Theorem 2.7(iii) and (iv) for square matrices from \cite{alt2017local}, the probability of the complement event is bounded by $\P(\mathcal{E}_{\textup{rig}}^c) \le O(m^{-D})$. (Those theorems state the rigidity on the range $\tau\in(0,10s_*^{\textup{AEK}}] = (0,20]$, but the extra condition $\pi(\tau)\ge\eps_*$ restricts $\tau$ to the interior of $\operatorname{supp}(\nu)\subseteq[0,4s_*^{\textup{AEK}}]=[0,8]$, so the two ranges are interchangeable.)

To conclude the proof, let $\mathcal{E}_{\textup{bad}}$ denote the eigenvalue deviation event inside the probability bound \eqref{eq:rigidity_rectangular} for case (i). Set $\mathcal{E}_{2}' := \mathcal{E}_{2} \cap \{\|\X-\bLambda\|_2 \le t_D\}$, the event on which Lemma \ref{lem:rescaled_approx} applies. We decompose the probability of the intersection:
\[
\P\left(\{ \X \in \mathcal{S}(\r,\c) \} \cap \mathcal{E}_{\textup{bad}}\right) \le \P\left((\mathcal{E}_2')^{c}\right) + \P\left(\mathcal{E}_2' \cap \mathcal{E}_{\textup{bad}}\right).
\]
The first term is bounded by $\P(\mathcal{E}_2^{c}) + \P(\|\X-\bLambda\|_2 > t_D)$. For the second term, we apply Lemma \ref{lem:rescaled_approx}: on $\mathcal{E}_2'$ the explicit spectral-norm bound $\|\A\A^T - \check{\A}\check{\A}^T\|_2 \le \eps_{\textup{cov}}$ holds, and by Weyl's inequality $|\lambda_k(\A\A^T) - \lambda_k(\check{\A}\check{\A}^T)| \le \eps_{\textup{cov}}$ for every $k$.

If the event $\mathcal{E}_{\textup{bad}}$ also occurs, the triangle inequality demands that there must exist some classical location $\tau$ where the comparison model fails its rigidity bound: $|\lambda_{i(\tau)}(\check{\A}\check{\A}^T) - \tau| > \frac{m^\eps}{m}$. However, this is exactly the complement of the rigidity event $\mathcal{E}_{\textup{rig}}$. Therefore, we have the deterministic set inclusion $\mathcal{E}_2' \cap \mathcal{E}_{\textup{bad}} \subset \mathcal{E}_{\textup{rig}}^c$.

Taking probabilities and invoking Lemma \ref{lem:spectral norm of centered poisson}:
\[
\P\left(\{ \X \in \mathcal{S}(\r,\c) \} \cap \mathcal{E}_{\textup{bad}}\right) \;\le\; \P(\mathcal{E}_2^{c}) + (m\vee n)^{-D} + \P(\mathcal{E}_{\textup{rig}}^c) \;\le\; \P(\mathcal{E}_2^{c}) + \frac{C_3}{m^D},
\]
absorbing the two $O(m^{-D})$ terms into the constant $C_3$, which establishes part \textbf{(i)}. The exact same logical decomposition applied to the rigidity bounds for square matrices yields part \textbf{(ii)}.

To establish the spectral confinement bound in part \textbf{(iii)}, let $\mathcal{E}_{\textup{out}}$ denote the event that $\A\A^T$ has an eigenvalue outside the $(\eps_* + \eps_{\textup{cov}})$-neighborhood of $\text{supp}(\nu)$, namely:
\[
\mathcal{E}_{\textup{out}} = \left\{ \sigma({\A}{\A}^T) \cap \left\{ \tau : \text{dist}(\tau, \text{supp}(\nu)) \ge \eps_* + \eps_{\textup{cov}} \right\} \ne \emptyset \right\}.
\]
We decompose the probability of the intersection as before, with $\mathcal{E}_2' := \mathcal{E}_2 \cap \{\|\X-\bLambda\|_2 \le t_D\}$:
\[
\P\left(\{ \X \in \mathcal{S}(\r,\c) \} \cap \mathcal{E}_{\textup{out}}\right) \le \P\left((\mathcal{E}_2')^{c}\right) + \P\left(\mathcal{E}_2' \cap \mathcal{E}_{\textup{out}}\right).
\]
Assume the event $\mathcal{E}_2' \cap \mathcal{E}_{\textup{out}}$ occurs. By definition of $\mathcal{E}_{\textup{out}}$, there exists some eigenvalue index $k$ such that $\text{dist}(\lambda_k(\A\A^T), \text{supp}(\nu)) \ge \eps_* + \eps_{\textup{cov}}$. Because $\mathcal{E}_2'$ holds, Lemma \ref{lem:rescaled_approx} and Weyl's inequality give $|\lambda_k(\A\A^T) - \lambda_k(\check{\A}\check{\A}^T)| \le \eps_{\textup{cov}}$. Applying the reverse triangle inequality for the distance to a set yields:
\begin{align*}
	\text{dist}(\lambda_k(\check{\A}\check{\A}^T), \text{supp}(\nu)) &\ge \text{dist}(\lambda_k(\A\A^T), \text{supp}(\nu)) - |\lambda_k(\A\A^T) - \lambda_k(\check{\A}\check{\A}^T)| \\
	&\ge (\eps_* + \eps_{\textup{cov}}) - \eps_{\textup{cov}} = \eps_*.
\end{align*}
This implies $\sigma(\check{\A}\check{\A}^T)$ is not fully confined to the open $\eps_*$-neighborhood of $\text{supp}(\nu)$, which means the first condition defining the rigidity event $\mathcal{E}_{\textup{rig}}$ has failed. Thus, we have the deterministic set inclusion $\mathcal{E}_2' \cap \mathcal{E}_{\textup{out}} \subset \mathcal{E}_{\textup{rig}}^c$.

Taking probabilities and invoking Lemma \ref{lem:spectral norm of centered poisson}:
\[
\P\left(\{ \X \in \mathcal{S}(\r,\c) \} \cap \mathcal{E}_{\textup{out}}\right) \le \P(\mathcal{E}_2^{c}) + (m\vee n)^{-D} + \P(\mathcal{E}_{\textup{rig}}^c) \le \P(\mathcal{E}_2^{c}) + \frac{1}{m^D},
\]
after absorbing the two $O(m^{-D})$ terms, which establishes \eqref{eq:no_eigenvalues_outside} and completes the proof.
\end{proof}

Next, we prove the CLT for the Schr\"{o}dinger potentials stated in Theorem \ref{thm: CLT_potentials_fixed_dim}. The proof relies on the multivariate delta method applied to the Sinkhorn scaling map.

\begin{proof}[\textbf{Proof of Theorem \ref{thm: CLT_potentials_fixed_dim}}]
Under Assumption \ref{assump:random_matrix}, the entries of the random matrix $\X$ have finite moments of all orders; in particular, the covariance matrix $\bSigma = \textup{Cov}(\textup{vec}(\X))$ is well-defined and finite. By the multivariate Central Limit Theorem, the fluctuations of the empirical mean $\bar{\X}_M$ around the true mean $\bLambda$ satisfy:
\begin{align}\label{eq:clt_foundation}
	\sqrt{M} \textup{vec}(\bar{\X}_M - \bLambda) \xrightarrow{d} \mathbf{Z}, \quad \text{where } \mathbf{Z} \sim \mathcal{N}(\mathbf{0}, \bSigma).
\end{align}

Let $B: \mathcal{U} \to \mathcal{E}$ be the continuously differentiable map from Lemma \ref{lem: implicit_function_justification} which satisfies $B(\bLambda) = (\balpha, \bbeta)$. By the Strong Law of Large Numbers, $\bar{\X}_M \xrightarrow{a.s.} \bLambda$. Since $\mathcal{U}$ is an open neighborhood of $\bLambda$, the event $\{\bar{\X}_M \in \mathcal{U}\}$ occurs with probability approaching $1$. On this event, we apply the Fundamental Theorem of Calculus to the map $B$ along the line segment connecting $\bLambda$ and $\bar{\X}_M$:
\begin{align}\label{eq:taylor_integral}
	B(\bar{\X}_M) - B(\bLambda) = \left[ \int_{0}^1 J_{B}(\bLambda + t(\bar{\X}_M - \bLambda)) \, dt \right] \textup{vec}(\bar{\X}_M - \bLambda),
\end{align}
where $J_{B}(\X) \in \R^{(m+n) \times mn}$ is the Jacobian matrix of $B$ evaluated at $\X$. We now analyze the difference between this integral and the Jacobian at the limit point $J_{\bLambda} = -\mathbf{L}^{\dagger} \mathbf{H}$:
\begin{align*}
	\left\| \int_{0}^1 J_{B}(\bLambda + t(\bar{\X}_M - \bLambda)) \, dt - J_{\bLambda} \right\|_{\textup{op}} \le \int_{0}^1 \left\| J_{B}(\bLambda + t(\bar{\X}_M - \bLambda)) - J_{\bLambda} \right\|_{\textup{op}} \, dt.
\end{align*}
Because $B$ is $\mathcal{C}^1$, the Jacobian mapping $\X \mapsto J_B(\X)$ is continuous. Since the segment $\{\bLambda + t(\bar{\X}_M - \bLambda) : t \in [0,1]\}$ collapses to the point $\bLambda$ as $M \to \infty$, the integrand converges to zero uniformly in $t$ almost surely. Thus, the entire integral expression in the brackets of \eqref{eq:taylor_integral} converges to $J_{\bLambda}$ in probability.

Multiplying \eqref{eq:taylor_integral} by $\sqrt{M}$, we have:
\begin{align*}
	\sqrt{M}(B(\bar{\X}_M) - B(\bLambda)) = \left( J_{\bLambda} + o_P(1) \right) \sqrt{M} \textup{vec}(\bar{\X}_M - \bLambda).
\end{align*}
From \eqref{eq:clt_foundation}, the term $\sqrt{M} \textup{vec}(\bar{\X}_M - \bLambda)$ converges in distribution, which implies it is stochastically bounded ($O_P(1)$). By Slutsky's Theorem, the $o_P(1)$ term times the $O_P(1)$ term vanishes in probability, yielding the linear approximation:
\begin{align}
	\sqrt{M}\begin{pmatrix} \balpha_{\bar{\X}_M} - \balpha \\ \bbeta_{\bar{\X}_M} - \bbeta \end{pmatrix} = J_{\bLambda} \sqrt{M} \textup{vec}(\bar{\X}_M - \bLambda) + o_P(1).
\end{align}
Finally, by the Continuous Mapping Theorem, the linear transformation of the Gaussian limit $\mathbf{Z}$ results in:
\begin{align}
	J_{\bLambda} \sqrt{M} \textup{vec}(\bar{\X}_M - \bLambda) \xrightarrow{d} J_{\bLambda} \mathbf{Z} \sim \mathcal{N}\left(\mathbf{0}, J_{\bLambda} \bSigma J_{\bLambda}^\top \right).
\end{align}
Substituting the explicit identity $J_{\bLambda} = -\mathbf{L}^{\dagger} \mathbf{H}$ into the covariance quadratic form (where the negative signs cancel) completes the proof.
\end{proof}

\section*{Acknowledgements} 

DD was supported in part by NSF Award DMS-2023239. HL was partially supported by NSF Award DMS-2206296.

\small{
\bibliographystyle{amsalpha}
\bibliography{refs}
}

\end{document}